\documentclass[11pt]{article}

\usepackage{amsfonts,amssymb,amsmath,graphicx,graphics}
\usepackage[english]{babel}
 \def\botcaption#1#2{\medskip\centerline{{\scshape #1.}\kern8pt
 {\rm #2}}\bigskip}

 \makeatletter
 \@addtoreset{equation}{section}
 \makeatother

 \makeatletter
 \@addtoreset{enunciato}{section}
 \makeatother

 \newcounter{enunciato}[section]

 \newtheorem{ittheorem}{Theorem}
 \newtheorem{itlemma}{Lemma}
 \newtheorem{itproposition}{Proposition}
 \newtheorem{itdefinition}{Definition}
 \newtheorem{itassumption}{Assumption}
 \newtheorem{itremark}{Remark}
 \newtheorem{itclaim}{Claim}
 \newtheorem{itcorollary}{Corollary}

 \newenvironment{theorem}{\addtocounter{enunciato}{1}
 \begin{ittheorem}}{\end{ittheorem}}

 \newenvironment{corollary}{\addtocounter{enunciato}{1}
 \begin{itcorollary}}{\end{itcorollary}}

 \newenvironment{lemma}{\addtocounter{enunciato}{1}
 \begin{itlemma}}{\end{itlemma}}

 \newenvironment{proposition}{\addtocounter{enunciato}{1}
 \begin{itproposition}}{\end{itproposition}}

 \newenvironment{definition}{\addtocounter{enunciato}{1}
 \begin{itdefinition}}{\end{itdefinition}}

 \newenvironment{assumption}{\addtocounter{enunciato}{1}
 \begin{itassumption}}{\end{itassumption}}

 \newenvironment{remark}{\addtocounter{enunciato}{1}
 \begin{itremark}}{\end{itremark}}

 \newenvironment{claim}{\addtocounter{enunciato}{1}
 \begin{itclaim}}{\end{itclaim}}

 \newenvironment{proof}{\noindent {\bf Proof.\,}
 }{\hspace*{\fill}$\square$\medskip}

 \newcommand{\be}[1]{\begin{equation}\label{#1}}
 \newcommand{\ee}{\end{equation}}

 \newcommand{\bl}[1]{\begin{lemma}\label{#1}}
 \newcommand{\el}{\end{lemma}}

 \newcommand{\br}[1]{\begin{remark}\label{#1}}
 \newcommand{\er}{\end{remark}}

 \newcommand{\bt}[1]{\begin{theorem}\label{#1}}
 \newcommand{\et}{\end{theorem}}

 \newcommand{\bd}[1]{\begin{definition}\label{#1}}
 \newcommand{\ed}{\end{definition}}

 \newcommand{\bcl}[1]{\begin{claim}\label{#1}}
 \newcommand{\ecl}{\end{claim}}

 \newcommand{\bp}[1]{\begin{proposition}\label{#1}}
 \newcommand{\ep}{\end{proposition}}

 \newcommand{\bc}[1]{\begin{corollary}\label{#1}}
 \newcommand{\ec}{\end{corollary}}

 \newcommand{\bpr}{\begin{proof}}
 \newcommand{\epr}{\end{proof}}

 \newcommand{\bi}{\begin{itemize}}
 \newcommand{\ei}{\end{itemize}}

 \newcommand{\ben}{\begin{enumerate}}
 \newcommand{\een}{\end{enumerate}}

 \def\botcaption#1#2{\medskip\centerline{{\scshape #1.}\kern8pt
 {\rm #2}}\bigskip}

 \def \ba {\begin{array}}
 \def \ea {\end{array}}

 \def \Z {{\mathbb Z}}
 \def \R {{\mathbb R}}
 
 \def \N {{\mathbb N}}
 \def \P {{\mathbb P}}
 \def \E {{\mathbb E}}

 \def \cX {{\mathcal X}}
 \def \cC {{\mathcal C}}
 
 \def \cW {{\mathcal W}}
 \def \cD {{\mathcal D}}
 \def \cL {{\mathcal L}}
 \def \cI {{\mathcal I}}
 \def \cJ {{\mathcal J}}
 \def \cR {{\mathcal R}}
 \def \cA {{\mathcal A}}
 \def \cM {{\mathcal M}}
 \def \cN {{\mathcal N}}
 \def \cE {{\mathcal E}}
 \def \cP {{\mathcal P}}
 \def \cQ {{\mathcal Q}}
 \def \cB {{\mathcal B}}

 \def \cC {{\mathcal C}}
 \def \cU {{\mathcal U}}
 \def \cV {{\mathcal V}}
 \def \cH {{\mathcal H}}

 \def \cS {{\mathcal S}}

 \def \AB {\mathrm{int}}
 \def \nAB {\mathrm{nint}}

 \def \ind {{1}}
 \def \gep {{\varepsilon}}

 \def \DOM {{\hbox{\footnotesize\rm DOM}}}
 \def \CONE {{\hbox{\footnotesize\rm CONE}}}
 \def \EIGH {{\hbox{\footnotesize\rm EIGH}}}

\begin{document}

\title{Free energy of a copolymer in a micro-emulsion}

\author{\renewcommand{\thefootnote}{\arabic{footnote}}
F.\ den Hollander
\footnotemark[1]\,\,\,\footnotemark[2]
\\
\renewcommand{\thefootnote}{\arabic{footnote}}
N.\ P\'etr\'elis
\footnotemark[3]
}

\footnotetext[1]
{Mathematical Institute, Leiden University, P.O.\ Box 9512,
2300 RA Leiden, The Netherlands}\,

\footnotetext[2]
{EURANDOM, P.O.\ Box 513, 5600 MB Eindhoven, The Netherlands}

\footnotetext[3]
{Laboratoire de Math\'ematiques Jean Leray UMR 6629,
Universit\'e de Nantes, 2 Rue de la Houssini\`ere,\\ 
BP 92208, F-44322 Nantes Cedex 03, France}
\maketitle

\begin{abstract}
In this paper we consider a two-dimensional model of a copolymer consisting of a 
random concatenation of hydrophilic and hydrophobic monomers, immersed in a 
micro-emulsion of random droplets of oil and water. The copolymer interacts with 
the micro-emulsion through an interaction Hamiltonian that favors matches and 
disfavors mismatches between the monomers and the solvents, in such a way that 
the interaction with the oil is stronger than with the water.

The configurations of the copolymers are directed self-avoiding paths in which 
only steps up, down and right are allowed. The configurations of the micro-emulsion 
are square blocks with oil and water arranged in percolation-type fashion. The 
only restriction imposed on the path is that in every column of blocks its vertical 
displacement on the block scale is bounded. The way in which the copolymer enters 
and exits successive columns of blocks is a directed self-avoiding path as well, 
but on the block scale. We refer to this path as the coarse-grained self-avoiding 
path. We are interested in the limit as the copolymer and the blocks become large, 
in such a way that the copolymer spends a long time in each block yet visits many 
blocks. This is a coarse-graining limit in which the space-time scales of the 
copolymer and of the micro-emulsion become separated. 

We derive a {\it variational formula} for the {\it quenched free energy per monomer}, 
where quenched means that the disorder in the copolymer and the disorder in the 
micro-emulsion are both frozen. In a sequel paper we will analyze this variational
formula and identify the phase diagram. It turns out that there are two regimes, 
{\it supercritical} and {\it subcritical}, depending on whether the oil blocks 
percolate or not along the coarse-grained self-avoiding path. The phase diagrams 
in the two regimes turn out to be completely different.

In earlier work we considered the same model, but with an unphysical restriction: 
paths could enter and exit blocks only at diagonally opposite corners. Without 
this restriction, the variational formula for the quenched free energy is more 
complicated, but in the sequel paper we will see that it is still tractable enough 
to allow for a qualitative analysis of the phase diagram.

Part of our motivation is that our model can be viewed as a coarse-grained version 
of the well-known {\it directed polymer with bulk disorder}. The latter has been 
studied intensively in the literature, but no variational formula is as yet available. 

\vskip 0.5truecm
\noindent
\emph{AMS} 2000 \emph{subject classifications.} 60F10, 60K37, 82B27.
Acknowledgment: FdH is supported by ERC Advanced Grant 267356 VARIS. NP is 
grateful for hospitality at EURANDOM in November 2010, and at the Mathematical 
Institute of Leiden University in October and November 2011 as visiting researcher 
within the VARIS-project.
\end{abstract}

%%%%%%%%%%%%%%%%%%%%%%%%%%%% SECTION 1 %%%%%%%%%%%%%%%%%%%%%%%%%%%%%%%%%%%%%%%%

\section{Introduction and main result}
\label{S1}

In Section~\ref{Mfe} we define the model. In Section~\ref{varfor} we state our main 
result, a \emph{variational formula for the quenched free energy per monomer of a random 
copolymer in a random emulsion} (Theorem \ref{varformula} below). In Section~\ref{discus}
we discuss the significance of this variational formula and place it in a broader
context. Section~\ref{keyingr} gives a precise definition of the various ingredients 
in the variational formula, and states some key properties of these ingredients 
formulated in terms of a number of propositions. The proof of these propositions 
is deferred to Section~\ref{proofprop1}. The proof of the variational formula is 
given in Section~\ref{proofofgene}. Appendices~\ref{Path entropies}--\ref{Computation} 
contain a number of technical facts that are needed in 
Sections~\ref{keyingr}--\ref{proofofgene}. 

For a general overview on polymers with disorder, we refer the reader to the monographs 
by Giacomin~\cite{G07} and den Hollander~\cite{dH}.

%%%%%%%%%%%%

\subsection{Model and free energy}
\label{Mfe}

To build our model, we distinguish between three scales: (1) the \emph {microscopic} 
scale associated with the size of the monomers in the copolymer ($=1$, by convention); 
(2) the {\emph {mesoscopic}} scale associated with the size of the droplets in the 
micro-emulsion ($L_n\gg1$); (3) the {\emph {macroscopic}} scale associated with the 
size of the copolymer ($n\gg L_n$). 

\medskip\noindent
{\bf Copolymer configurations.}
Pick $n\in \N\cup \{\infty\}$ and let $\cW_n$ be the set of $n$-step \emph{directed 
self-avoiding paths} starting at the origin and being allowed to move \emph{upwards, 
downwards and to the right}, i.e., 
\begin{align}
\label{defw}
\nonumber
\cW_n &= \big\{\pi=(\pi_i)_{i=0}^n \in (\N_0\times\Z)^{n+1}\colon\,\pi_0=(0,1),\\
&\qquad \pi_{i+1}-\pi_i\in \{(1,0),(0,1),(0,-1)\}\,\,\forall\,0\leq i< n,
\,\pi_i\neq \pi_j\,\,\forall\,0\leq i<j \leq n\big\}.
\end{align}
The copolymer is associated with the path $\pi$. The $i$-th monomer is associated with 
the bond $(\pi_{i-1},\pi_i)$. The starting point $\pi_0$ is located at $(0,1)$ for 
technical convenience only.

%%%%%%%%%%%%%%%%%%%%%%%%%%%%%%%%%%%%%%%%%%%%%%%%%%%%%%
\begin{figure}[htbp]
\begin{center}
\vspace{-.5cm}
\includegraphics[width=.27\textwidth]{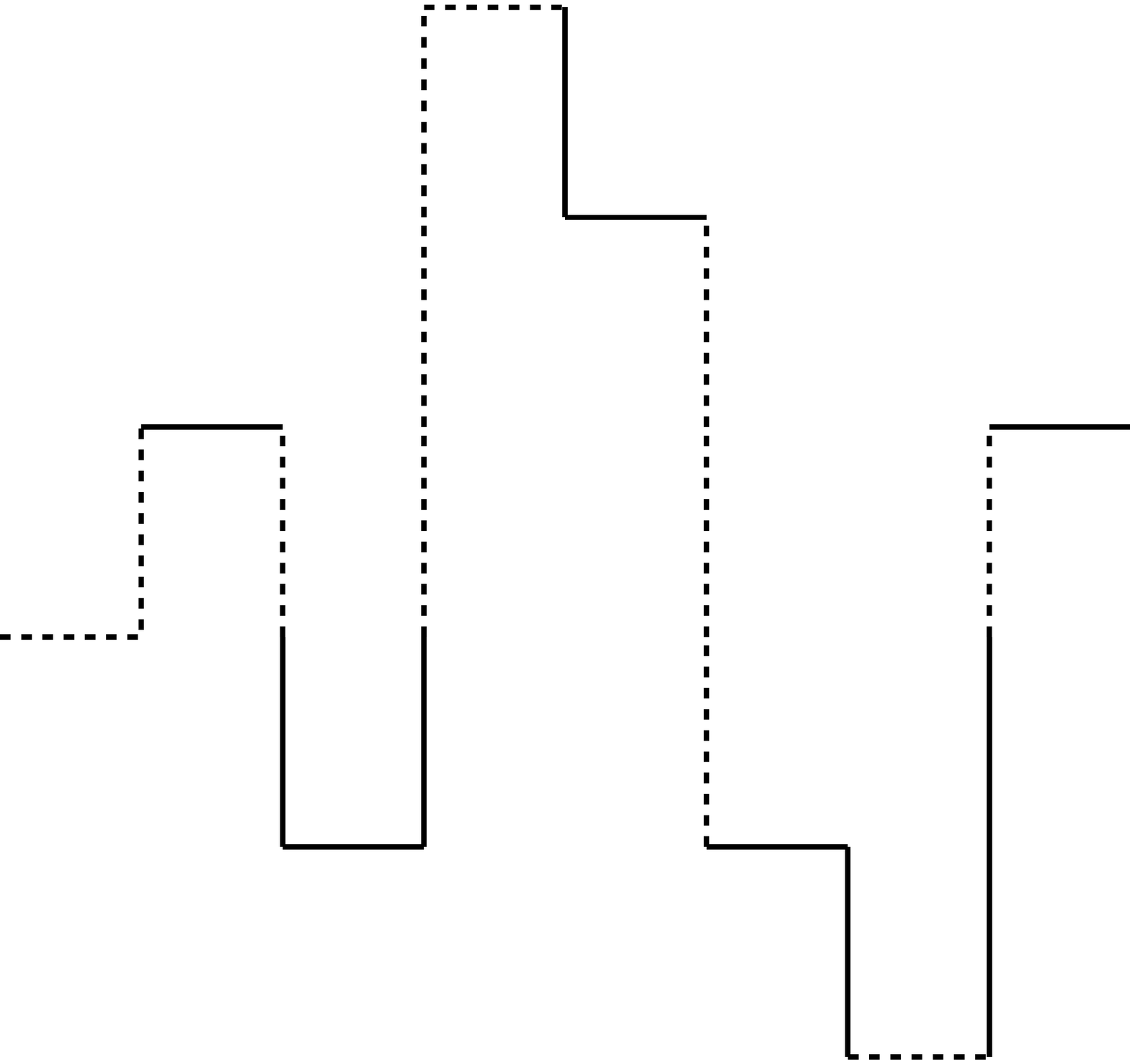}
\vspace{-1cm}
\caption{Microscopic disorder $\omega$ in the copolymer. Dashed edges represent 
monomers of type $A$ (hydrophobic), drawn edges represent monomers of type $B$
(hydrophilic).}
\label{fig-micrdis}
\end{center}
\end{figure}
%%%%%%%%%%%%%%%%%%%%%%%%%%%%%%%%%%%%%%%%%%%%%%%%%%%%%%%%

\medskip\noindent
{\bf Microscopic disorder in the copolymer.} 
Each monomer is randomly labelled $A$ (hydrophobic) or $B$ (hydrophilic), with probability 
$\frac{1}{2}$ each, independently for different monomers. The resulting labelling is 
denoted by
\be{bondlabel}
\omega = \{\omega_i \colon\, i \in \N\} \in \{A,B\}^\N
\ee
and represents the \emph{randomness of the copolymer}, i.e., $\omega_i=A$ (respectively,
$\omega_i=B$) means that the $i$-th monomer is of type $A$ (respectively, $B$); see 
Fig.~\ref{fig-micrdis}.

%%%%%%%%%%%%%%%%%%%%%%%%%%%%%%%%%%%%%%%%%%%%%%%%%%%%%%%%%%%%%%
\begin{figure}[htbp]
\begin{center}
\includegraphics[width=.42\textwidth]{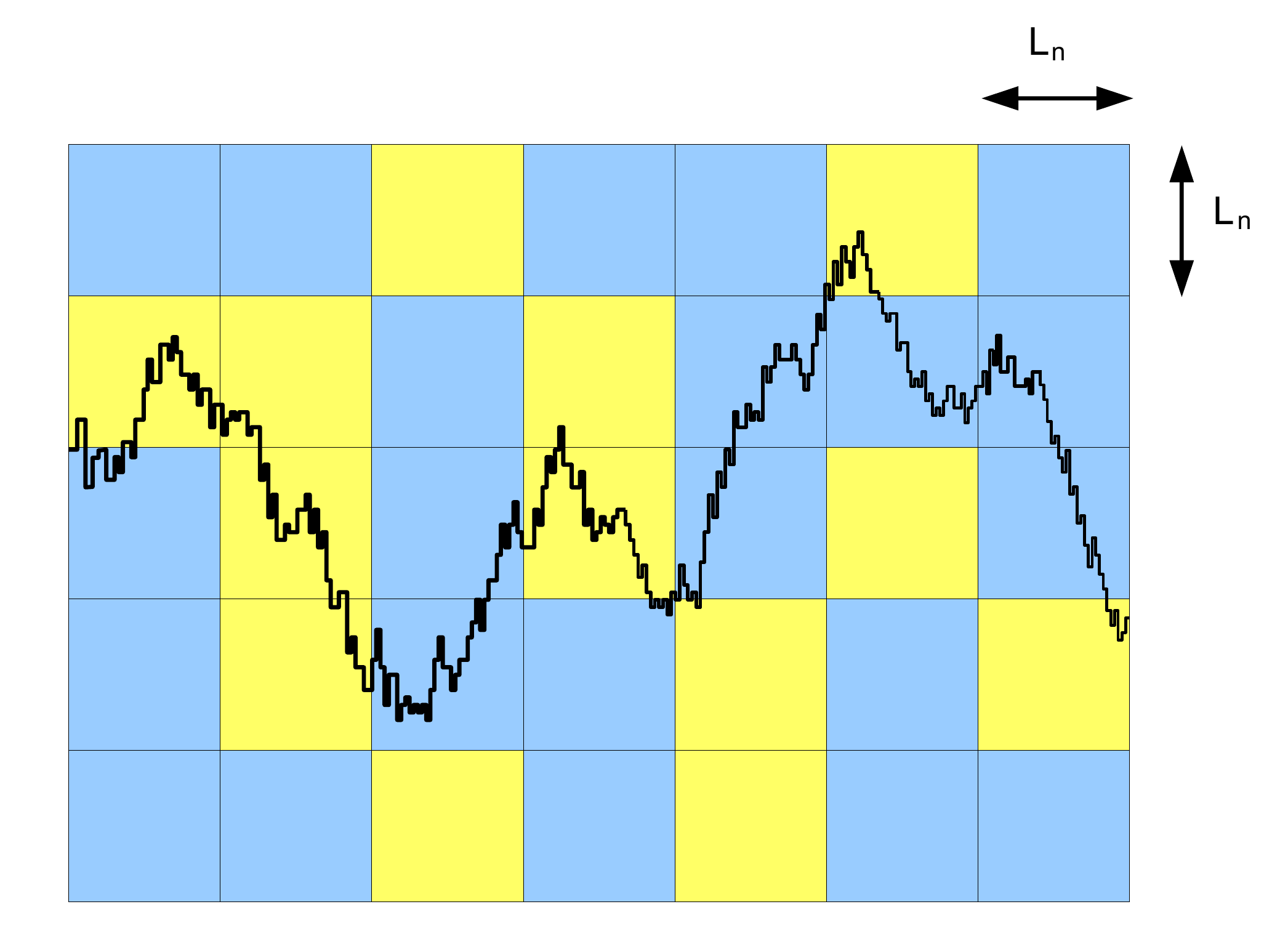}
\caption{Mesoscopic disorder $\Omega$ in the micro-emulsion. Light shaded blocks 
represent droplets of type $A$ (oil), dark shaded blocks represent droplets of 
type $B$ (water). Drawn is also the copolymer, but without an indication of the 
microscopic disorder $\omega$ attached to it.} 
\label{fig-mesdis}
\end{center}
\end{figure} 
%%%%%%%%%%%%%%%%%%%%%%%%%%%%%%%%%%%%%%%%%%%%%%%%%%%%%%%%%%%%%%%%

\medskip\noindent
{\bf Mesoscopic disorder in the micro-emulsion.}
Fix $p \in (0,1)$ and $L_n \in \N$. Partition $(0,\infty)\times\R$ into square blocks of 
size $L_n$:
\be{blocks}
(0,\infty)\times \R= \bigcup_{x \in \N_0 \times \Z} \Lambda_{L_n}(x), \qquad
\Lambda_{L_n}(x) = xL_n + (0,L_n]^2.
\ee
Each block is randomly labelled $A$ (oil) or $B$ (water), with probability $p$, 
respectively, $1-p$, independently for different blocks. The resulting labelling is 
denoted by
\be{blocklabel}
\Omega = \{\Omega(x)\colon\,x \in \N_0\times \Z\} \in \{A,B\}^{\N_0\times \Z}
\ee
and represents the \emph{randomness of the micro-emulsion}, i.e., $\Omega(x)=A$ 
(respectively, $\Omega(x)=B$) means that the $x$-th block is of type $A$ (repectively, 
$B$); see Fig.~\ref{fig-mesdis}. The size of the blocks $L_n$ is assumed to be 
non-decreasing and to satisfy
\be{speed}
\lim_{n\to \infty} L_n = \infty \quad\text{and} \quad 
\lim_{n\to \infty} \tfrac{L_n}{n} = 0,
\ee  
i.e., the blocks are large compared to the monomer size but (sufficiently) small 
compared to the copolymer size. For convenience we assume that if an $A$-block 
and a $B$-block are on top of each other, then the interface belongs to the 
$A$-block.

\medskip\noindent 
{\bf Path restriction.} 
We bound the vertical displacement on the block scale in each column of blocks by 
$M\in\N$. The value of $M$ will be arbitrary but fixed. In other words, instead
of considering the full set of trajectories $\cW_n$, we consider only trajectories 
that exit a column through a block at most $M$ above or $M$ below the block where
the column was entered (see Fig.~\ref{fig-const}). Formally, we partition $(0,\infty)
\times \R$ into columns of blocks of width $L_n$, i.e.,  
\be{blockcol}
(0,\infty)\times\R = \cup_{j\in \N_0} \cC_{j,L_n}, \qquad 
\cC_{j,L_n}=\cup_{k\in \Z} \Lambda_{L_n}(j,k),
\ee
where $C_{j,L_n}$ is the $j$-th column. For each $\pi \in \cW_n$, we let $\tau_j$ 
be the time at which $\pi$ leaves the $(j-1)$-th column and enters the $j$-th column, 
i.e.,
\be{deftau}
\tau_j = \sup\{i\in \N_{0} \colon \,\pi_i\in \cC_{j-1,n}\}
= \inf \{i\in \N_0 \colon \, \pi_i\in \cC_{j,n} \}-1,
\qquad j = 1,\dots,N_\pi-1,
\ee
where $N_\pi$ indicates how many columns have been visited by $\pi$. Finally, we let 
$v_{-1}(\pi)=0$ and, for $j \in \{0,\dots,N_\pi-1\}$, we let $v_{j}(\pi) \in \Z$ be 
such that the block containing the last step of the copolymer in $\cC_{j,n}$ is 
labelled by $(j,v_j(\pi))$, i.e., $(\pi_{\tau_{j+1}-1},\pi_{\tau_{j+1}}) \in \Lambda_{L_N}
(j,v_j(\pi))$. Thus, we restrict $\cW_{n}$ to the subset $\cW_{n,M}$ defined as 
\be{defwalt}
\cW_{n,M} = \big\{\pi\in \cW_n \colon\, |v_j(\pi)-v_{j-1}(\pi)|\leq M
\,\,\forall\,j\in \{0,\dots,N_\pi-1\} \big\}.
\ee

%%%%%%%%%%%%%%%%%%%%%%%%%%%%%%%%%%%%%%%%%%%%%%%%%%%%%%%%%%%%%%
\begin{figure}[htbp]
\begin{center}
\includegraphics[width=.3\textwidth]{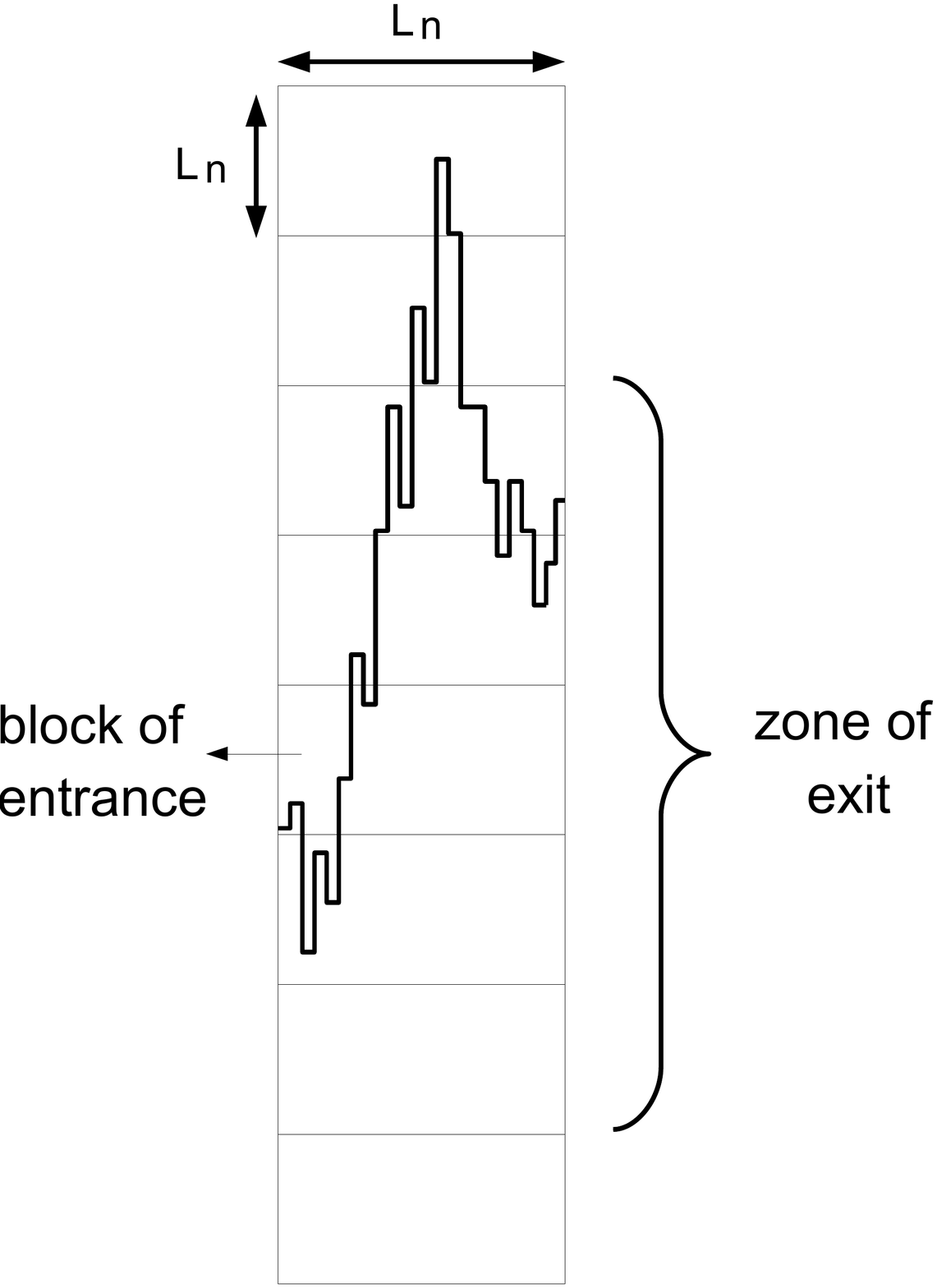}
\caption{Example of a trajectory $\pi \in\cW_{n,M}$ with $M=2$ crossing the column 
$\cC_{0,L_n}$ with $v_0(\pi)=2$.} 
\label{fig-const}
\end{center}
\end{figure} 
%%%%%%%%%%%%%%%%%%%%%%%%%%%%%%%%%%%%%%%%%%%%%%%%%%%%%%%%%%%%%%%%

\medskip\noindent 
{\bf Hamiltonian and free energy.} Given $\omega,\Omega,M$ and $n$, with each path 
$\pi \in \cW_{n,M}$ we associate an \emph{energy} given by the Hamiltonian
\be{Hamiltonian}
H_{n,L_n}^{\omega,\Omega}(\pi)
=  \sum_{i=1}^n \Big(\alpha\, 1\Big\{\omega_i=\Omega^{L_n}_{(\pi_{i-1},\pi_i)}=A\Big\}
+ \beta\, 1\left\{\omega_i=\Omega^{L_n}_{(\pi_{i-1},\pi_i)}=B\right\}\Big),
\ee
where $\Omega^{L_n}_{(\pi_{i-1},\pi_i)}$ denotes the label of the block the step 
$(\pi_{i-1},\pi_i)$ lies in. What this Hamiltonian does is count the number of 
$AA$-matches and $BB$-matches and assign them energy $\alpha$ and $\beta$, respectively, 
where $\alpha,\beta\in\R$. (Note that the interaction is assigned to bonds rather than 
to sites, and that we do not follow the convention of putting a minus sign in front of 
the Hamiltonian.) Similarly to what was done in our earlier papers \cite{dHW06}, 
\cite{dHP07b}, \cite{dHP07c}, \cite{dHP07a}, without loss of generality we may 
restrict the interaction parameters to the cone
\be{defcone}
\CONE = \{(\alpha,\beta)\in\R^2\colon\,\alpha\geq |\beta|\}.
\ee
For $n\in \N$, the free energy per monomer is defined as
\be{partfunc}
f_{n}^{\omega,\Omega}(M;\alpha,\beta)
=\tfrac{1}{n}\log Z_{n,L_n}^{\omega,\Omega}(M;\alpha,\beta) 
\quad \text{with}\quad Z_{n,L_n}^{\omega,\Omega}(M)
=\sum_{\pi \in \cW_{n,M}} e^{H_{n,L_n}^{\omega,\Omega}(\pi)},
\ee
and in the limit as $n\to\infty$ the free energy per monomer is given by
\be{felimdef}
f(M;\alpha,\beta) = \lim_{n\to \infty}f_{n,L_n}^{\omega,\Omega}(M;\alpha,\beta), 
\ee
provided this limit exists. 

Henceforth, we subtract from the Hamiltonian the quantity $\alpha \sum_{i=1}^n 
1\left\{\omega_i=A\right\}$, which by the law of large numbers is $\tfrac{\alpha}{2} 
n (1+o(1))$ as $n\to \infty$ and corresponds to a shift of $-\tfrac{\alpha}{2}$ in 
the free energy. The latter transformation allows us to lighten the notation, 
starting with the Hamiltonian, which becomes
\be{Hamiltonianalt}
H_{n,L_n}^{\omega,\Omega}(\pi)
= \sum_{i=1}^n \Big(\beta\, 1\left\{\omega_i=B\right\}
-\alpha\,1\left\{\omega_i=A\right\}\Big)\,  
1\left\{\Omega^{L_n}_{(\pi_{i-1},\pi_i)}=B\right\}.
\ee

%%%%%%%%%%%%%%%%%%%%%%%%%%

\subsection{Variational formula for the quenched free energy}
\label{varfor}

Theorem~\ref{varformula} below is the main result of our paper. It expresses the 
quenched free energy per monomer in the form of a \emph{variational formula}. To 
state this variational formula, we need to define some quantities that capture 
the way in which the copolymer moves inside single columns of blocks and samples 
different columns. A precise definition of these quantities will be given in 
Section~\ref{keyingr}. 

Given $M\in \N$, the \emph{type} of a column is denoted by $\Theta$ and takes values 
in a \emph{type space} $\overline\cV_M$, defined in Section~\ref{frenp1}. The type 
indicates both the vertical displacement of the copolymer in the column and the 
mesoscopic disorder seen relative to the block where the copolymer enters the column. 
In Section~\ref{frenp1} we further associate with each $\Theta\in \overline\cV_M$ a 
quantity $u_\Theta \in [t_\Theta,\infty)$ that indicates how many \emph{steps on scale} 
$L_n$ the copolymer makes in columns of type $\Theta$, where $t_\Theta$ is the minimal 
number of steps required to cross a column of type $\Theta$. These numbers are gathered 
into the set 
\be{BVdef}
\cB_{\overline\cV_M}=\big\{(u_\Theta)_{\Theta\in \overline\cV_M}\in\R^{\overline\cV_M}
\colon\,u_\Theta\geq t_\Theta\,\,\forall\,\Theta \in \overline\cV_M,\,
\Theta \mapsto u_\Theta \mbox{ continuous}\big\}.
\ee 
In Section~\ref{newsec} we introduce the \emph{free energy per step} $\psi(\Theta,
u_\Theta;\alpha,\beta)$ associated with the copolymer when crossing a column of 
type $\Theta$ in $u_\Theta$ steps, which depends on the parameters $\alpha,\beta$. 
After that it remains to define the family of \emph{frequencies with which successive
pairs of different types of columns can be visited by the copolymer}. This is done in 
Section \ref{Percofreq} and is given by a family of probability laws $\rho$ in 
$\cM_1(\overline\cV_M)$, the set of probability measures 
on $\overline\cV_M$, forming a set 
\be{Rpdef}
\cR_{p,M} \subset \cM_1(\overline\cV_M),
\ee
which depends on $M$ and on the parameter $p$.

\begin{theorem}
\label{varformula}
For every $(\alpha,\beta)\in\CONE$, $M\in \N$ and $p \in (0,1)$ the free energy in 
{\rm \eqref{felimdef}} exists for $\P$-a.e.\ $(\omega,\Omega)$ and in 
$L^1(\P)$, and is given by
\be{genevar}
f(M;\,\alpha,\beta) =\sup_{\rho\in \cR_{p,M}}\,\sup_{(u_\Theta)_{\Theta\in \overline\cV_M}\,
\in\,\cB_{\,\overline\cV_M}} V(\rho,u)
\ee
with 
\be{genevar2}
V(\rho,u) = \frac{\int_{\overline\cV_M}\,u_\Theta\,\psi(\Theta,u_{\Theta};\alpha,\beta)\,
\rho(d\Theta)}{\int_{\overline\cV_M}\,u_\Theta\, \rho(d\Theta)}\quad \text{if}\quad 
\int_{\overline\cV_M}\,u_\Theta\, \rho(d\Theta)=\infty,
\ee
and $V(\rho,u)=-\infty$ otherwise.
\end{theorem}

%%%%%%%%%%%%%%%%%%%%%%%%%%%%%%%%%%%%%%%%%%%%%%%%%%%%%%%%%

\subsection{Discussion}
\label{discus}

{\bf Structure of the variational formula.}
The variational formula in \eqref{genevar} has a simple structure: each column
type $\Theta$ has its own number of monomers $u_\Theta$ and its own free energy 
per monomer $\psi(\Theta,u_\Theta;\alpha,\beta)$ (both on the mesoscopic scale), 
and the total free energy per monomer is obtained by weighting each column type 
with the frequency $\rho_1(d\Theta)$ at which it is visited by the copolymer. 
The numerator is the total free energy, the denominator is the total number of 
monomers (both on the mesoscopic scale). The variational formula optimizes over 
$(u_\Theta)_{\Theta\in\overline\cV_M} \in\cB_{\,\overline \cV_M}$ and $\rho\in \cR_{p,M}$. 
The reason why these two suprema appear in \eqref{genevar} is that, as a consequence 
of assumption \eqref{speed}, the \emph{mesoscopic scale carries no entropy}: all 
the entropy comes from the microscopic scale, through the free energy per monomer 
in single columns.   

In Section~\ref{keyingr} we will see that $\psi(\Theta,u_\Theta;\alpha,\beta)$
in turn is given by a variational formula that involves the entropy of the 
copolymer inside a single column (for which an explicit expression is available) 
and the quenched free energy per monomer of a copolymer near a \emph{single linear 
interface} (for which there is an abundant literature). Consequently, the free 
energy of our model with a \emph{random geometry} is directly linked to the free 
energy of a model with a \emph{non-random geometry}. This will be crucial for our 
analysis of the free energy in the sequel paper. 

\medskip\noindent
{\bf Removal of the corner restriction.} In our earlier papers \cite{dHW06}, \cite{dHP07b}, 
\cite{dHP07c}, \cite{dHP07a}, we allowed the configurations of the copolymer to be given 
by the subset of $\cW_n$ consisting of those paths that enter pairs of blocks through a 
common corner, exit them at one of the two corners diagonally opposite and in between 
stay confined to the two blocks that are seen upon entering. The latter is an 
{\it unphysical restriction} that was adopted to simplify the model. In these 
papers we derived a variational formula for the free energy per step that had a 
simpler structure. We analyzed this variational formula as a function of $\alpha,
\beta,p$ and found that there are two regimes, \emph{supercritical} and \emph{subcritical}, 
depending on whether the oil blocks percolate or not along the coarse-grained self-avoiding 
path. In the supercritical regime the phase diagram turned out to have two phases, 
in the subcritical regime it turned out to have four phases, meeting at two tricritical 
points. 

In a sequel paper we will show that the phase diagrams found in the restricted model 
are largely \emph{robust} against the removal of the corner restriction, despite the 
fact that the variational formula is more complicated. In particular, there are again
two types of phases: \emph{localized phases} (where the copolymer spends a positive 
fraction of its time near the $AB$-interfaces) and \emph{delocalized phases} (where 
it spends a zero fraction near the $AB$-interfaces). Which of these phases occurs 
depends on the parameters $\alpha,\beta,p$. It is energetically favorable for the 
copolymer to stay close to the $AB$-interfaces, where it has the possibility of 
placing more than half of its monomers in their preferred solvent (by switching 
sides when necessary), but this comes with a loss of entropy. The competition 
between energy and entropy is controlled by the energy parameters $\alpha,\beta$ 
(determining the reward of switching sides) and by the density parameter $p$ 
(determining the density of the $AB$-interfaces).
 
%%%%%%%%%%%%%%%%%%%%%%%%%%%%%%%%%%%%%%%%%%%%%%%%%%%%%%%%%%%%%%
\begin{figure}[htbp]
\vspace{.3cm}
\begin{center}
\includegraphics[width=.3\textwidth]{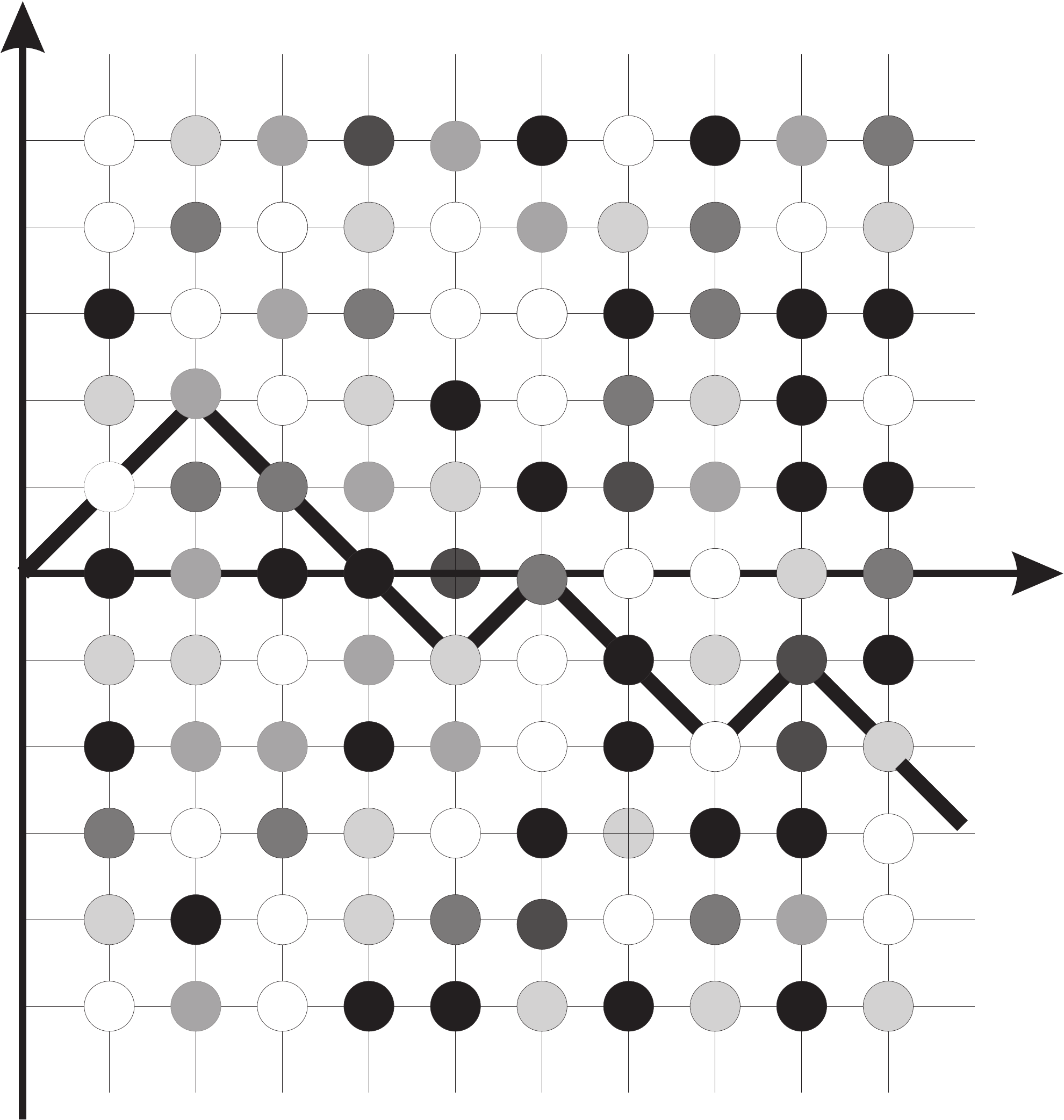}
\caption{Picture of a directed polymer with bulk disorder. The different shades
of black, grey and white represent different values of the disorder.} 
\label{fig-dirpolbulk}
\end{center}
\end{figure} 
\vspace{-.3cm}
%%%%%%%%%%%%%%%%%%%%%%%%%%%%%%%%%%%%%%%%%%%%%%%%%%%%%%%%%%%%%%%%
  
\medskip\noindent 
{\bf Comparison with the directed polymer with bulk disorder.}
A model of a polymer with disorder that has been studied intensively in the literature 
is the \emph{directed polymer with bulk disorder}. Here, the set of paths is 
\be{dpre1}
\cW_n = \big\{\pi=(i,\pi_i)_{i=0}^n \in (\N_0\times\Z^d)^{n+1}\colon\,\pi_0=0,\,
\|\pi_{i+1}-\pi_i\|=1\,\,\forall\,0 \leq i<n\big\},
\ee
where $\|\cdot\|$ is the Euclidean norm on $\Z^d$, and the Hamiltonian is  
\be{dpre2}
H^\omega_n(\pi) = \lambda \sum_{i=1}^n \omega(i,\pi_i),
\ee
where $\lambda>0$ is a parameter and $\omega = \{\omega(i,x)\colon\,i\in\N,\,x\in\Z^d\}$ 
is a field of i.i.d.\ $\R$-valued random variables with zero mean, unit variance and
finite moment generating function, where $\N$ is time and $\Z^d$ is space (see 
Fig.~\ref{fig-dirpolbulk}). This model can be viewed as a version of a copolymer 
in a micro-emulsion where the droplets are of the \emph{same} size as the monomers. 
For this model \emph{no variational formula is known for the free energy}, and the 
analysis relies on the application of martingale techniques (for details, see e.g.\ 
den Hollander~\cite{dH}, Chapter 12).

In our model (which is restricted to $d=1$ and has self-avoiding paths that may move 
north, south and east instead of north-east and south-east), the droplets are much larger 
than the monomers. This causes a \emph{self-averaging of the microscopic disorder}, 
both when the copolymer moves inside one of the solvents and when it moves near an 
interface. Moreover, since the copolymer is much larger than the droplets, also 
\emph{self-averaging of the mesoscopic disorder} occurs. This is why the free energy 
can be expressed in terms of a variational formula, as in Theorem~\ref{varformula}.
In the sequel paper we will see that this variational formula acts as a \emph{jumpboard} 
for a detailed analysis of the phase diagram. Such a detailed analysis is lacking for 
the directed polymer with bulk disorder.

The directed polymer in random environment has two phases: a \emph{weak disorder
phase} (where the quenched and the annealed free energy are asymptotically comparable) 
and a \emph{strong disorder phase} (where the quenched free energy is asymptotically 
smaller than the annealed free energy). The strong disorder phase occurs in dimension 
$d=1,2$ for all $\lambda>0$ and in dimension $d \geq 3$ for $\lambda>\lambda_c$, with
$\lambda_c \in [0,\infty]$ a critical value that depends on $d$ and on the law of the 
disorder. It is predicted that in the strong disorder phase the copolymer moves within 
a narrow corridor that carries sites with high energy (recall our convention of not
putting a minus sign in front of the Hamiltonian), resulting in \emph{superdiffusive} 
behavior in the spatial direction. We expect a similar behavior to occur in the localized 
phases of our model, where the polymer targets the $AB$-interfaces. It would be interesting 
to find out how far the coarsed-grained path in our model travels vertically as 
a function of $n$.

%%%%%%%%%%%%%%%%%%%%%%% SECTION 2 %%%%%%%%%%%%%%%%%%%%%%%%%%%%%%%%%%%%%%%%%%%%%%%%%

\section{Key ingredients of the variational formula}
\label{keyingr}

In this section we give a precise definition of the various ingredients in 
Theorem~\ref{varformula}. In Section~\ref{lininter} we define the entropy 
of the copolymer inside a single column (Proposition~\ref{lementr}) and the 
quenched free energy per monomer for a random copolymer near a \emph{single 
linear interface} (Proposition~\ref{l:feinflim}), which serve as the key 
\emph{microscopic} ingredients. In Section~\ref{freeenp} these quantities 
are used to derive variational formulas for the quenched free energy per 
monomer in a single column (Proposition~\ref{convfree1}). These variational 
formulas come in two varieties (Propositions~\ref{energ} and \ref{energ1}). 
In Section~\ref{Percofreq} we define certain percolation frequencies 
describing how the copolymer samples the droplets in the emulsion 
(Proposition~\ref{properc}), which serve as the key \emph{mesoscopic} 
ingredients. Propositions~\ref{convfree1}--\ref{energ1} will be proved in
Section~\ref{proofprop1}. The results in Sections~\ref{freeenp}--\ref{Percofreq} 
will be used in Section~\ref{proofofgene} to prove our variational formula 
in Theorem~\ref{varformula} for the copolymer in the emulsion, which is our 
main \emph{macroscopic} object of interest.

%%%%%%%%%%%%%%%%%

\subsection{Path entropies and free energy along a single linear interface}
\label{lininter}

{\bf Path entropies.} 
We begin by defining the entropy of a path crossing a single column. Let 
\begin{align}
\label{add4}
\nonumber 
\cH &= \{(u,l)\in [0,\infty) \times \R \colon\, u\geq 1+|l|\},\\
\cH_L &= \big\{(u,l)\in \cH \colon\,l \in \tfrac{\Z}{L},\, 
u\in 1+|l|+\tfrac{2\N}{L}\big\}, \qquad  L \in \N,
\end{align}
and note that $\cH\cap \mathbb{Q}^2=\cup_{L\in \N} \cH_L$. For $(u,l) \in \cH_L$, 
denote by $\cW_L(u,l)$ (see Fig.~\ref{figtra}) the set containing those paths 
$\pi=(0,-1)+\widetilde \pi$ with $\widetilde\pi\in\cW_{uL}$ (recall \eqref{defw}) 
for which $\pi_{uL}= (L,lL)$. The entropy per step associated with the paths in 
$\cW_L(u,l)$ is given by
\be{ttrajblock}
\tilde{\kappa}_L(u,l)=\tfrac{1}{uL} \log |\cW_L(u,l)|.
\ee

%%%%%%%%%%%%%%%%%%%%%%%%%%%%%%%%%%%%%%%%%%%%%%%%%%%%%%%%%%%%%%
\begin{figure}[htbp]
\vspace{-.5cm}
\begin{center}
\includegraphics[width=.44\textwidth]{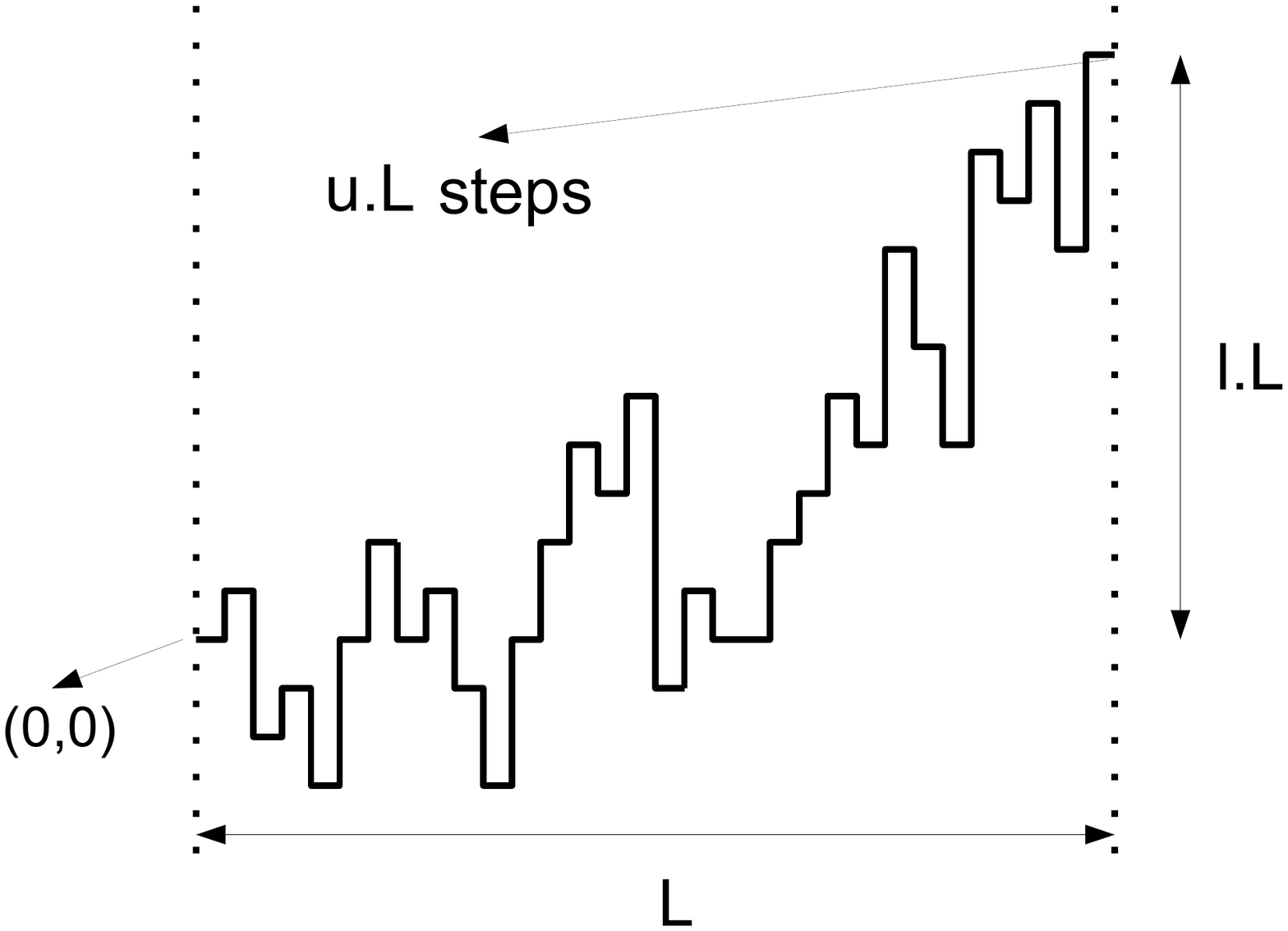}
\end{center}
\vspace{-4.7cm}
\caption{A trajectory in $\cW_L(u,l)$.}
\label{figtra}
\end{figure}
%%%%%%%%%%%%%%%%%%%%%%%%%%%%%%%%%%%%%%%%%%%%%%%%%%%%%%%%%%%%%%%

The following proposition will be proved in Appendix~\ref{Path entropies}. 

\bp{lementr}
For all $(u,l)\in\cH\cap\mathbb{Q}^2$ there exists a $\tilde{\kappa}(u,l) \in 
[0,\log 3]$ such that 
\be{conventr}
\lim_{ {L\to \infty} \atop {(u,l)\in \cH_L} } \tilde{\kappa}_L(u,l)
= \sup_{  {L\in \N} \atop {(u,l)\in \cH_L}} \tilde{\kappa}_L(u,l)
= \tilde{\kappa}(u,l). 
\ee  
\ep

\noindent
An explicit formula is available for $\tilde{\kappa}(u,l)$, namely, 
\be{kapexplform}
\tilde{\kappa}(u,l) = \left\{\begin{array}{ll}
\kappa(u/|l|,1/|l|), &\l \neq 0,\\
\hat{\kappa}(u), &l = 0,
\end{array}
\right.
\ee
where $\kappa(a,b)$, $a\geq 1+b$, $b\geq 0$, and $\hat{\kappa}(\mu)$, $\mu \geq 1$, 
are given in \cite{dHW06}, Section 2.1, in terms of elementary variational formulas
involving entropies (see \cite{dHW06}, proof of Lemmas 2.1.1--2.1.2).

\medskip\noindent 
{\bf Free energy along a single linear interface.} 
To analyze the free energy per monomer in a single column we need to first analyse the
free energy per monomer when the path moves in the vicinity of an $AB$-interface. To that
end we consider a \emph{single linear interface} $\cI$ separating a liquid $B$ in the 
lower halfplane from a liquid $A$ in the upper halfplane (including the interface 
itself).

For $L\in \N$ and $\mu\in 1+\frac{2\N}{L}$, let $\cW^\cI_L(\mu)=\cW_L(\mu,0)$ denote 
the set of $\mu L$-step directed self-avoiding paths starting at $(0,0)$ and ending at 
$(L,0)$. Define
\be{feinf}
\phi^{\omega,\cI}_L(\mu) = \frac{1}{\mu L} \log Z^{\omega,\cI}_{L,\mu} 
\quad \text{ and } \quad \phi^\cI_L(\mu)=\E[\phi^{\omega,\cI}_L(\mu)] ,
\ee
with
\be{Zinf}
\begin{aligned}
Z^{\omega,\cI}_{L,\mu}
&= \sum_{\pi\in\cW_L^\cI(\mu)} \exp\left[H^{\omega,\cI}_{L}(\pi)\right],\\
H^{\omega,\cI}_{L}(\pi)
&= \sum_{i=1}^{\mu L}\big(\beta\, 1\{\omega_i=B\}-\alpha\,
1\{\omega_i=A\}\big)\  1\{(\pi_{i-1},\pi_i) < 0\},
\end{aligned}
\ee
where $(\pi_{i-1},\pi_i) < 0$ means that the $i$-th step lies in the lower halfplane, 
strictly below the interface (see Fig.~\ref{fig7}).

The following proposition was derived in \cite{dHW06}, Section 2.2.2.

\bp{l:feinflim} 
For all $(\alpha,\beta)\in\CONE$ and $\mu \in \mathbb{Q}\cap[1,\infty)$ there exists 
a $\phi^\cI(\mu)=\phi^\cI(\mu;\alpha,\beta)$ $\in\R$ such that
\be{fesainf}
\lim_{ {L\to \infty} \atop {\mu\in 1+\frac{2\N}{L}} } 
\phi^{\omega,\cI}_L(\mu) = \phi^\cI(\mu) = \phi^\cI(\mu;\alpha,\beta)
\quad \text{for $\P$-a.e. $\omega$ and in $L^1(\P)$.}
\ee
\ep
It is easy to check (via concatenation of trajectories) that $\mu \mapsto \mu 
\phi^{\cI}(\mu;\alpha,\beta)$ is concave. For technical reasons we need to assume 
that it is {\it strictly concave}, a property which we believe to be true but are 
unable to verify:

\begin{assumption}
\label{assu}
For all $(\alpha,\beta)\in\CONE$ the function $\mu\mapsto\mu \phi^{\cI}(\mu;\alpha,\beta)$ 
is strictly concave on $[1,\infty)$.
\end{assumption}

%%%%%%%%%%%%%%%%%%%%%%%%%%%%%%%%%%%%%%%%%%%%%%%%%%%%%%%%%%%%%%
\begin{figure}[htbp]
\vspace{-1cm}
\begin{center}
\includegraphics[width=.58\textwidth]{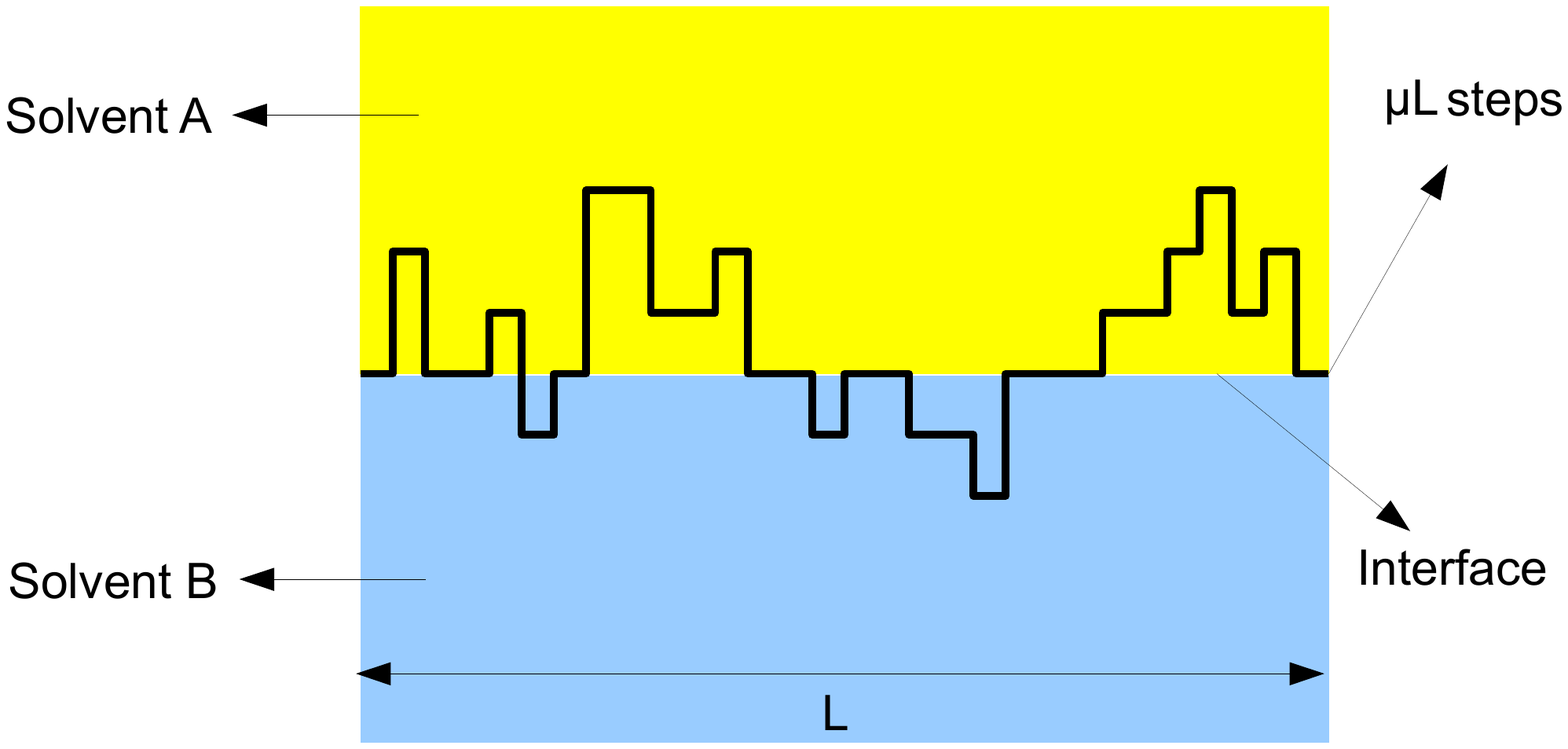}
\end{center}
\vspace{-8cm}
\caption{Copolymer near a single linear interface.}
\label{fig7}
\end{figure}
%%%%%%%%%%%%%%%%%%%%%%%%%%%%%%%%%%%%%%%%%%%%%%%%%%%%%%%%%%%%%%%

\subsection{Free energy in a single column and variational formulas}
\label{freeenp}

In this section we use Propositions~\ref{lementr}--\ref{l:feinflim} to derive
a variational formula for the free energy per step in a single column 
(Proposition~\ref{convfree1}). The variational formula comes in three varieties 
(Propositions~\ref{energ} and \ref{energ1}), depending on \emph{whether there 
is or is not an $AB$-interface between the heights where the copolymer enters 
and exits the column, and in the latter case whether an $AB$-interface is reached 
or not}. 

In what follows we need to consider the randomness in a single column. To that 
aim, we recall \eqref{blockcol}, we pick $L\in \N$ and once $\Omega$ is chosen, 
we can record the randomness of $C_{j,L}$ as
\be{add1}
\Omega_{(j,~\cdot~)}=\{\Omega_{(j,l)}\colon\, l\in \Z\}.
\ee
We will also need to consider the randomness of the $j$-th column seen by a 
trajectory that enters $\cC_{j,L}$ through the block $\Lambda_{j,k}$ with 
$k\neq 0$ instead of $k=0$. In this case, the randomness of $\cC_{j,L}$ is 
recorded as
\be{add2}
\Omega_{(j,k+~\cdot~)}=\{\Omega_{(j,k+l)}\colon\, l\in \Z\}.
\ee

Pick $L\in \N$, $\chi\in \{A,B\}^\Z$ and consider $\cC_{0,L}$ endowed with the disorder 
$\chi$, i.e., $\Omega(0,\cdot)=\chi$.  Let $(n_i)_{i\in \Z}\in \Z^\Z$ be the successive 
heights of the $AB$-interfaces in $\cC_{0,L}$ divided by $L$, i.e.,  
\be{efn}
\dots<n_{-1}<n_0\leq 0< n_1<n_2<\dots.
\ee
and the $j$-th interface of $\cC_{0,L}$ 
is $\cI_j=\{0,\dots,L\}\times \{n_j L\}$ (see 
Fig.~\ref{fig5bis}). Next, for $r\in \N_0$ we set
\begin{align}\label{defchi1}
k_{r,\chi}=&\,0 \text{ if } n_1>r \text{ and } 
k_{r,\chi}=\max\{i\geq 1\colon n_i\leq r\} \text{ otherwise},
\end{align}
while for $r\in -\N$ we set
\begin{align}\label{defchi2}
\quad \quad \quad \quad  k_{r,\chi}=&\, 0 \text{ if }
n_0\leq r\text{ and } k_{r,\chi}
=\min\{i\leq 0\colon n_{i}\geq r+1\}-1 \text{ otherwise}.
\end{align}
Thus, $|k_{r,\chi}|$ is the number of $AB$-interfaces between heigths $1$ and $r L$ 
in $\cC_{0,L}$. 

%%%%%%%%%%%%%%%%%%%%%%%%%%%%%%%%%%%%%%%%%%%%%%%%%%%%%%%%%%%%%%
\begin{figure}[htbp]
\vspace{-.5cm}
\begin{center}
\includegraphics[width=.38\textwidth]{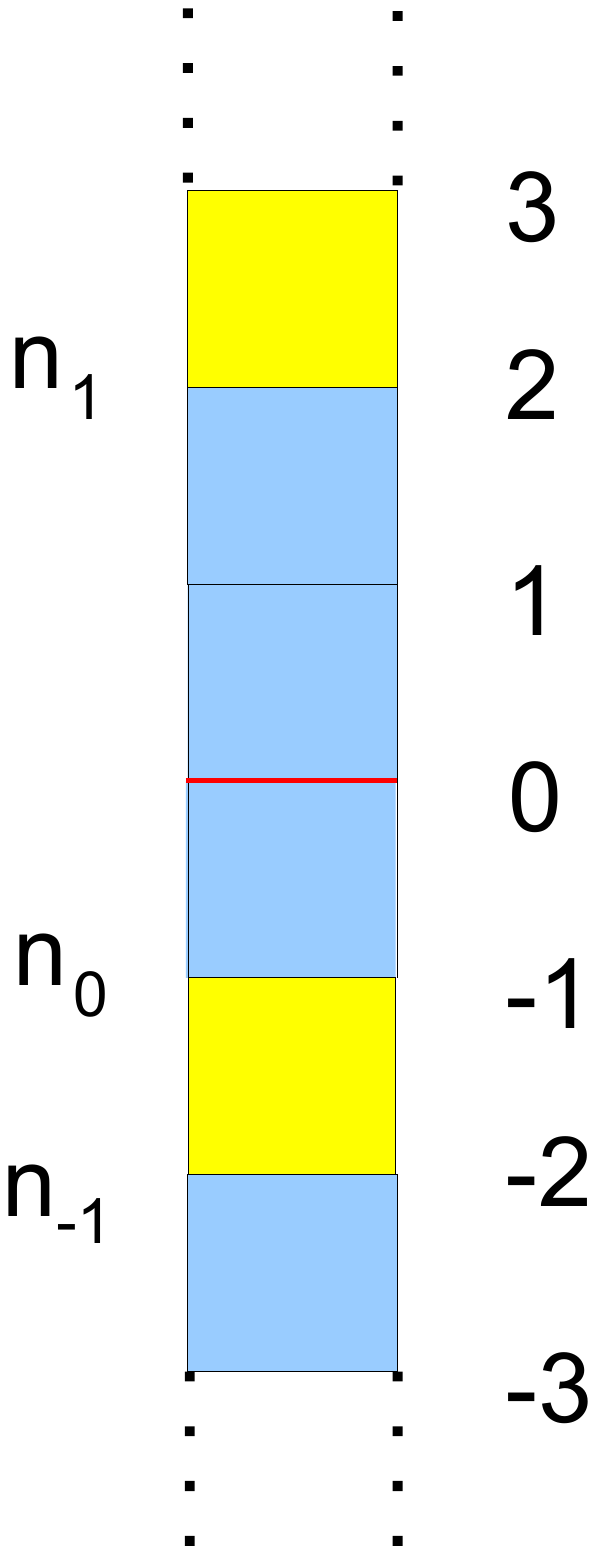}
\end{center}
\vspace{-3.5cm}
\caption{Example of a column with disorder $\chi=(\dots,\chi(-3),\chi(-2),\chi(-1),
\chi(0),\chi(1),\chi(2),$ $\dots)=(\dots,B,A,B,B,B,A,,\dots)$. In this example, for 
instance, $k_{-2,\chi}=-1$ and $k_{1,\chi}=0$.}
\label{fig5bis}
\end{figure}
%%%%%%%%%%%%%%%%%%%%%%%%%%%%%%%%%%%%%%%%%%%%%%%%%%%%%%%%%%%%%%%

%%%%%%%%%%%%%%%%%%%%%%

\subsubsection{Free energy in a single column}
\label{frenp1}

{\bf Column crossing characteristics.} 
Pick $L,M\in \N$, and consider the first column $\cC_{0,L}$. The type of $\cC_{0,L}$ is 
determined by $\Theta=(\chi,\Xi,x)$, where $\chi=(\chi_j)_{j\in \Z}$ encodes the type 
of each block in $\cC_{0,L}$, i.e., $\chi_j=\Omega_{(0,j)}$ for $j\in \Z$, and $(\Xi,x)$ 
indicates which trajectories $\pi$ are taken into account. In the latter, $\Xi$ is 
given by $(\Delta \Pi,b_0,b_1)$ such that the vertical increment in $\cC_{0,L}$ on 
the block scale is $\Delta \Pi$ and satisfies $|\Delta \Pi|\leq M$ , i.e., $\pi$ 
enters $\cC_{0,L}$ at $(0,b_0 L)$ and exits $\cC_{0,L}$ at $(L,(\Delta \Pi+b_1)L)$. 
As in \eqref{defchi1} and \eqref{defchi2}, we set $k_\Theta=k_{\Delta\Pi,\chi}$ and 
we let $\cV_\AB$ be the set containing those $\Theta$ satisfying $k_\Theta\neq 0$. 
Thus, $\Theta\in \cV_\AB$ means that the trajectories crossing $\cC_{0,L}$ from 
$(0,b_0 L)$ to $(L,(\Delta \Pi+b_1)L)$ necessarily hit an $AB$-interface, and in 
this case we set $x=1$. If, on the other hand, $\Theta\in \cV_{\nAB}=\cV\setminus
\cV_{\AB}$, then we have $k_\Theta= 0$ and we set $x=1$ when the set of trajectories 
crossing $\cC_{0,L}$ from $(0,b_0 L)$ to $(L,(\Delta \Pi+b_1)L)$ is restricted to 
those that do not reach an $AB$-interface before exiting $\cC_{0,L}$, while we set 
$x=2$ when it is restricted to those trajectories that reach at least one $AB$-interface 
before exiting $\cC_{0,L}$. To fix the possible values taken by $\Theta=(\chi,\Xi,x)$ 
in a column of width $L$, we put $\cV_{L,M}=\cV_{\AB,L,M}\cup\cV_{\nAB,L,M}$ with 
\begin{align}
\label{set2}
\nonumber \cV_{\AB,L,M}&=\big\{(\chi,\Delta\Pi,b_0,b_1,x)\in 
\{A,B\}^\Z\times \Z\times \big\{\tfrac{1}{L},\tfrac{2}{L},\dots,1\big\}^2
\times \{1\}\colon\\
\nonumber 
&\hspace{6cm}\qquad \qquad\qquad\qquad 
|\Delta\Pi|\leq M,\,k_{\Delta\Pi,\chi}\neq 0\big\},\\
\nonumber \cV_{\nAB,L,M}&=\big\{(\chi,\Delta\Pi,b_0,b_1,x)\in 
\{A,B\}^\Z\times \Z\times\big\{\tfrac{1}{L},\tfrac{2}{L},\dots,1\big\}^2
\times \{1,2\}\colon\\
&\hspace{6cm}
\qquad \qquad\qquad\qquad |\Delta\Pi|\leq M,\,  k_{\Delta\Pi,\chi}= 0\big\}.
\end{align}
Thus, the set of all possible values of $\Theta$ is $\cV_M=\cup_{L\geq 1} \cV_{L,M}$, 
which we partition into $\cV_M=\cV_{\AB,M}\cup\cV_{\nAB,M}$ (see Fig.~\ref{fig6}) with
\begin{align}
\label{sset2}
\nonumber \cV_{\AB,M}
&=\cup_{L\in \N}\  \cV_{\AB,L,M}\\
\nonumber
&= \big\{(\chi,\Delta\Pi,b_0,b_1,x)\in \{A,B\}^\Z\times \Z\times 
(\mathbb{Q}_{(0,1]})^2\times\{1\}\colon\,\,
|\Delta\Pi|\leq M,\, k_{\Delta\Pi,\chi}\neq 0\big\},\\
\nonumber \cV_{\nAB,M}&=\cup_{L\in \N}\  \cV_{\nAB,L,M}\\
&= \big\{(\chi,\Delta\Pi,b_0,b_1,x)\in \{A,B\}^\Z\times \Z\times 
(\mathbb{Q}_{(0,1]})^2\times\{1,2\}\colon\,
|\Delta\Pi|\leq M,\, k_{\Delta\Pi,\chi}= 0\big\},
\end{align}
where, for all $I\subset \R$, we set $\mathbb{Q}_{I}=I\cap \mathbb{Q}$. We define the 
closure of $\cV_M$ as $\overline\cV_M=\overline\cV_{\AB,M}\cup\overline\cV_{\nAB,M}$ with
\begin{align}
\label{sset3}
\nonumber &\overline\cV_{\AB,M}
= \big\{(\chi,\Delta\Pi,b_0,b_1,x)\in \{A,B\}^\Z\times \Z\times 
[0,1]^2\times\{1\}\colon\,|\Delta\Pi|\leq M,\, k_{\Delta\Pi,\chi}\neq 0\big\},\\
&\overline\cV_{\nAB,M}=\big\{(\chi,\Delta\Pi,b_0,b_1,x)\in \{A,B\}^\Z\times\Z
\times [0,1]^2\times\{1,2\}\colon\,|\Delta\Pi|\leq M,\, k_{\Delta\Pi,\chi}= 0\big\}.
\end{align}

%%%%%%%%%%%%%%%%%%%%%%%%%%%%%%%%%%%%%%%%%%%%%%%%%%%%%%%%%%%%%%%%%%%%
\begin{figure}[htbp]
\begin{center}
\includegraphics[width=.42\textwidth]{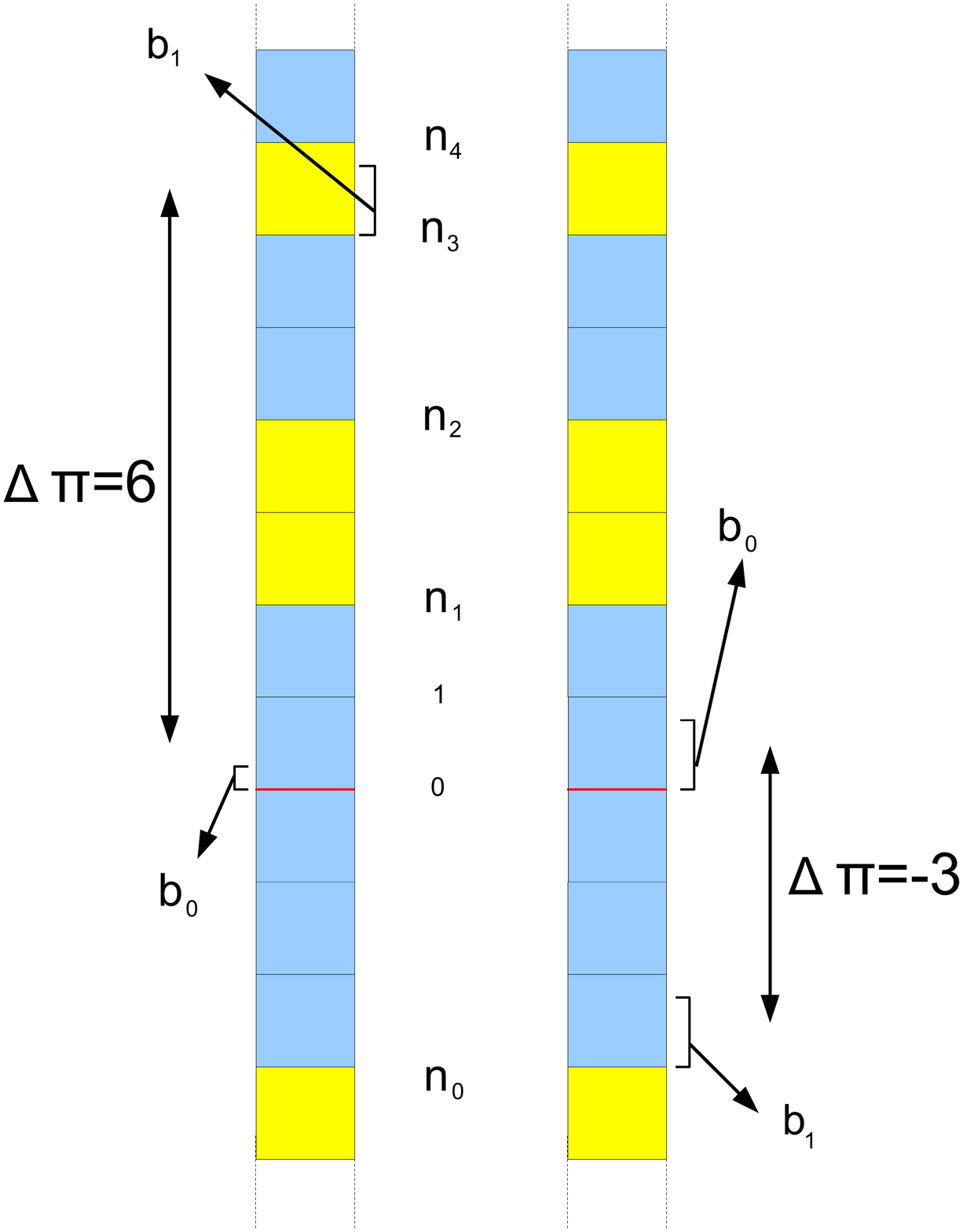}
\end{center}
\vspace{-1cm}
\caption{Labelling of coarse-grained paths and columns. On the left the type of the 
column is in $\cV_{\AB,M}$, on the right it is in $\cV_{\nAB,M}$ (with $M\geq 6$).}
\label{fig6}
\end{figure}
%%%%%%%%%%%%%%%%%%%%%%%%%%%%%%%%%%%%%%%%%%%%%%%%%%%%%%%%%%%%%%%%%%%%%%

\medskip\noindent 
{\bf Time spent in columns.} 
We pick $L,M\in \N$, $\Theta=(\chi,\Delta\Pi,b_0,b_1,x)\in \cV_{L,M}$ and we specify the 
total number of steps that a trajectory crossing the column $\cC_{0,L}$ of type 
$\Theta$ is allowed to make. For $\Theta=(\chi,\Delta\Pi,b_0,b_1,1)$, set
\be{deft}
t_{\Theta}= 1+\mathrm{sign} (\Delta \Pi)\,
(\Delta \Pi+ b_1-b_0) \,\ind_{\{\Delta \Pi\neq 0\}} + |b_1-b_0|
\,\ind_{\{\Delta \Pi=0\}},
\ee
so that a trajectory $\pi$ crossing a column of width $L$ from $(0,b_0 L)$ to 
$(L,(\Delta \Pi+b_1)L)$ makes a total of $uL$ steps with $u\in t_{\Theta}
+\tfrac{2\N}{L}$.
For $\Theta=(\chi,\Delta\Pi,b_0,b_1,2)$ in turn, recall \eqref{efn} and let
\be{deftt}
t_{\Theta}= 1+\mathrm{min}\{2 n_1-b_0-b_1-\Delta \Pi, 2 |n_0| +b_0+b_1+\Delta \Pi\},
\ee
so that  a trajectory $\pi$ crossing a column of width $L$ and type $\Theta\in 
\cV_{\nAB,L,M}$ from $(0,b_0 L)$ to $(L,(\Delta \Pi+b_1)L)$ and reaching an $AB$-interface 
makes a total of $uL$ steps with $u\in t_{\Theta}+\tfrac{2\N}{L}$.

%%%%%%%%%%%%%%%%%%%%%%%%%%%%%%%%%%%%%%%%%%%%%%%%%%%%%%%%%%%%%%%%%%%%%%%
\begin{figure}[htbp]
\begin{center}
\includegraphics[width=.3\textwidth]{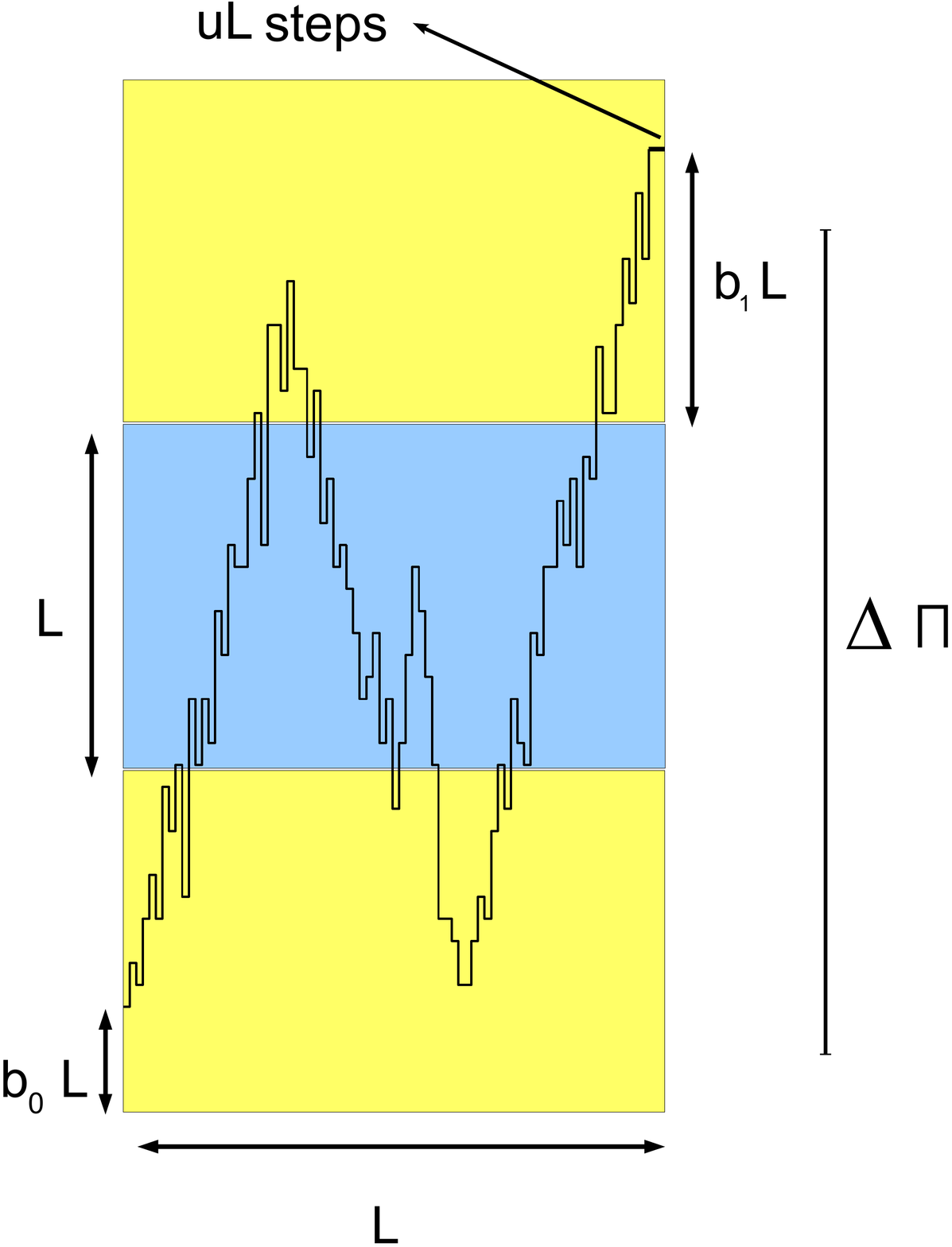}
\end{center}
\vspace{-1cm}
\caption{Example of a $uL$-step path inside a column of type $(\chi,\Delta\Pi,b_0,b_1,1)
\in \cV_{\AB,L}$ with disorder $\chi=(\dots,\chi(0),\chi(1),\chi(2),\dots)=(\dots,A,B,
A,\dots)$, vertical displacement $\Delta\Pi=2$, entrance height $b_0$ and exit height 
$b_1$.}
\label{fig5}
\end{figure}
%%%%%%%%%%%%%%%%%%%%%%%%%%%%%%%%%%%%%%%%%%%%%%%%%%%%%%%%%%%%%%%%%%%%%%%%%%%%

%%%%%%%%%%%%%%%%%%%%%%%%%%%%%%%%%%%%%%%%%%%%%%%%%%%%%%%%%%%%%%%%%%%%%%%
\begin{figure}[htbp]
\begin{center}
\includegraphics[width=.3\textwidth]{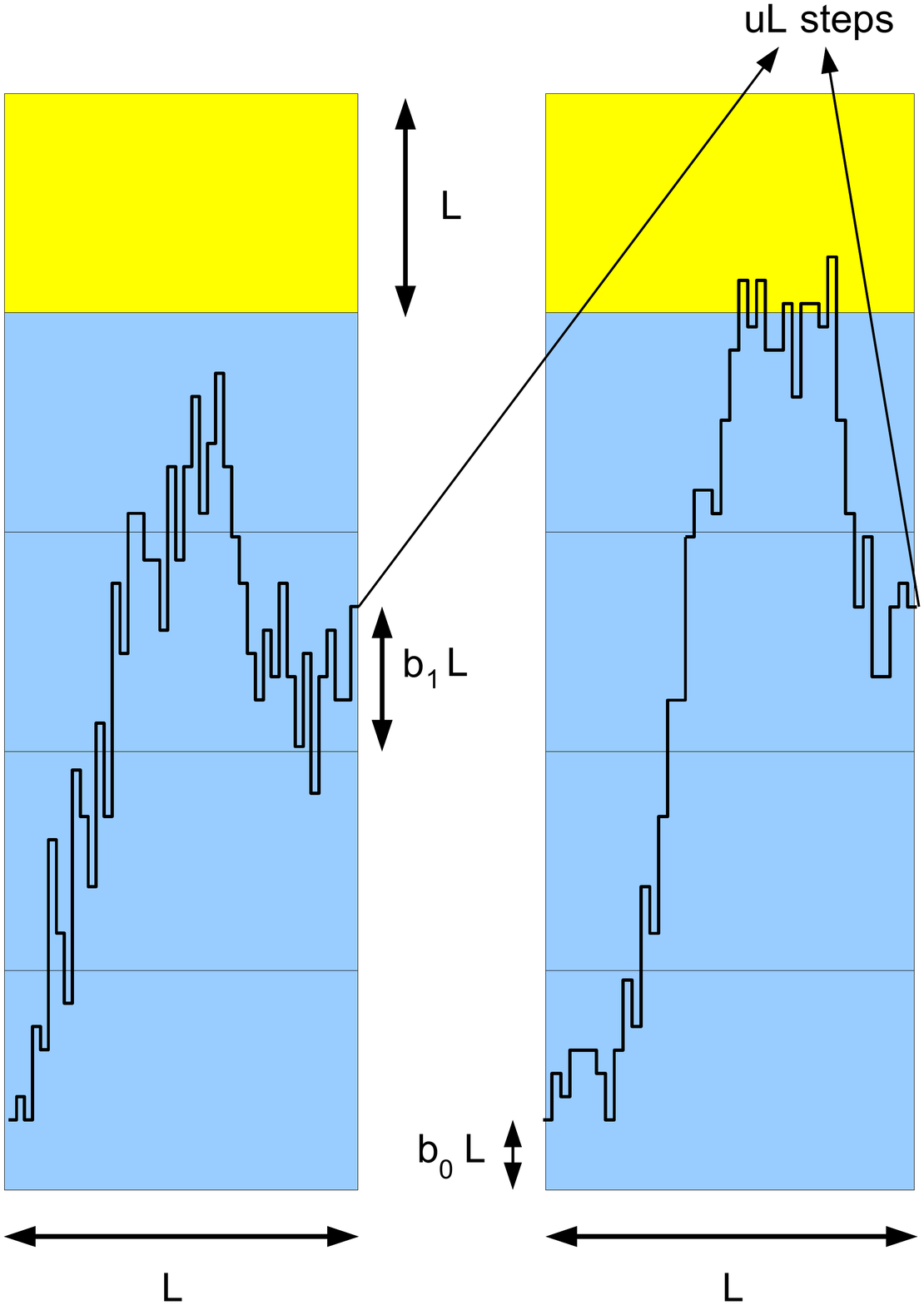}
\end{center}
\vspace{-1cm}
\caption{Two examples of a $uL$-step path inside a column of type $(\chi,\Delta\Pi,b_0,b_1,1)
\in \cV_{\nAB,L}$ (left picture) and $(\chi,\Delta\Pi,b_0,b_1,2)\in \cV_{\nAB,L}$ 
(right picture) with disorder $\chi=(\dots,\chi(0),\chi(1),\chi(2),\chi(3),\chi(4),
\dots)=(\dots,B,B,B,B,A,\dots)$, vertical displacement $\Delta\Pi=2$, entrance
height $b_0$ and exit height $b_1$.}
\label{block2}
\end{figure}
%%%%%%%%%%%%%%%%%%%%%%%%%%%%%%%%%%%%%%%%%%%%%%%%%%%%%%%%%%%%%%%%%%%%%%%%%%%%

At this stage, we can fully determine the set $\cW_{\Theta,u,L}$ consisting of the 
$uL$-step trajectories $\pi$ that are considered in a column of width $L$ and type 
$\Theta$. To that end, for $\Theta\in \cV_{\AB,L,M}$ we map the trajectories $\pi\in
\cW_{L}(u,\Delta\Pi+b_1-b_0)$ onto $\cC_{0,L}$ such that $\pi$ enters $\cC_{0,L}$ 
at $(0,b_0 L)$ and exits $\cC_{0,L}$ at $(L,(\Delta\Pi+b_1)L)$ (see Fig.~\ref{fig5}), 
and for $\Theta\in\cV_{\nAB,L,M}$ we remove, dependencing on $x\in\{1,2\}$, those 
trajectories that reach or do not reach an $AB$-interface in the column (see 
Fig.~\ref{block2}). Thus, for  $\Theta\in \cV_{\AB,L,M}$ and $u\in t_{\Theta}+
\tfrac{2\N}{L}$, we let
\be{ww}
\cW_{\Theta,u,L}=\big\{\pi=(0,b_0 L)+\widetilde\pi\colon\,
\widetilde\pi\in \cW_{L}(u,\Delta\Pi+b_1-b_0) \big\},
\ee
and, for $\Theta\in \cV_{\nAB,L,M}$ and $u\in t_{\Theta}+\tfrac{2\N}{L}$,
\begin{align}
\label{ww2}
\nonumber \cW_{\Theta,u,L}
&=\big\{\pi\in(0,b_0 L)+ \cW_{L}(u,\Delta\Pi+b_1-b_0)\colon\,
\text{$\pi$ reaches no $AB$-interface} \big\}\ \text{if}\ x_\Theta=1,\\
\cW_{\Theta,u,L}&=\big\{\pi\in (0,b_0 L)+\cW_{L}(u,\Delta\Pi+b_1-b_0)\colon\, 
\text{$\pi$ reaches an $AB$-interface} \big\}\ \text{if}\  x_\Theta=2,
\end{align}
with $x_\Theta$ the last coordinate of $\Theta \in \cV_M$. Next, we set
\begin{align}
\label{setp}
\nonumber \cV_{L,M}^* &= \Big\{(\Theta,u)\in \cV_{L,M} \times [0,\infty)\colon\, 
u\in t_{\Theta}+\tfrac{2\N}{L}\Big\},\\
\nonumber \cV_M^* 
&=\big\{(\Theta,u)\in \cV_M \times \mathbb{Q}_{[1,\infty)}\colon\,
u\geq  t_{\Theta}\big\},\\
\overline \cV_M^{*}
&=\big\{(\Theta,u)\in \overline\cV_M\times [1,\infty) \colon\,
u\geq  t_{\Theta}\big\},
\end{align}
which we partition into $\cV^{*}_{\AB,L,M}\cup\cV^{*}_{\nAB,L,M}$, $\cV^*_{\AB,M}\cup
\cV^*_{\nAB,M}$ and $\overline\cV^{*}_{\AB,M}\cup\overline\cV^{*}_{\nAB,M}$. Note that 
for every $(\Theta,u)\in\cV^*_M$ there are infinitely many $L\in \N$ such that 
$(\Theta,u)\in \cV_{L,M}^*$, because $(\Theta,u)\in \cV_{qL,M}^*$ for all $q\in \N$ 
as soon as $(\Theta,u)\in \cV_{L,M}^*$.

\medskip\noindent 
{\bf Restriction on the number of steps per column.} 
In what follows, we set
\be{defhe}
\EIGH = \{(M,m)\in\N\times \N\colon\,m\geq M+2\},
\ee
and, for $(M,m)\in \EIGH$, we consider the situation where the number of steps 
$uL$ made by a trajectory $\pi$ in a column of width $L\in \N$ is bounded by $m L$. 
Thus, we restrict the set $\cV_{L,M}$ to the subset $\cV_{L,M}^{\,m}$ containing only 
those types of columns $\Theta$ that can be crossed in less than $m L$ steps, i.e., 
\be{}
\cV_{L,M}^{\,m}=\{\Theta\in \cV_{L,M}\colon\, t_\Theta\leq m\}.
\ee
Note that the latter restriction only conconcerns those $\Theta$ satisfying $x_\Theta=2$. 
When $x_\Theta=1$ a quick look at \eqref{deft} suffices to state that $t_\Theta\leq M+2
\leq m$. Thus, we set $\cV_{L,M}^{\,m}=\cV_{\AB,L,M}^{\,m}\cup\cV_{\nAB,L,M}^{\,m}$ with 
$\cV_{\AB,L,M}^{\,m}=\cV_{\AB,L,M}$ and with
\begin{align}
\label{set2alt1}
\nonumber \cV_{\nAB,L,M}^{\,m}
= \Big\{\Theta\in \{A,B\}^{\Z}\times \Z \times
\big\{\tfrac{1}{L},\tfrac{2}{L},\dots,1\big\}^2&\times \{1,2\}\colon\\
& |\Delta\Pi| \leq M,\  k_{\Theta}= 0\ \ \text{and}\  t_\Theta\leq m\Big\}.
\end{align}
The sets $\cV_{M}^{\,m}=\cV_{\AB,M}^{\,m}\cup\cV_{\nAB,M}^{\,m}$ and $\overline\cV_{M}^{\,m}
=\overline\cV_{\AB,M}^{\,m}\cup\overline\cV_{\nAB,M}^{\,m}$ are obtained by mimicking 
(\ref{sset2}--\ref{sset3}). In the same spirit, we restrict $\cV_{L,M}^{*}$ to 
\be{setpalt}
\cV_{L,M}^{*,\,m}=\{(\Theta,u)\in \cV_{L,M}^{*}\colon\, \Theta\in \cV_{L,M}^{\,m}, u\leq m\}
\ee 
and $\cV_{L,M}^*=\cV_{\AB,L,M}^*\cup \cV_{\nAB,L,M}^*$ with
\be{setpaltalt}
\begin{aligned}
\cV_{\AB,L,M}^{*,\,m} &= \Big\{(\Theta,u)&\in \cV_{\AB,L,M}^{\,m} \times [1,m]\colon\, 
u\in t_{\Theta}+\tfrac{2\N}{L}\Big\},\\
\cV_{\nAB,L,M}^{*\,m} &= \Big\{(\Theta,u)&\in \cV_{\nAB,L,M}^{\,m}\times [1,m]\colon\, 
u\in t_{\Theta}+\tfrac{2\N}{L} \Big\}.
\end{aligned}
\ee
We set also  $\cV_M^{*,\,m}=\cV^{*,\,m}_{\AB,M}\cup \cV^{*,\,m}_{\nAB,M}$ with 
$\cV_{\AB,M}^{*,\,m}=\cup_{L\in \N}\cV^{*,\,m}_{\AB, L,M}$ and $\cV^{*,\,m}_{\nAB,M}
=\cup_{L\in \N} \cV^{*,\,m}_{\nAB,L,M}$, and rewrite these as
\begin{align}
\label{set2alt2}
\nonumber \cV^{*,\,m}_{\AB,M} 
=\big\{(\Theta,u)&\in \cV_{\AB,M}^{\,m} \times \mathbb{Q}_{[1,m]}\colon\,
u\geq  t_{\Theta}\big\},\\
\cV_{\nAB,M}^{*,\,m}= \big\{(\Theta,u)&\in \cV_{\nAB,M}^{\,m} \times  \mathbb{Q}_{[1,m]}\colon\, 
u\geq t_{\Theta}\big\}.
\end{align}
We further set  $\overline{\cV}^{\,*}_M = \overline{\cV}^{\,*,\,m}_{\AB,M}\cup
\overline{\cV}^{\,*,\,m}_{\nAB,M}$ with 
\be{set2alt3}
\begin{aligned}
\overline\cV^{\,*,\,m}_{\AB,M} 
&=\big\{(\Theta,u)\in \overline\cV_{\AB,M}^{\,m} \times [1,m]\colon\,
u\geq  t_{\Theta}\big\},\\
\overline\cV_{\nAB,M}^{\,*,\,m} 
&= \Big\{(\Theta,u) \in \overline\cV_{\nAB,M}^{\,m} \times [1,m]\colon\, 
u\geq  t_{\Theta}\Big\}.
\end{aligned}
\ee

\medskip\noindent 
{\bf Existence and uniform convergence of free energy per column.} 
Recall \eqref{ww}, \eqref{ww2} and, for $L\in \N$, $\omega\in \{A,B\}^\N$ and $(\Theta,u)
\in\cV^{\,*}_{L,M}$, we associate with each $\pi \in \cW_{\Theta,u,L}$ the energy
\be{Hamiltonian1}
H_{uL,L}^{\omega,\chi}(\pi)
=  \sum_{i=1}^{uL} \big(\beta\, 1\left\{\omega_i=B\right\}
-\alpha\, 1\left\{\omega_i=A\right\}\big)\,
1\Big\{\chi^{L}_{(\pi_{i-1},\pi_i)}=B\Big\},
\ee
where $\chi^{L}_{(\pi_{i-1},\pi_i)}$ indicates the label of the block containing 
$(\pi_{i-1},\pi_i)$ in a column with disorder $\chi$ of width $L$. (Recall that
the disorder in the block is part of the type of the block.) The latter allows us 
to define the quenched free energy per monomer in a column of type $\Theta$ and size 
$L$ as
\be{partfunc2}
\psi^{\omega}_L(\Theta,u)
= \frac{1}{u L} \log Z^{\omega}_L(\Theta,u) 
\quad \text{with} \quad Z^{\omega}_L(\Theta,u)
=\sum_{\pi \in \cW_{\Theta,u,L}} e^{\,H_{uL,L}^{\omega,\chi}(\pi)}.
\ee
Abbreviate $\psi_L(\Theta,u)=\E[\psi^\omega_L(\Theta,u)]$, and note that for $M\in \N$,
$m\geq M+2$ and $(\Theta,u)\in \cV_{L,M}^{\,*,\,m}$ all $\pi\in \cW_{\Theta,u,L}$ 
necessarily remain in the blocks $\Lambda_L(0,i)$ with $i\in \{-m+1,\dots,m-1\}$. 
Consequently, the dependence on $\chi$ of $\psi^{\omega}_L(\Theta,u)$ is restricted 
to those coordinates of $\chi$ indexed by $\{-m+1,\dots,m-1\}$. The following proposition 
will be proven in Section \ref{proofprop1}.

\begin{proposition}
\label{convfree1}
For every $M\in\N$ and $(\Theta,u)\in\cV_M^*$ there exists a $\psi(\Theta,u)\in\R$ such that 
\be{conve}
\lim_{ {L\to \infty} \atop {(\Theta,u)\in\cV_{L,M}^*} } \psi^{\omega}_L(\Theta,u)
= \psi(\Theta,u) = \psi(\Theta,u;\alpha,\beta) \quad \omega-a.s.
\ee
Moreover, for every $(M,m)\in \EIGH$ the convergence is uniform in $(\Theta,u)\in
\cV^{*,\,m}_M$.
\ep

\medskip\noindent
{\bf Uniform bound on the free energies.}
Pick $(\alpha,\beta)\in \CONE$, $n\in \N$, $\omega\in \{A,B\}^\N$, $\Omega\in
\{A,B\}^{\N_0\times\Z}$, and let $\bar\cW_n$ be any non-empty subset of $\cW_n$ 
(recall \eqref{defw}). Note that the quenched free energies per monomer introduced 
until now are all of the form
\be{unifb}
\psi_n=\tfrac1n\log \sum_{\pi\in\bar\cW_n} e^{\,H_n(\pi)},
\ee
where $H_n(\pi)$ may depend on $\omega$ and $\Omega$ and satisfies $-\alpha n\leq 
H_n(\pi)\leq \alpha n$ for all $\pi\in \bar\cW_n$ (recall that $|\beta|\leq \alpha$ 
in $\CONE$). Since $1\leq|\bar\cW_n|\leq |\cW_n|\leq 3^n$, we have
\be{boundel}
|\psi_n|\leq \log 3+\alpha =^\mathrm{def} C_{\text{uf}}(\alpha).
\ee
The uniformity of this bound in $n$, $\omega$ and $\Omega$ allows us to average over 
$\omega$ and/or $\Omega$ or to let $n\to \infty$.

%%%%%%%%%%%%%%%%%%%%%%%%%%%%

\subsubsection{Variational formulas for the free energy in a single column}
\label{newsec}

We next show how the free energies per column can be expressed in terms of two 
variational formulas involving the path entropy and the single interface free 
energy defined in Section \ref{lininter}. Note that $M\in \N$ is given
until the end of the section.

\medskip\noindent 
{\bf Free energy in columns of class $\AB$.}
Pick $\Theta\in \cV_{\AB,M}$ and put
\begin{align}
\label{defl1}
\nonumber l_1&=1_{\{\Delta \Pi>0\}} 
(n_1-b_0)+1_{\{\Delta \Pi<0\}} (b_0-n_0),\\
\nonumber l_j&=1_{\{\Delta \Pi>0\}} 
(n_j-n_{j-1})+1_{\{\Delta \Pi<0\}} (n_{-j+2}-n_{-j+1}) \quad \text{for} \quad 
j\in \{2,\dots,|k_\Theta|\},\\
l_{|k_\Theta|+1}&=1_{\{\Delta \Pi>0\}} 
(\Delta\Pi+b_1-n_{k_\Theta})+1_{\{\Delta \Pi<0\}} (n_{k_\Theta+1}-\Delta\Pi-b_1),
\end{align}
i.e., $l_1$ is the vertical distance between the entrance point and the first 
interface, $l_{i}$ is the vertical distance between the $i$-th interface and 
the $(i+1)$-th interface, and $l_{|k_\Theta|+1}$ is the vertical distance between 
the last interface and the exit point.

Denote by $(h)$ and $(a)$ the triples $(h_A,h_B,h^\cI)$ and $(a_A,a_B,a^\cI)$. 
For $(l_A,l_B)\in (0,\infty)^2$ and $u\geq l_A+l_B+1$, put
\begin{align}
\label{defnu}
\nonumber
\cL(l_A,l_B; u)=\big\{(h),(a)\in [0,1]^3\times [0,\infty)^3\colon\, 
&h_A+h_B+h^\cI=1,\, a_A+a_B+a^\cI=u \\
& a_A\geq h_A+l_A,\, a_B\geq h_B+l_B,\,  a^\cI\geq h^\cI\big\}.
\end{align}
With the help of \eqref{defl1} and \eqref{defnu} we can now provide a variational 
characterization of the free energy in columns of type $\Theta$ of class $\AB$.
Let $l_A(\chi,\Delta \Pi,b_0,b_1)$ and $l_B(\chi,\Delta \Pi,b_0,b_1)$ correspond 
to the minimal vertical distance the copolymer must cross in blocks of type $A$ 
and $B$, respectively, in a column with disorder $\chi$ when going from $(0,b_0)$ 
to $(1,\Delta\Pi+b_1)$, i.e., 
\begin{align}
\label{bl}
\nonumber
l_A(\chi,\Delta \Pi,b_0,b_1)
&=1_{\{\Delta \Pi>0\}} \sum_{j=1}^{|k_\Theta|+1} l_j 1_{\{\chi(n_{j-1})=A\}} 
+1_{\{\Delta \Pi<0\}} \sum_{j=1}^{|k_\Theta|+1} l_j 1_{\{\chi(n_{-j+1})=A\}}, \\
l_B(\chi,\Delta \Pi,b_0,b_1)
&=1_{\{\Delta \Pi>0\}} \sum_{j=1}^{|k_\Theta|+1} l_j 1_{\{\chi(n_{j-1})=B\}} 
+1_{\{\Delta \Pi<0\}} \sum_{j=1}^{|k_\Theta|+1} l_j 1_{\{\chi(n_{-j+1})=B\}}.
\end{align}
The following proposition will be proven in Section \ref{proofprop1}.

\begin{proposition}
\label{energ}
For $(\Theta,u)\in \cV^*_{\AB,M}$, 
\begin{align}
\label{Bloc of type I}
\nonumber \psi(\Theta,u)
&=\psi_{\AB}(u,l_A,l_B)\\
&=\sup_{(h),(a) \in \cL(l_A,\, l_B;\,u)}
\frac{a_A\, \tilde{\kappa}\big(\tfrac{a_A}{h_A},
\tfrac{l_A}{h_A}\big)+a_B\,\big[\tilde{\kappa}\big(\tfrac{a_B}{h_B},
\tfrac{l_B}{h_B}\big)+\tfrac{\beta-\alpha}{2}\big]
+a^\cI\,\phi^\cI(\tfrac{a^\cI}{h^\cI})}{u}.
\end{align}
\end{proposition}

\medskip\noindent
{\bf Free energy in columns of class $\nAB$.}
Pick $\Theta\in \cV_{\nAB,M}$. In this case, there is no $AB$-interface between 
$b_0$ and $\Delta\Pi+b_1$, which means that $\Delta\Pi< n_1$ if $\Delta\Pi\geq 0$ 
and $\Delta\Pi\geq n_0$ if $\Delta\Pi<0$ ($n_0$ and $n_1$ being defined in 
\eqref{efn}). Let $l_{\nAB}(\Delta \Pi,b_0,b_1)$ be the vertical distance between 
the entrance point $(0,b_0)$ and the exit point $(1,\Delta\Pi+b_1)$, i.e.,
\begin{align}
l_{\nAB}(\Delta \Pi,b_0,b_1) &= 1_{\{\Delta \Pi\geq 0\}} (\Delta \Pi-b_0+b_1)
+ 1_{\{\Delta \Pi<0\}} (|\Delta \Pi|+b_0-b_1)+ 1_{\{\Delta \Pi=0\}} |b_1-b_0|,
\end{align}
and let $l_\AB(\chi,\Delta \Pi,b_0,b_1)$ be the minimal vertical distance a trajectory 
has to cross in a column with disorder $\chi$, starting from $(0,b_0)$, to reach 
the closest $AB$-interface before exiting at $(1,\Delta\Pi+b_1)$, i.e.,
\begin{align}
l_\AB(\chi,\Delta \Pi,b_0,b_1)
&=\mathrm{min}\{2 n_1-b_0-b_1-\Delta \Pi, 2 |n_0| +b_0+b_1+\Delta \Pi\}.
\end{align}
The following proposition will be proved in Section \ref{proofprop1}.

\begin{proposition}
\label{energ1}
For $(\Theta,u)\in \cV^*_{\nAB,M}$ such that $x_\Theta=1$,
\begin{align}
\label{Bloc of type NI,B}
\psi(\Theta,u)
=\tilde{\kappa}(u,l_{\nAB})+\tfrac{\beta-\alpha}{2}\,1_{\{\chi(0)=B\}}.
\end{align}
For  $(\Theta,u)\in \cV^*_{\nAB,M}$ such that $x_\Theta=2$,
\begin{align}
\nonumber \psi(\Theta,u)&=\psi_{\nAB}(u,l_\AB;\,\chi(0))\\
&=\sup_{\stackrel{h^\cI\in [0,1],}{u^\cI\in [h^\cI, u+h^\cI-1-l_\AB]}}
\frac{(u-u^\cI) \big[\tilde{\kappa}\big(\tfrac{u-u^\cI}{1-h^\cI},
\tfrac{l_\AB}{1-h^\cI}\big)+\tfrac{\beta-\alpha}{2}\,
1_{\{\chi(0)=B\}}\big]+u^\cI \phi^\cI(\tfrac{u^\cI}{h^\cI})}{u}.
\end{align}
\end{proposition}

The importance of Propositions~\ref{energ}--\ref{energ1} is that they \emph{express 
the free energy in a single column in terms of the path entropy in a single column 
$\tilde\kappa$ and the free energy along a single linear interface $\phi^\cI$}, 
which were defined in Section~\ref{lininter} and are well understood.

%%%%%%%%%%%%%%%%%%%%%%%%%%%%%%%%%%%%%%%%%%%%%%%%%%%

\subsection{Mesoscopic percolation frequencies}
\label{Percofreq}

In this section, we define a set of probability laws providing the frequencies 
with which each type of column can be crossed by the copolymer. 

\medskip\noindent
{\bf Coarse-grained paths.}
For $x\in \N_0\times\Z$ and $n\in \N$, let $c_{x,n}$ denote the center of the block
$\Lambda_{L_n}(x)$ defined in \eqref{blocks}, i.e., 
\be{defc}
c_{x,n}=x L_n+(\tfrac12,\tfrac12) L_n,
\ee
and abbreviate
\be{defc2}
(\N_0\times\Z)_n=\{c_{x,n}\colon\, x\in \N_0\times\Z\}.
\ee
Let $\widehat{\cW}$ be the set of \emph{coarse-grained paths} on $(\N_0\times\Z)_n$ 
that start at $c_{0,n}$, are self-avoiding and are allowed to jump up, down and 
to the right between neighboring sites of $(\N_0\times\Z)_n$, i.e., the increments 
of $\widehat{\Pi} = (\widehat{\Pi}_j)_{j\in\N_0} \in \widehat{\cW}$ are $(0,L_n), 
(0,-L_n)$ and $(L_n,0)$. (These paths are the coarse-grained counterparts of the 
paths $\pi$ introduced in \eqref{defw}.) For $l\in \N\cup \{\infty\}$, let 
$\widehat{\cW}_l$ be the set of $l$-step coarse-grained paths.

Recall, for $\pi\in \cW_{n}$, the definitions of $N_\pi$ and $(v_j(\pi))_{j\leq N_\pi-1}$ 
given below \eqref{deftau}.  With $\pi$ we associate a coarse-grained path 
$\widehat{\Pi}\in \widehat{\cW}_{N_\pi}$ that describes how $\pi$ moves with 
respect to the blocks. The construction of $\widehat{\Pi}$ is done as follows: 
$\widehat{\Pi}_0=c_{(0,0)}$, $\widehat{\Pi}$ moves vertically until it reaches 
$c_{(0,v_0)}$, moves one step to the right to $c_{(1,v_0)}$, moves vertically 
until it reaches $c_{(1,v_1)}$, moves one step to the right to $c_{(2,v_1)}$, and 
so on. The vertical increment of $\widehat{\Pi}$ in the $j$-th column is 
$\Delta\widehat{\Pi}_j=(v_j-v_{j-1}) L_n$ (see Figs.~\ref{fig6}--\ref{block2}).

%%%%%%%%%%%%%%%%%%%%%%%%%%%%%%%%%%%%%%%%%%%%%%%%%%%%%%%%%%%%%%%%%%
\begin{figure}[htbp]
\begin{center}
\includegraphics[width=.42\textwidth]{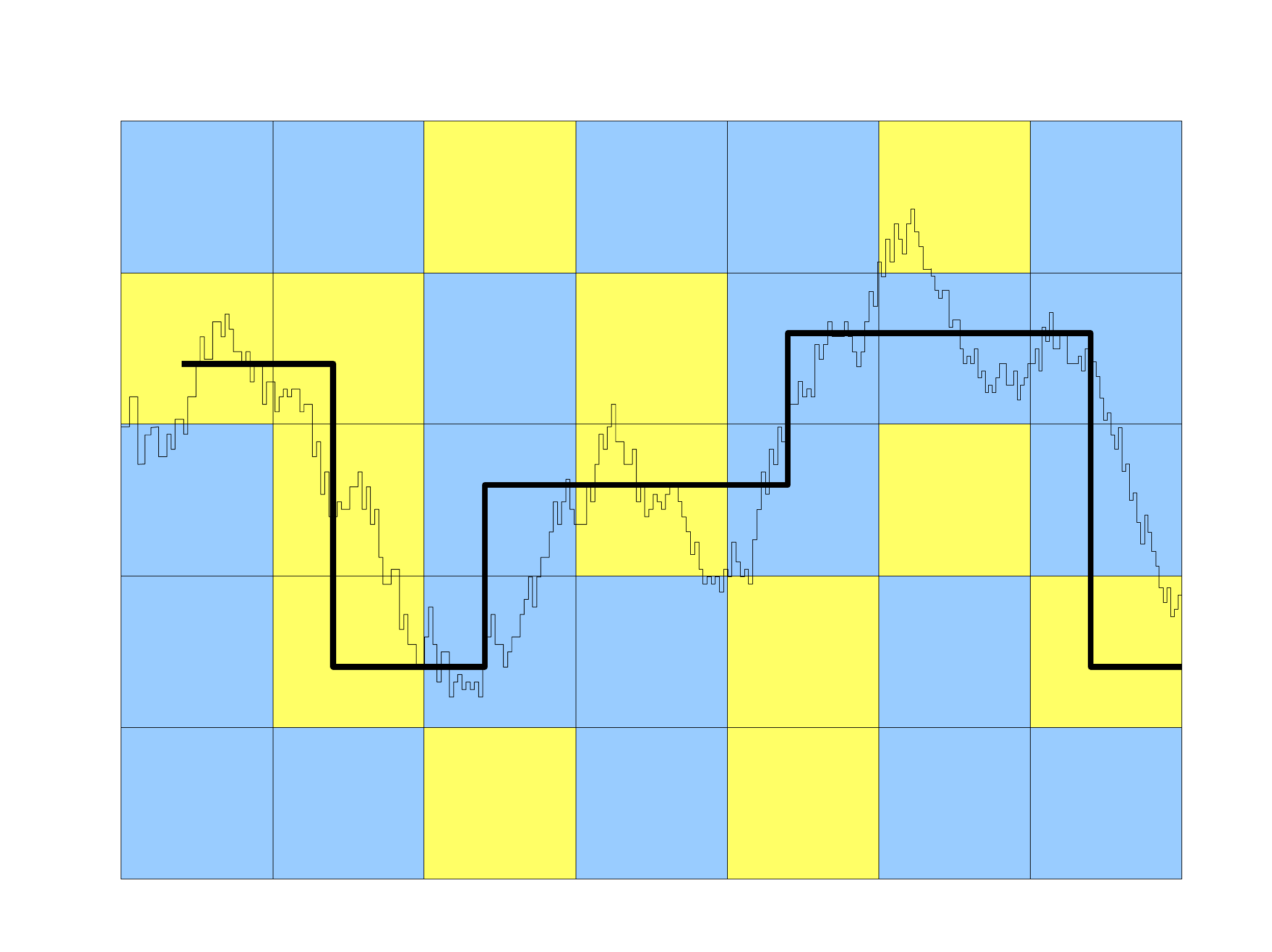}
\end{center}
\vspace{-.5cm}
\caption{Example of a coarse-grained path.}
\label{fig4}
\end{figure}
%%%%%%%%%%%%%%%%%%%%%%%%%%%%%%%%%%%%%%%%%%%%%%%%%%%%%%%%%%%%%%%%%%%

To characterize a path $\pi$, we will often use the sequence of vertical 
increments of its associated coarse-grained path $\widehat{\Pi}$, modified in 
such a way that it does not depend on $L_n$ anymore. To that end, with every 
$\pi\in \cW_n$ we associate $\Pi=(\Pi_k)_{k=0}^{N_\pi-1}$ such that 
$\Pi_0=0$ and,
\begin{align}
\label{defpi}
\Pi_k=\sum_{j=0}^{k-1} \Delta\Pi_j\quad \text{with} \quad  
\Delta\Pi_j=\frac{1}{L_n} \Delta \widehat{\Pi}_j, \qquad
j = 0,\dots,N_\pi-1.
\end{align}
Pick $M\in \N$ and note that $\pi\in\cW_{n,M}$ if and only if $|\Delta \Pi_j|\leq M$ for 
all $j\in \{0,\dots, N_\pi-1\}$.

\medskip\noindent
{\bf Percolation frequencies along coarse-grained paths.}
Given $M\in \N$, we denote by $\cM_1(\overline\cV_M)$ 
the set of probability measures on $\overline\cV_M$. Pick $\Omega \in \{A,B\}^{\N_0\times\Z}$, $\Pi\in \Z^{\N_0}$ such that 
$\Pi_0=0$ and $|\Delta\Pi_i|\leq M$ for all $i\geq 0$ and $b=(b_j)_{j\in \N_0}\in 
(\mathbb{Q}_{(0,1]})^{\N_0}$. Set $\Theta_{\text{traj}}=(\Xi_j)_{j\in \N_0}$ with
\be{defxii}
\Xi_j=\big(\Delta \Pi_j,b_j,b_{j+1}\big), \qquad j\in \N_0,
\ee 
let
\be{defX}
\cX_{\,\Pi,\Omega}=\big\{x\in \{1,2\}^{\N_0}\colon\,
(\Omega(i,\Pi_i+\cdot),\Xi_i,x_i)\in  \cV_M \,\,\, \forall\,i\in \N_0\big\},
\ee
and for $x\in \cX_{\,\Pi,\Omega}$ set
\be{defth}
\Theta_j=\big(\Omega(j,\Pi_j+\cdot),
\Delta \Pi_j,b_j,b_{j+1},x_j\big), \qquad j\in \N_0.
\ee
With the help of \eqref{defth}, we can define the empirical distribution
\begin{equation}
\label{defrho}
\rho_N(\Omega,\Pi,b,x)(\Theta) = \frac{1}{N} \sum_{j=0}^{N-1} 
1_{\{\Theta_j=\Theta\}},
\quad N\in\N,\,\Theta\in \overline\cV_M,
\end{equation}

\begin{definition}
\label{RMNdef}
For $\Omega \in \{A,B\}^{\N_0\times\Z}$ and $M\in \N$, let 
\be{RMNdefalt}
\begin{aligned}
\cR^\Omega_{M,N} &= \big\{\rho_N(\Omega,\Pi,b,x)\ \text{with}\  
b=(b_j)_{j\in \N_0} \in (\mathbb{Q}_{(0,1]})^{\N_0},\\
&\qquad \Pi=(\Pi_j)_{j\in\N_0} \in \{0\}\times \Z^{\N} \colon\,
|\Delta\Pi_j|\leq M\,\ \ \forall\,j\in\N_0,\\
&\qquad x=(x_j)_{j\in \N_0} \in \{1,2\}^{N_0}\colon\, \big(\Omega(j,\Pi_j+\cdot),
\Delta \Pi_j,b_j,b_{j+1},x_j\big)\in \cV_M\big\}
\end{aligned} 
\ee
and
\be{RMdef}
\cR^\Omega_M = \mathrm{closure}\Big(\cap_{N'\in\N} \cup_{N \geq N'}\, 
\cR^\Omega_{M,N}\Big),
\ee
both of which are subsets of $\cM_1(\overline\cV_M)$.
\end{definition}

\begin{proposition}
\label{properc}
For every $p \in (0,1)$ and $M\in \N$ there exists a closed set $\cR_{p,M}\subsetneq
\cM_1(\overline\cV_M)$ such that 
\be{mesoperc}
\cR_{M}^\Omega=\cR_{p,M} \text{ for } \P\text{-a.e.}\,\Omega.
\ee
\end{proposition}

\begin{proof}
Note that, for every $\Omega\in \{A,B\}^{\N_0\times\Z}$, the set $\cR_{M}^\Omega$ 
does not change when finitely many variables in $\Omega$ are changed. Therefore 
$\cR_{M}^\Omega$ is measurable with respect to the tail $\sigma$-algebra of 
$\Omega$. Since $\Omega$ is an i.i.d.\ random field, the claim follows from 
Kolmogorov's zero-one law. Because of the constraint on the vertical displacement, 
$\cR_{p,M}$ does not coincide with $\cM_1(\overline\cV_M)$.
\end{proof}

%%%%%%%%%%%%%%%%%%% SECTION 3 %%%%%%%%%%%%%%%%%%%%%%%%%%%%%%%%%%%%%%%%%%%%%%%%%%%%%

\section{Proof of Propositions~\ref{convfree1}--\ref{energ1}}
\label{proofprop1}

In this section we prove Propositions~\ref{convfree1} and \ref{energ}--\ref{energ1}, 
which were stated in Sections~\ref{lininter}--\ref{Percofreq} and contain the 
precise definition of the key ingredients of the variational formula in 
Theorem~\ref{varformula}. In Section~\ref{proofofgene} we will use these 
propositions to prove Theorem~\ref{varformula}. 

In Section \ref{colcross} we associate with each trajectory $\pi$ in a column 
a sequence recording the indices of the $AB$-interfaces successively visited 
by $\pi$. The latter allows us to state a key proposition, Proposition~\ref{convunifr} 
below, from which Propositions~\ref{convfree1} and \ref{energ}--\ref{energ1} are 
straightforward consequences. In Section~\ref{convunifra} we give an outline of 
the proof of Proposition~\ref{convunifr}, in Sections~\ref{s1}--\ref{s3} we 
provide the details.

%%%%%%%%%%%%%%%%%%%%%%%%%%%%%

\subsection{Column crossing characteristic}
\label{colcross}

%%%%%%%%%%%%%%%%%%%%%%%%%%%%%%%%%%%%%%%%%%%%%%%

\subsubsection{The order of the visits to the interfaces}

Pick $(M,m)\in \EIGH$. To prove  Proposition~\ref{convfree1}, instead of considering 
$(\Theta,u)\in\cV^{*,\,m}_M$, we will restrict to $(\Theta,u)\in\cV^{*,\,m}_{\text{int},M}$. 
Our proof can be easily extended to $(\Theta,u)\in \cV^{*,\,m}_{\text{nint},M}$.

Pick $(\Theta,u)\in \cV^{*,\,m}_{\AB,M}$, recall \eqref{efn} and set $\cJ_{\Theta,u}
=\{\cN_{\Theta,u}^{\downarrow},\dots,\cN_{\Theta,u}^{\uparrow}\},$ with
\begin{align}
\label{defntheta}
\cN^{\uparrow}_{\Theta,u}
= &\max\{i\geq 1\colon n_i\leq u\}\quad \text{and}
\quad \cN^{\uparrow}_{\Theta,u}=0\quad  \text{if}\quad  n_1>u.  \\
\nonumber\cN^{\downarrow}_{\Theta,u}
=&\min \{i\leq 0\colon |n_i|\leq u\} \quad \text{and}
\quad \cN^{\downarrow}_{\Theta,u}=1\quad  \text{if}\quad  |n_0|>u.
\end{align}
Next pick $L\in \N$ so that $(\Theta,u)\in \cV^{*}_{\AB,L,M}$ and recall that for 
$j\in \cJ_{\Theta,u}$ the $j$-th interface of the $\Theta$-column is $\cI_j=\{0,\dots,L\}
\times \{n_j L\}$. Note also that $\pi \in \cW_{\Theta,u,L}$ makes $uL$ steps inside 
the column and therefore can not reach the $AB$-interfaces labelled outside 
$\{\cN^{\downarrow}_{\Theta,u},\dots,\cN^{\uparrow}_{\Theta,u}\}$.

First, we associate with each trajectory $\pi\in \cW_{\Theta,u,L}$ the sequence 
$J(\pi)$ that records the indices of the interfaces that are successively visited by 
$\pi$. Next, we pick $\pi\in \cW_{\Theta,u,L}$, and define $\tau_1, J_1$ as
\begin{equation}
\tau_1=\inf\{i\in\N\colon\, \exists j\in \cJ_{\Theta,u} 
\colon\,\pi_i\in \cI_j \}, \qquad \pi_{\tau_1}\in \cI_{J_1},
\end{equation} 
so that $J_1=0$ (respectively, $J_1=1$) if the first interface reached by $\pi$ 
is $\cI_0$ (respectively, $\cI_1$). For $i\in\N\setminus\{1\}$, we define $\tau_i,J_i$ 
as 
\begin{equation}
\tau_i=\inf\big\{t>\tau_{i-1}\colon\,\exists 
j\in \cJ_{\Theta,u}\setminus\{ J_{i-1}\}, \pi_i\in \cI_j\big\}, 
\qquad \pi_{\tau_i}\in \cI_{J_i},
\end{equation} 
so that the increments of $J(\pi)$ are restricted to $-1$ or $1$. The length of 
$J(\pi)$ is denoted by $m(\pi)$ and corresponds to the number of jumps made by 
$\pi$ between neighboring interfaces before time $uL$, i.e., 
$J(\pi)=(J_i)_{i=1}^{m(\pi)}$ with 
\be{add6}
m(\pi)=\max\{i\in\N\colon\,\tau_i\leq uL\}.
\ee
Note that $(\Theta,u)\in \cV_{\AB,M}^{*,\,m}$ necessarily implies $k_{\Theta}\leq 
m(\pi)\leq u\leq m$. Set
\be{add7}
\cS_r=\{j=(j_i)_{i=1}^r\in \Z^\N\colon\, j_1\in \{0,1\},\,  
j_{i+1}-j_i\in \{-1,1\}\,\,\forall\,1\leq i\leq r-1\}, \qquad r\in \N, 
\ee
and, for $\Theta\in \cV$, $r\in \{1,\dots, m\}$ and $j\in \cS_r$, define
\begin{align}
\label{defl}
\nonumber l_1&=1_{\{j_1=1\}} (n_{1}-b_0)+1_{\{j_1= 0\}} (b_0-n_{0}),\\
\nonumber l_i&=|n_{j_i}-n_{j_{i-1}}| \text{ for } i\in \{2,\dots,r\},\\
l_{r+1}&=1_{\{j_r=k_\Theta+1\}} (n_{k_\Theta+1}-\Delta \Pi-b_1)
+1_{\{j_r= k_\Theta\}} (\Delta \Pi +b_1-n_{k_\Theta}), 
\end{align}
so that $(l_i)_{i\in \{1,\dots,r+1\}}$ depends on $\Theta$ and $j$. Set
\begin{align}
\label{setofA}
\cA_{\Theta,j}
&=\{i\in \{1,\dots,r+1\}\colon\,\text{$A$ between}\,\cI_{j_{i-1}}\,
\text{and}\,\cI_{j_{i}}\},\\
\nonumber \cB_{\Theta,j}
&=\{i\in \{1,\dots,r+1\}\colon\,\text{$B$ between}\,\cI_{j_{i-1}}\, 
\text{and} \ \cI_{j_{i}}\},
\end{align}
and set $l_{\Theta,j}=(l_{A,\Theta,j},l_{B,\Theta,j})$ with
\begin{align}
\label{bl1}
l_{A,\Theta,j}
&={\textstyle \sum_{i\in \cA_{\Theta,j}}} l_i,\,\,l_{B,\Theta,j}
={\textstyle \sum_{i\in \cB_{\Theta,j}}} l_i.
\end{align}
For $L\in \N$ and  $(\Theta,u)\in \cV^{*,\,m}_{\AB,L,M}$, we denote by $\cS_{\Theta,u,L}$ 
the set $\{J(\pi), \pi\in \cW_{\Theta,u,L}\}$. It is not difficult to see that a sequence 
$j\in \cS_r$ belongs to $\cS_{\Theta,u,L}$ if and only if it satisfies the two following 
conditions. First, $j_r \in \{k_\Theta,k_\Theta +1\}$, since $j_r$ is the index of the 
interface last visited before the $\Theta$-column is exited. Second, $u\geq 1+l_{A,\Theta,j}
+l_{B,\Theta,j}$ because the number of steps taken by a trajectory $\pi\in \cW_{\Theta,u,L}$
satisfying $J(\pi)=j$ must be large enough to ensures that all interfaces $\cI_{j_s}$, 
$s\in \{1,\dots,r\}$, can be visited by $\pi$ before time $uL$. Consequently, 
$\cS_{\Theta,u,L}$ does not depend on $L$ and can be written as $\cS_{\Theta,u}
= \cup_{r=1}^{m}\cS_{\Theta,u,r}$, where 
\be{redefS}
\cS_{\Theta,u,r}=\{j\in \cS_r\colon j_r\in \{k_\Theta,k_\Theta+1\}, 
u \geq 1+l_{A,\Theta,j}+l_{B,\Theta,j}\}.
\ee
Thus, we partition $\cW_{\Theta,u,L}$ according to the value taken by $J(\pi)$, i.e.,
\be{zto}
\cW_{\Theta,u,L}=\bigcup_{r=1}^{m} \ \bigcup_{j\in \cS_{\Theta,u,r}} \ \cW_{\Theta,u,L,j}, 
\ee
where $\cW_{\Theta,u,L,j}$ contains those trajectories $\pi\in \cW_{\Theta,u,L}$  
for which $J(\pi)=j$. 

Next, for $j \in \cS_{\Theta,u}$, we define (recall \eqref{Hamiltonian1})
\be{partfunc3}
\psi^{\omega}_L(\Theta,u,j)
= \frac{1}{u L} \log Z^{\omega}_L(\Theta,u,j), \qquad 
\psi_L(\Theta,u,j)=\E\big[\psi^{\omega}_L(\Theta,u,j)\big],
\ee
with
\be{partfunc3alt}
Z^{\omega}_L(\Theta,u,j)
=\sum_{\pi \in \cW_{\Theta,u,L,j}} e^{H_{uL,L}^{\omega,\chi}(\pi)}.
\ee
For each $L\in \N$ satisfying $(\Theta,u)\in \cV^{*,\,m}_{\AB,L,M}$ and each $j\in 
\cS_{\Theta,u}$, the quantity $l_{A,\Theta,j} L$ (respectively, $l_{B,\Theta,j} L$) 
corresponds to the minimal vertical distance a trajectory $\pi\in \cW_{\Theta,u,L,j}$ 
has to cross in solvent $A$ (respectively, $B$).

%%%%%%%%%%%%%%%%%%%%%%%

\subsubsection{Key proposition}

Recalling \eqref{Bloc of type I} and \eqref{bl1}, we define the free energy associated with 
$\Theta,u,j$ as 
\begin{align}
\label{Bloc of type I1}
\psi(\Theta,u,j)
&=\psi_{\AB}(u,l_{\Theta,j})\\ \nonumber 
&=\sup_{(h),(u) \in \cL(l_{\Theta,j};\, u)}
\frac{u_A\, \tilde{\kappa}\big(\tfrac{u_A}{h_A},
\tfrac{l_{A,\Theta,j}}{h_A}\big)+u_B\,
\big[\tilde{\kappa}\big(\tfrac{u_B}{h_B},
\tfrac{l_{B,\Theta,j}}{h_B}\big)
+\tfrac{\beta-\alpha}{2}\big]+u_I\, \phi(\tfrac{u^I}{h^I})}{u}.
\end{align}
Proposition~\ref{convunifr} below states that $\lim_{L\to\infty} \psi_L(\Theta,u,j) 
=\psi(\Theta,u,j)$ uniformly in $(\Theta,u) \in \cV_{\AB,M}^{*,\,m}$ and $j\in \cS_{\Theta,u}$.  

\begin{proposition}
\label{convunifr}
For every $M,m\in \N$ such that $m\geq M+2$ and every $\gep>0$ there exists an $L_\gep\in \N$ such that 
\be{une int}
\big|\psi_L(\Theta,u,j)-\psi(\Theta,u,j)\big|
\leq \gep \quad \forall\,(\Theta,u)\in \cV^{*,\,m}_{\AB,L,M},\  \,
j\in \cS_{\Theta,u},\ \, L\geq L_\gep.
\ee
\end{proposition}

\medskip\noindent 
{\bf Proof of Propositions~\ref{convfree1} and \ref{energ}--\ref{energ1} subject 
to Proposition~\ref{convunifr}}. Pick $\gep>0$, $L\in \N$ and $(\Theta,u)\in
\cV^{*,\,m}_{\AB,L,M}$. Recall \eqref{bl} and note that $l_A(\Theta)L$ and $l_B(\Theta)L$ 
are the minimal vertical distances the trajectories of $\cW_{\Theta,u,L}$ have to 
cross in blocks of type $A$, respectively, $B$. For simplicity, in what follows 
the $\Theta$-dependence of $l_A$ and $l_B$ will be suppressed. In other words, 
$l_A$ and $l_B$ are the two coordinates of $l_{\Theta,f}$ (recall \eqref{bl1}) 
with $f=(1,2,\dots,|k_{\Theta}|)$ when $\Delta\Pi\geq 0$ and $f=(0,-1,\dots,
-|k_{\Theta}|+1)$ when $\Delta\Pi< 0$, so \eqref{Bloc of type I} and 
\eqref{Bloc of type I1} imply
\be{pisdefalt}
\psi_{\AB}(u,l_A,l_B)= \psi(\Theta,u,f).
\ee
Hence Propositions~\ref{convfree1} and \ref{energ} will be proven once we show 
that $\lim_{L\to\infty} \psi_L(\Theta,u)=\psi(\Theta,u,f)$ uniformly in $(\Theta,u)
\in\cV^{*,\,m}_{\AB,L,M}$. Moreover, a look at \eqref{Bloc of type I1}, \eqref{pisdefalt}
and \eqref{Bloc of type I} allows us to assert that for every $j\in \cS_{\Theta,u}$ 
we have $\psi(\Theta,u,j)\leq \psi(\Theta,u,f)$. The latter is a consequence of 
the fact that $l\mapsto \tilde{\kappa}(u,l)$ decreases on $[0,u-1]$ (see 
Lemma~\ref{l:lemconv2}(ii) in Appendix~\ref{Path entropies}) and that 
\begin{align}
\label{add8}
\nonumber 
l_{A}&=l_{A,\Theta,f}=\min\{l_{A,\Theta,j} \colon\, j\in \cS_{\Theta,u}\},\\
l_{B}&=l_{B,\Theta,f}=\min\{l_{B,\Theta,j} \colon\, j\in \cS_{\Theta,u}\}.
\end{align}

By applying Proposition~\ref{convunifr} we have, for $L\geq L_\gep$, 
\begin{align}
\label{eq:restore}
\nonumber 
\psi_L(\Theta,u,j)
&\leq \psi(\Theta,u,f)+\gep \qquad \forall\,(\Theta,u)\in \cV^{*,\,m}_{\AB,L,M},\ 
\forall\,j\in \cS_{\Theta,u},\\
\psi_L(\Theta,u,f)
&\geq \psi(\Theta,u,f)-\gep \qquad \forall\,(\Theta,u)\in \cV^{*,\,m}_{\AB,L,M}.
\end{align} 
The second inequality in \eqref{eq:restore} allows us to write, for $L\geq L_\gep$, 
\be{inegleft}
\psi(\Theta,u,f)-\gep\leq \psi_L(\Theta,u,f)\leq \psi_L(\Theta,u) \qquad \forall\,
(\Theta,u)\in \cV^{*,\,m}_{\AB,L,M}.
\ee
To obtain the upper bound we introduce
\be{cAdef}
\cA_{L,\gep}=\Big\{\omega\colon\, |\psi^\omega_L(\Theta,u,j)-\psi_L(\Theta,u,j)|
\leq \gep \ \quad \forall\, 
(\Theta,u)\in \cV^{*,\,m}_{\AB,L,M},\, \forall\,j\in \cS_{\Theta,u}\Big\},
\ee
so that 
\begin{align}
\label{refre}
\psi_L(\Theta,u)
&\leq \E\big[ 1_{\cA^c_{L,\gep}}\, \psi^\omega_L(\Theta,u)\big]
+ \E\big[ 1_{\cA_{L,\gep}}\, \psi^\omega_L(\Theta,u)\big]\\
\nonumber 
&\leq C_{\text{uf}}(\alpha)\,\P(\cA_{L,\gep}^c)+\tfrac{1}{uL}
\E\Big[ 1_{\cA_{L,\gep}}\,\log {\textstyle \sum_{j\in  \cS_{\Theta,u}}}\, 
e^{uL (\psi_L(\Theta,u,j)+\gep)}\Big],
\end{align}
where we use \eqref{boundel} to bound the first term in the right-hand side, 
and the definition of $\cA_{L,\gep}$ to bound the second term. Next, with the help 
of the first inequality in \eqref{eq:restore} we can rewrite \eqref{refre} for 
$L\geq L_\gep$ and $(\Theta,u)\in \cV^{*,\,m}_{\AB,L,M}$ in the form
\begin{align}
\label{restora}
\psi_L(\Theta,u)&\leq  C_{\text{uf}}(\alpha)\, \P(\cA_{L,\gep}^c) 
+\tfrac{1}{uL}\log |\cup_{r=1}^{m} \cS_{r}|+ \psi(\Theta,u,f)+2\gep.
\end{align} 
At this stage we want to prove that $\lim_{L\to\infty} \P(\cA^c_{L,\gep})=0$. To 
that end, we use the concentration of measure property in \eqref{concmesut} in
Appendix~\ref{Ann2} with $l=uL$, $\Gamma=\cW_{\Theta,u,L,j}$, $\eta=\gep uL$, 
$\xi_i=-\alpha 1\{\omega_i=A\} +\beta 1\{\omega_i=B\}$ for all $i\in\N$ and 
$T(x,y)=1\{\chi^{L_n}_{(x,y)}=B\}$. We then obtain that there exist $C_{1},C_{2}>0$ 
such that, for all $L\in \N$,  $(\Theta,u)\in\cV^{*,\,m}_{\AB,L,M}$ and 
$j\in\cS_{\Theta,u}$, 
\be{rec222}
\P\big( |\psi^\omega_L(\Theta,u,j)-\psi_L(\Theta,u,j)|> \gep\big)
\leq C_{1} \, e^{-C_{2}\, \gep^2\, uL}.
\ee
The latter inequality, combined with the fact that $|\cV^{*,\,m}_{\AB,L,M}|$ grows 
polynomialy in $L$, allows us to assert that $\lim_{L\to\infty} \P(\cA^c_{L,\gep})=0$. 
Next, we note that $|\cup_{r=1}^{m} \cS_{r}|<\infty$, so that for $L_\gep$ large enough 
we obtain from \eqref{restora} that, for $L\geq L_\gep$,
\be{restora1}
\psi_L(\Theta,u) \leq \psi(\Theta,u,f)+3\gep \qquad \forall\,
(\Theta,u)\in \cV^{*,\,m}_{\AB,L,M}.
\ee 
Now \eqref{inegleft} and \eqref{restora1} are sufficient to complete the 
proof of Propositions~\ref{convfree1}--\ref{energ}. The proof of 
Proposition~\ref{energ1} follows in a similar manner after minor modifications.
\hspace*{\fill}$\square$

%%%%%%%%%%%%%%%%%%%%%%%%%%

\subsection{Structure of the proof of Proposition \ref{convunifr}}
\label{convunifra}

{\bf Intermediate column free energies.}
Let 
\be{add9}
G_M^{\,m}=\big\{(L,\Theta,u,j)\colon\,
(\Theta,u)\in \cV^{*,\,m}_{\AB,L,M},\,  j\in \cS_{\Theta,u}\big\},
\ee 
and define the following order relation. 

\begin{definition}
For $g,\widetilde{g}\colon\,G_M^{\,m}\mapsto \R$, write $g\prec \widetilde{g}$ when  
for every $\gep>0$ there exists an $L_\gep\in \N$ such that
\be{reparco}
g(L,\Theta,u,j)\leq \widetilde{g}(L,\Theta,u,j)
+\gep \qquad \forall\,(L,\Theta,u,j)\in G_M^{\,m}\colon\, L\geq L_\gep.
\ee
\end{definition}

Recall \eqref{partfunc3} and  \eqref{Bloc of type I1}, set
\be{psidefvar}
\psi_1(L,\Theta,u,j) = \psi_L(\Theta,u,j), \qquad
\psi_4(L,\Theta,u,j)=\psi(\Theta,u,j),
\ee 
and note that the proof of Proposition~\ref{convunifr} will be complete 
once we show that $\psi_1\prec \psi_4$ and $\psi_4\prec \psi_1$. In what 
follows, we will focus on $\psi_1\prec\psi_4$. Each step of the proof can 
be adapted to obtain $\psi_4\prec\psi_1$ without additional difficulty. 

In the proof we need to define two intermediate free energies $\psi_2$ 
and $\psi_3$, in addition to $\psi_1$ and $\psi_4$ above. Our proof is divided 
into 3 steps, organized in Sections \ref{s1}--\ref{s3}, and consists of showing 
that $\psi_1\prec \psi_2\prec \psi_3\prec \psi_4$. 

\medskip\noindent 
{\bf Additional notation.} 
Before stating Step 1, we need some further notation. First, we partition 
$\cW_{\Theta,u,L,j}$ according to the total number of steps and the number 
of horizontal steps made by a trajectory along and in between $AB$-interfaces. 
To that end, we assume that $j\in \cS_{\Theta,u,r}$ with $r\in \{1,\dots,m\}$, we 
recall \eqref{defl} and we let 
\begin{align}
\label{setindice}
\nonumber 
\cD_{\Theta,L,j}
&=\big\{(d_i,t_i)_{i=1}^{r+1}\colon\, d_i\in\N\,\,\text{and}\,\,
t_i\in d_i+l_i L+2\N_0\,\,\forall\,1\leq i\leq r+1\big\},\\
\cD_{r}^\cI&=\big\{(d_i^\cI,t_i^\cI)_{i=1}^r\colon\, 
d_i^\cI \in\N\,\,\text{and}\,\,t_i^\cI\in d_i^\cI+2\N_0\,\,\forall\, 1\leq i\leq r\big\},
\end{align}
where $d_i,t_i$ denote the number of horizontal steps and the total number of 
steps made by the trajectory between the $(i-1)$-th and $i$-th interfaces, and
$d_i^\cI,t_i^\cI$ denote the number of horizontal steps and the total number of 
steps made by the trajectory along the $i$-th interface. For $(d,t)\in 
\cD_{\Theta,L,j}$, $(d^\cI,t^\cI)\in \cD_{r}^\cI$ and $1\leq i\leq r$, we set 
$T_0=0$ and
\begin{align}
\label{defU}
\nonumber 
V_{i}
&=\sum_{j=1}^i t_j+\sum_{j=1}^{i-1} t_j^{\cI}, \qquad i = 1,\dots,r,\\
T_{i}
&=\sum_{j=1}^{i} t_j+\sum_{j=1}^{i} t_j^{\cI}, \qquad i = 1,\dots,r,
\end{align}
so that $V_i$, respectively, $T_i$ indicates the number of steps made by the trajectory 
when reaching, respectively, leaving the $i$-th interface. 

Next, we let $\theta\colon\,\R^\N\mapsto\R^\N$ be the left-shift acting on infinite 
sequences of real numbers and, for $u\in \N$ and $\omega\in \{A,B\}^\N$, we put
\be{defB}
H_u^\omega(B)=\sum_{i=1}^u \big[\beta\,
1_{\{\omega_i=B\}}-\alpha\,1_{\{\omega_i=A\}}\big].
\ee
Finally, we recall that 
\begin{equation}
\label{psi1}
\psi_1(L,\Theta,u,j)=\tfrac{1}{ uL}\,\mathbb{E}[\log Z_1^{\omega}(L,\Theta,u,j)],
\end{equation}
where the partition function defined in \eqref{partfunc2} has been renamed $Z_1$ and
can be written in the form
\be{defZ}
Z^{\omega}_1(L,\Theta,u,j)=\sum_{(d,t)\in \cD_{\Theta,L,j}}\, 
\sum_{(d^\cI,t^\cI)\in \cD_{r}^\cI} A_1\,B_1\,C_1,
\ee
where (recall \eqref{setofA} and \eqref{feinf})
\begin{align}
\label{defA}
A_1 &=\prod_{i\in \cA_{\Theta,j}}\, 
e^{t_{i}\, \tilde{\kappa}_{d_i}\big(\tfrac{t_{i}}{d_i},\,
\tfrac{l_{i} L}{d_i}\big)}\,  
\prod_{i\in \cB_{\Theta,j}}\,e^{t_{i}\,
\tilde{\kappa}_{d_i}\big(\tfrac{t_{i}}{d_i},\,
\tfrac{ l_{i} L}{d_i}\big)}\,e^{H^{\theta^{T_{i-1}}(w)}_{t_{i}}(B)},\\
\nonumber 
B _1&=\prod_{i=1}^{r}\,e^{t_{i}^\cI\,\phi^{\theta^{V_{i}}(w)}_{d^{\cI}_{i}}
\big(\tfrac{t^{\cI}_{i}} {d_{i}^\cI}\big)},\\
\nonumber 
C _1&= \ind_{\big\{\sum_{i=1}^{r+1} d_i+\sum_{i=1}^{r} d_i^\cI=L\big\}}\, 
\ind_{\big\{\sum_{i=1}^{r+1} t_i+\sum_{i=1}^{r} t_i^\cI=u L\big\}}.
\end{align}
It is important to note that a simplification has been made in the term $A_1$ in 
\eqref{defA}. Indeed, this term is not $\tilde{\kappa}_{d_i}(\cdot,\cdot)$ 
defined in \eqref{ttrajblock}, since the latter does not take into account 
the vertical restrictions on the path when it moves from one interface to 
the next. However, the fact that two neighboring $AB$-interfaces are necessarily 
separated by a distance at least $L$ allows us to apply Lemma~\ref{conunifalt1} 
in Appendix~\ref{A.3}, which ensures that these vertical restrictions can be
removed at the cost of a negligible error.

To show that $\psi_1\prec\psi_2\prec\psi_3\prec\psi_4$, we fix $(M,m)\in \EIGH$ and $\gep>0$, 
and we show that there exists an $L_\gep\in \N$s such that $\psi_k(L,\Theta,u,j)
\leq \psi_{k+1}(L,\Theta,u,j)+\gep$ for all $(L,\Theta,u,j)\in G_M^{\,m}$ and 
$L\geq L_\gep$. The latter will complete the proof of Proposition~\ref{convunifr}.

%%%%%%%%%%%%%%%%%%%%%%%

\subsection{Step 1}
\label{s1}

In this step, we remove the $\omega$-dependence from $Z_1^{\,\omega}(L,\Theta,u,j)$. 
To that aim,  we put
\be{psiLTdef}
\psi_2(L,\Theta,u,j)=\frac{1}{ uL} \log Z_2(L,\Theta,u,j)
\ee 
with
\be{Ztilde}
Z_2(L,\Theta,u,j) = \sum_{(d,t)\in \cD_{\Theta,L,j}}\, 
\sum_{(d^\cI,t^\cI)\in \cD_{r}^\cI} A_2 \  B_2 \  C_2,
\ee
where
\begin{align}
\label{defA1}
A_2 &= \prod_{i\in \cA_{\Theta,j}}\,e^{t_{i}\, \tilde{\kappa}_{d_i}
\big(\tfrac{t_{i}}{d_i},\,\tfrac{ l_{i} L}{d_i}\big)}
\prod_{i\in \cB_{\Theta,j}}\,e^{t_{i}\,\tilde{\kappa}_{d_i}
\big(\tfrac{t_{i}}{d_i},
\,\tfrac{ l_{i} L}{d_i}\big)}\,e^{\tfrac{\beta-\alpha}{2}\, t_{i}},\\
\nonumber 
B_2 &= \prod_{i=1}^{r}\,e^{t_{i}^\cI\,\phi_{d^\cI_{i}}
\Big(\tfrac{t^\cI_{i}}{d^\cI_{i}}\Big)},\\ 
\nonumber 
C_2 &= C_1.
\end{align}
Next, for $n\in \N$ we define 
\begin{align}
\label{subset1}
\nonumber 
\cA_{\gep,n} &= \Big\{\exists\,0\leq t,s\leq n\colon\, t\geq \gep n, 
\,\big|H_t^{\theta^s(\omega)}(B)-\tfrac{\beta-\alpha}{2} t\big|>\gep t\Big\},\\
\cB_{\gep,n} &= \Big\{\exists\,0\leq t,d,s\leq n\colon\,t\in d+2\N_0,\, 
t\geq \gep n,\,\big|\phi^{\theta^s(w)}_d(\tfrac{t}{d})-\phi_d 
(\tfrac{t}{d})\big|>\gep \Big\}.
\end{align}
By applying Cram\'er's theorem for i.i.d.\ random variables (see e.g.\ den 
Hollander~\cite{dH}, Chapter 1), we obtain that there exist $C_{1}(\gep),C_{2}(\gep)>0$ 
such that
\be{rec}
\P\big(\big|H_t^{\theta^s(w)}(B)-\tfrac{\beta-\alpha}{2} t \big|>\gep t\big)
\leq C_{1}(\gep)\, e^{-C_2(\gep) t}, \qquad t,s\in \N.
\ee
By using the concentration of measure property in \eqref{concmesut} in Appendix~\ref{Ann2} 
with $l=t$, $\Gamma=\cW^\cI_{d}(\tfrac{t}{d})$, $T(x,y)=1\{(x,y)< 0\}$, 
$\eta=\gep t$ and $\xi_i=-\alpha 1\{\omega_i=A\} +\beta 1\{\omega_i=B\}$ for 
all $i\in\N$, we find that there exist $C_{1},C_{2}>0$ such that
\be{rec22}
\P\big(\big|\phi^{\theta^s(w)}_d(\tfrac{t}{d})-\phi_d(\tfrac{t}{d})\big|>\gep \big|\big)
\leq C_{1} \, e^{-C_{2}\, \gep^2 t}, \qquad t,d,s \in\N,\,t\in d+2\N_0. 
\ee
With the help of \eqref{boundel} and \eqref{psi1} we may write, for $(L,\Theta,u,j)
\in G_M^{\,m}$,
\be{boundZ}
\psi_1(L,\Theta,u,j)\leq C_{\text{uf}}(\alpha)\,
\P\big(\cA_{\gep,m L}\cup \cB_{\gep,m L}\big)
+\tfrac{1}{uL}\,\E\big[1_{\{\cA^c_{\gep,m L}\cap\cB^c_{\gep,m L}\}}\,
\log Z_1^\omega(L,\Theta,u,j)\big].
\ee
With the help of \eqref{rec} and \eqref{rec22}, we get that $\P(\cA_{\gep,m L})\to 0$ 
and $\P(\cB_{\gep,m L})\to 0$ as $L\to \infty$. Moreover, from (\eqref{defZ}-\eqref{subset1}) 
it follows that, for $(L,\Theta,u,j)\in G_M^{\,m}$ and $\omega\in\cA^c_{\gep,m L}\cap
\cB^c_{\gep,ML}$, 
\be{boundZalt}
Z^\omega_1(L,\Theta,u,j)\leq Z_2(L,\Theta,u,j)\,e^{\gep u L}.
\ee
The latter completes the proof of $\psi_1\prec\psi_2$. 

%%%%%%%%%%%%%%%%%%%%%%%%%%%

\subsection{Step 2}

In this step, we concatenate the pieces of trajectories that travel in $A$-blocks, 
respectively, $B$-blocks, respectively, along the $AB$-interfaces and replace the 
finite-size entropies and free energies by their infinite-size counterparts. Recall 
the definition of $l_{A,\Theta,j}$ and $l_{B,\Theta,j}$ in \eqref{bl1} and define, for 
$(L,\Theta,u,j)\in G_M^{\,m}$, the sets  
\begin{align}
\cJ_{\Theta,L,j}
&=\Big\{\big(a_A,h_A,a_B,h_B\big)\in \N^4\colon\, a_A\in l_{A,\Theta,j} L+h_A+2\N_0,\, 
a_B\in l_{B,\Theta,j} L+h_B+2 \N_0\Big\},\\
\nonumber 
\cJ^\cI
&=\Big\{\big(a^\cI,h^\cI\big)\in \N^2\colon\, a^\cI\in h^\cI+ 2\N_0\Big\},
\end{align}
and put $\psi_3(L,\Theta,u,j)=\frac{1}{ uL} \log Z_3(L,\Theta,u,j)$ with
\begin{align}
\label{fatZ}
Z_3(L,\Theta,u,j)=\sum_{(a,h)\in \cJ_{\Theta,L,j}} 
&\sum_{(a^\cI,h^\cI)\in \cJ^\cI} A_3\,B_3\,C_3,
\end{align}
where
\begin{align}
\label{defAhat1}
\nonumber 
A_3 &= e^{a_A\,\tilde{\kappa}\Big(\tfrac{a_A}{h_A},\,\tfrac{l_{A,\Theta,j} L}{h_A}\Big)}\, 
e^{a_B\, \tilde{\kappa}\Big(\tfrac{a_B}{h_B},\,\tfrac{l_{B,\Theta,j} L}{h_B}\Big)}\, 
e^{\tfrac{\beta-\alpha}{2}\, a_B},\\
\nonumber 
B_3 &= e^{a^\cI\, \phi \big(\tfrac{a^\cI}{h^\cI}\big)},\\
C_3 &= 1_{\{a_A+a_B+a^\cI=uL\}}\,1_{\{h_A+h_B+h^\cI=L\}}.
\end{align}
In order to establish a link between $\psi_2$ and $\psi_3$ we define, for $(a,h)\in
\cJ_{\Theta,L,j}$ and $(a^\cI,h^\cI)\in \cJ^\cI$,
\begin{align}
\label{recod}
\nonumber 
\cP_{(a,h)} &= \big\{(t,d)\in \cD_{\Theta,L,j}\colon\,
\textstyle\sum_{i\in \cA_{\Theta,j}} (t_{i},d_{i})=(a_A,h_A),\, 
\sum_{i\in \cB_{\Theta,j}} (t_{i},d_{i})=(a_B,h_B)\big\},\\
\cQ_{(a^\cI,h^\cI)}&=\big\{(t^\cI,d^\cI)\in \cD_{r}^\cI\colon\,
\textstyle \sum_{i=1}^{r}(t^\cI_{i},d^\cI_{i})=(a^\cI,h^\cI)\big\}.
\end{align}
Then we can rewrite $Z_2$ as 
\begin{align}
\label{fatZT}
Z_2(L,\Theta,u,j)=\sum_{(a,h)\in \cJ_{\Theta,L,j}} 
&\sum_{(a^\cI,h^\cI)\in \cJ^\cI} C_3
\sum_{(t,d)\in \cP_{(a,h)}}
\sum_{(t^\cI,d^\cI)\in \cQ_{(a^\cI,h^\cI)}} A_2\,B_2.
\end{align}

To prove that $\psi_2\prec \psi_3$, we need the following lemma.

\begin{lemma}
\label{tesla3}
For every $\eta>0$ there exists an $L_\eta\in \N$ such that, for every $(L,\Theta,u,j)
\in G_M^{\,m}$ with $L\geq L_\eta$ and every $(d,t)\in \cD_{\Theta,L,j}$ and 
$(d^\cI,t^\cI)\in \cD_{r}^\cI$ satisfying $\sum_{i=1}^{r+1} d_i+\sum_{i=1}^{r} 
d_i^\cI=L$ and $\sum_{i=1}^{r+1} t_i+\sum_{i=1}^{r} t_i^\cI=u L$,
\begin{align}
\label{necbound}
t_i\,\tilde{\kappa}\big(\tfrac{t_i}{d_i},\,\tfrac{l_i L}{d_i}\big)-\eta uL 
&\leq t_i\,\tilde{\kappa}_{d_i}\big(\tfrac{t_i}{d_i},\,\tfrac{l_i L}{d_i}\big)
\leq  t_i\,\tilde{\kappa}\big(\tfrac{t_i}{d_i},\,\tfrac{l_i L}{d_i}\big)
+\eta uL \quad i = 1,\dots,r+1,\\
\nonumber
t_i^\cI \phi(\tfrac{t_i^\cI}{d_i^\cI})-\eta uL 
&\leq t_i^\cI \phi_{d^\cI_i}(\tfrac{t_i^\cI}{d_i^\cI})
\leq t_i^\cI \phi(\tfrac{t_i^\cI}{d_i^\cI})+\eta uL \quad i = 1,\dots,r.
\end{align}
\end{lemma}

\begin{proof}  
By using Lemmas~\ref{conunif1} and \ref{l:feinflim1} in Appendix~\ref{Path entropies},  
we have that there exists a $\tilde{L}_\eta\in \N$ such that, for $L\geq \tilde{L}_\eta$, 
$(u,l)\in \cH_L$ and $\mu\in 1+\frac{2\N}{L}$, 
\be{tesla}
|\tilde{\kappa}_L(u,l)-\tilde{\kappa}(u,l)| 
\leq \eta, \qquad |\phi^\cI_L(\mu)- \phi^\cI(\mu)|\leq \eta.
\ee 
Moreover, Lemmas~\ref{lementr}, \ref{l:lemconv2}(ii--iii), \ref{l:lemconv}(ii) and
\ref{l:feinflim1} ensure that there exists a $v_\eta>1$ such that, for $L\geq 1$, 
$(u,l)\in \cH_L$ with $u\geq v_\eta$ and $\mu\in 1+\frac{2\N}{L}$ with $\mu\geq v_\eta$,
\be{tesla2}
0\leq \tilde{\kappa}_L(u,l)\leq \eta, \qquad 
0\leq  \phi_L(\mu)\leq \eta.
\ee 
Note that the two inequalities in \eqref{tesla2} remain valid when $L=\infty$. Next, 
we set $r_\eta=\eta/(2 v_\eta C_{\text{uf}})$ and $L_\eta=\tilde{L}_\eta/r_\eta$, and 
we consider $L\geq L_\eta$. Because of the left-hand side of \eqref{tesla}, the two 
inequalities in the first line of \eqref{necbound} hold when $d_i \geq r_\eta L\geq 
\tilde{L}_\eta$. We deal with the case $d_i \leq r_\eta L$ by considering first the 
case $t_i\leq \eta u L/2 C_{\text{uf}}$, which is easy because $\tilde{\kappa}_{d_i}$ 
and $\tilde{\kappa}$ are uniformly bounded by $C_{\text{uf}}$ (see \eqref{boundel}). 
The case $t_i\geq \eta u L/2 C_{\text{uf}}$ gives $t_i/d_i\geq u v_\eta\geq v_\eta$, 
which by the left-hand side of \eqref{tesla2} completes the proof of the first line 
in \eqref{necbound}. The same observations applied to $t_i^\cI, d_i^\cI$ combined 
with the right-hand side of \eqref{tesla} and \eqref{tesla2} provide the two inequalities 
in the second line in \eqref{necbound}.
\end{proof}

To prove that $\psi_2\prec\psi_3$, we apply Lemma~\ref{tesla3} with $\eta=\gep/(2m+1)$ 
and we use \eqref{defA1} to obtain, for $L\geq L_{\gep/(2m+1)}$, $(d,t)\in\cD_{\Theta,L,j}$ 
and $(d^\cI,t^\cI)\in \cD_{r}^\cI$, 
\begin{align}
\label{defAA1}
A_2 &\leq \prod_{i\in \cA_{\Theta,j}} e^{ t_{i}\, \tilde{\kappa}
\big(\tfrac{t_{i}}{d_i}, \,\tfrac{ l_{i} L}{d_i}\big) 
+ \tfrac{\gep uL}{2m+1}} \prod_{i\in \cB_{\Theta,j}} 
\,e^{t_{i}\,\tilde{\kappa}\big(\tfrac{t_{i}}{d_i}, 
\,\tfrac{ l_{i} L}{d_i}\big) + t_i\,\tfrac{\beta-\alpha}{2} 
+ \tfrac{\gep uL}{2m+1}},\\
\nonumber 
B_2 &\leq \prod_{i=1}^{r}\, e^{t_{i}^\cI\,\phi\Big(\tfrac{t^\cI_{i}}{d^\cI_{i}}\Big)
+ \tfrac{\gep uL}{2m+1}}.
\end{align}
Next, we pick $(a,h)\in \cJ_{\Theta,L,j}$, $(a^\cI,h^\cI)\in \cJ^\cI$, $(t,d)\in
\cP_{(a,h)}$ and $(t^\cI,d^\cI)\in \cQ_{(a^\cI,h^\cI)}$, and we use the concavity 
of $(a,b)\mapsto a \tilde{\kappa}(a,b)$ and $\mu\mapsto \phi^\cI(\mu)$ (see 
Lemma~\ref{l:lemconv2} in Appendix~\ref{Path entropies} and Lemma~\ref{l:lemconv}
in Appendix~\ref{B}) to rewrite \eqref{defAA1} as
\begin{align}
\label{defAA2}
A_2 &\leq e^{a_A\,\tilde{\kappa}\big(\tfrac{a_A}{h_A},\,
\tfrac{l_{A,\Theta,j} L}{h_A}\big) + a_B \, \tilde{\kappa}\big(\tfrac{a_B}{h_B},\,
\tfrac{l_{B,\Theta,j} L}{h_B}\big) + \tfrac{\beta-\alpha}{2} a_B 
+ \tfrac{\gep (r+1)uL}{2m+1}} 
= A_3\,e^{\tfrac{\gep (r+1) u L}{2m+1}},\\
\nonumber
B_2 &\leq e^{a^\cI\, \phi^\cI\big(\tfrac{a^\cI}{h^\cI}\big)
+ \tfrac{\gep r uL}{2m+1}} = B_3\, e^{\tfrac{\gep r uL}{2m+1}}.
\end{align}
Moreover, $r$, which is the number of $AB$ interfaces crossed by the trajectories in 
$\cW_{\Theta,u,j,L}$, is at most $m$ (see \eqref{zto}), so that \eqref{defAA2} 
allows us to rewrite \eqref{fatZT} as
\begin{align}
\label{fatZTU}
Z_2(L,\Theta,u,j) \leq e^{\gep u L} \sum_{(a,h)\in \cJ_{\Theta,L,j}} 
&\sum_{(a^\cI,h^\cI)\in \cJ^\cI} C_3\,|\cP_{(a,h)}|\, |\cQ_{(a^\cI,h^\cI)}|\, A_3\,B_3.
\end{align}
Finally, it turns out that $|\cP_{(a,h)}|\leq (uL)^{8r}$ and $|\cQ_{(a^\cI,h^\cI)}|
\leq (uL)^{8r}$. Therefore, since $r\leq m$, \eqref{fatZ} and \eqref{fatZTU} allow 
us to write, for $(L,\Theta,u,j)\in G_M^{\,m}$ and $L\geq L_{\gep/2m+1}$,
\be{add10}
Z_2(L,\Theta,u,j)\leq (mL)^{16 m} Z_3(L,\Theta,u,j).
\ee
The latter is sufficient to conclude that $\psi_2\prec\psi_3$.

%%%%%%%%%%%%%%%%%%%%%%%%%%%%%%%%%%%%%%%

\subsection{Step 3}
\label{s3}

For every $(L,\Theta,u,j)\in G_M^{\,m}$ we have, by the definition of $\cL(l_{A,\Theta,j},l_{B,\Theta,j};u)$ 
in \eqref{defnu}, that $(a,h)\in \cJ_{\Theta,L,j}$ and $(a^\cI,h^\cI)\in \cJ^\cI$ 
satisfying $a_A+a_B+a^\cI=uL$ and $h_A+h_B+h^\cI=L$ also satisfy  
\be{innu}
\Big(\big(\tfrac{a_A}{L},\tfrac{a_B}{L}, 
\tfrac{a^\cI}{L}\big),\big(\tfrac{h_A}{L},\tfrac{h_B}{L},
\tfrac{h^\cI}{L}\big)\Big)\in \cL(l_{A,\Theta,j},l_{B,\Theta,j};u).
\ee
Hence, \eqref{innu} and the definition of $\psi_\cI$ in \eqref{Bloc of type I} ensure 
that, for this choice of $(a,h)$ and $(a^\cI,h^\cI)$,
\be{add11}
A_3 B_3\leq e^{uL \psi_\cI(u,\,l_{A,\Theta,j},\,l_{B,\Theta,j})}.
\ee
Because of $C_3$, the summation in \eqref{fatZ} is restricted to those $(a,h)\in
\cJ_{\Theta,L,j}$ and $(a^\cI,h^\cI)\in \cJ^\cI$ for which $a_A,a_B,a^\cI \leq uL$ 
and $h_A,h_B,h^\cI\leq L$. Hence, the summation is restricted to a set of cardinality 
at most $(uL)^3 L^3$. Consequently, for all $(L,\Theta,u,j)\in G_M^{\,m}$ we have
\begin{align}
\label{fatZalt}
Z_3(L,\Theta,u,j)&=\sum_{(a,h)\in \cJ_{\Theta,L,j}} 
\sum_{(a^\cI,h^\cI)\in \cJ^\cI}  A_4\, B_4\, C_4\leq (m L)^3 L^3\, 
e^{u\,L\, \psi_\cI(u,\,l_{A,\Theta,j},\,l_{B,\Theta,j})}.
\end{align}
The latter implies that $\psi_3\prec \psi_4$ since $\psi_4=\psi_\cI(u,\,l_{A,\Theta,j},
\,l_{B,\Theta,j})$ by definition (recall \eqref{Bloc of type I1} and \eqref{psidefvar}).

%%%%%%%%%%%%%%%% SECTION 4 %%%%%%%%%%%%%%%%%%%%%%%%%%%
\section{Proof of Theorem \ref{varformula}}
\label{proofofgene}

This section is technically involved because it goes through a sequence of approximation
steps in which the self-averaging of the free energy with respect to $\omega$ and $\Omega$ 
in the limit as $n\to\infty$ is proven, and the various ingredients of the variational 
formula in Theorem~\ref{varformula} that were constructed in Section~\ref{keyingr} are 
put together.

In Section~\ref{strprTh} we introduce additional notation and state Propositions~\ref{pr:formimp}, 
\ref{pr:formimpp} and \ref{varar} from which Theorem~\ref{varformula} is a straightforward 
consequence. Proposition~\ref{pr:formimp}, which deals with $(M,m)\in\EIGH$, is proven in 
Section~\ref{prooffor} and the details of the proof are worked out in 
Sections~\ref{s22}--\ref{s23alt}, organized into 5 Steps that link intermediate free 
energies. We pass to the limit $m\to\infty$ with Propositions~\ref{pr:formimpp} and 
\ref{pr:varar} which are proven in Section~\ref{sMinf} and \ref{svarar}, respectively.

%%%%%%%%%%%%%%%%%

\subsection{Proof of Theorem~\ref{varformula}}
\label{strprTh}

%%%%%%%%%%%%%%%%%%%%%%

\subsubsection{Additional notation}

Pick $(M,m)\in \EIGH$ and recall that $\Omega$ and $\omega$ are independent, i.e., 
$\P=\P_\omega \times \P_\Omega$. For $\Omega\in \{A,B\}^{\N_0\times\Z}$, $\omega\in 
\{A,B\}^\N$, $n\in \N$ and $(\alpha,\beta)\in \CONE$, define
\begin{equation}
\label{intermp}
f^{\omega,\Omega}_{1,n}(M,m;\alpha,\beta)
= \tfrac{1}{n} \log Z_{1,n,L_n}^{\,\omega,\Omega}(M,m)
\quad\text{with}\quad
Z_{1,n,L_n}^{\,\omega,\Omega}(M,m)
=\sum_{\pi \in \cW_{n,M,}^{\,m}} e^{\,H_{n,L_n}^{\omega,\Omega}(\pi)},
\ee 
where $\cW_{n,M}^{\,m}$ contains those paths in $\cW_{n,M}$ that, in each column, make 
at most $m L_n$ steps. We also restrict the set $\cR_{p,M}$ in \eqref{RMNdef} to those 
limiting empirical measures whose support is included in $\overline \cV_{M}^{\,m}$, i.e., those measures charging the types of column that 
can be crossed in less than $m L_n$ steps only. To that aim we recall \eqref{RMNdefalt} 
and define, for $\Omega\in \{A,B\}^{\N_0\times\Z}$ and $N\in \N$,
\be{RMNNdefalt}
\begin{aligned}
\cR^{\Omega,m}_{M,N} &= \big\{\rho_N(\Omega,\Pi,b,x)\ \text{with}\  
b=(b_j)_{j\in \N_0} \in (\mathbb{Q}_{(0,1]})^{\N_0},\\
&\qquad \Pi=(\Pi_j)_{j\in\N_0} \in \{0\}\times \Z^{\N} \colon\,
|\Delta\Pi_j|\leq M\,\ \ \forall\,j\in\N_0,\\
&\qquad x=(x_j)_{j\in \N_0} \in \{1,2\}^{N_0}\colon\, \big(\Omega(j,\Pi_j+\cdot),
\Delta \Pi_j,b_j,b_{j+1},x_j\big)\in \cV_M^{m}\big\}
\end{aligned} 
\ee
which is a  subset of $\cR_{M,N}^{\Omega}$ and allows us to define 
\be{RMmdef}
\cR^{\Omega,m}_M = \mathrm{closure}\Big(\cap_{N'\in\N} \cup_{N \geq N'}\, 
\cR^{\Omega,m}_{M,N}\Big),
\ee
which, for $\P$-a.e. $\Omega$ is equal to $\cR_{p,M}^{m}\subsetneq \cR_{p,M}$.

At this stage, we further define, 
\be{modvaralt}
f(M,m;\alpha,\beta)
= \sup_{\rho\in \cR_{p,M}^{\,m}}\,
\sup_{(u_\Theta)_{\Theta\in \overline\cV_M^{\,m}}\in\cB_{\,\overline\cV_M^{\,m}}}
V(\rho,u),
\ee
where
\be{ddeff}
V(\rho,u)=\frac{\int_{\overline \cV_M^{\,m}}\,u_\Theta\,
\psi(\Theta,u_{\Theta};\alpha,\beta)\,\rho(d\Theta)}
{\int_{\overline \cV_M^{\,m}}\,u_\Theta\, \rho(d\Theta)},
\ee
where (recall \eqref{set2alt1})
\be{defbm}
\cB_{\overline \cV_M^{\,m}}
= \Big\{(u_\Theta)_{\Theta\in \overline \cV_M^{\,m}}\in\R^{\overline \cV_M^{\,m}}
\colon \Theta \mapsto u_\Theta\in \cC^0\big(\overline \cV_M^{\,m},\R\big),\,
t_\Theta \leq u_\Theta\leq m\,\,\forall\,\Theta\in \overline \cV_M^{\,m} \Big\},
\ee
and where $\overline{\cV}_M^{\,m}$ is endowed with the distance $d_M$ defined in 
\eqref{dist} in Appendix~\ref{B.2}.

Let $\cW^{*,m}_{n,M}\subset\cW_{n,M}^{\,m}$ be the subset consisting of those paths 
whose endpoint lies at the boundary between two columns of blocks, i.e., satisfies
$\pi_{n,1}\in \N L_n$. Recall \eqref{intermp}, and define $Z^{*,\omega,\Omega}_{n,L_n}
(M)$ and $f^{*,\omega,\Omega}_{1,n}(M,m;\alpha,\beta)$ as the counterparts of 
$Z^{\omega,\Omega}_{n,L_n}(M,m)$ and  $f^{\omega,\Omega}_{1,n}(M,m;\alpha,\beta)$ 
when $\cW_{n,M}^{\,m}$ is replaced by $\cW_{n,M}^{*,m}$. Then there exists a constant $c>0$, 
depending on $\alpha$ and $\beta$ only, such that
\be{inegW}
\begin{aligned}
&Z^{\omega,\Omega}_{1,n,L_n}(M,m) e^{-c L_n}
\leq Z^{*,\omega,\Omega}_{1,n,L_n}(M,m)
\leq Z^{\omega,\Omega}_{1,n,L_n}(M,m),\\  
&n\in \N, \,\omega \in \{A,B\}^\N,\, \Omega \in \{A,B\}^{\N_0\times \Z}.
\end{aligned}
\ee
The left-hand side of the latter inequality is obtained by changing the last $L_n$ 
steps of each trajectory in $\cW_{n,M}^{\,m}$ to make sure that the endpoint falls in 
$L_n\N$. The energetic and entropic cost of this change are obviously $O(L_n)$. 
By assumption, $\lim_{n\to\infty} L_n/n=0$, which together with \eqref{inegW} 
implies that the limits of $f^{\omega,\Omega}_{1,n}(M,m;\alpha,\beta)$ and 
$f^{*,\omega,\Omega}_{1,n}(M,m;\alpha,\beta)$ as $n\to\infty$ are the same. In the 
sequel we will therefore restrict the summation in the partition function to 
$\cW_{n,M}^{*,m}$ and drop the $*$ from the notations.

Finally, let
\be{add12}
\begin{aligned}
f_{1,n}^{\Omega}(M,m;\alpha,\beta)
&= \E_\omega\big[f^{\omega,\Omega}_{1,n}(M,m;\alpha,\beta)\big],\\
f_{1,n}(M,m;\alpha,\beta)
&= \E_{\omega,\Omega}\big[f^{\omega,\Omega}_{1,n}(M,m;\alpha,\beta)\big],
\end{aligned}
\ee
and recall \eqref{partfunc} to set $f_n^\Omega(M;\alpha,\beta)=\E_\omega[f_n^{\omega,
\Omega}(M;\alpha,\beta)]$.

%%%%%%%%%%%%%%%%%%%%%%%%%%%%%%

\subsubsection{Key Propositions}
\label{Key Propositions}

Theorem~\ref{varformula} is a consequence of Propositions \ref{pr:formimp}, \ref{pr:formimpp} 
and \ref{pr:varar} stated below and proven in Sections~\ref{s22}--\ref{s23alt}, 
Sections~\ref{step1}--\ref{step3} and Section~\ref{svarar}, respectively. 

\bp{pr:formimp}
For all $(M,m)\in \EIGH$, 
\be{formimp}
\lim_{n\to \infty} f^{\Omega}_{1,n}(M,m;\alpha,\beta)
=f(M,m;\alpha,\beta)\quad \text{ for } \P-a.e.\,\Omega.
\ee
\ep
\bp{pr:formimpp}
For all $M\in \N$,
\be{formimpp}
\lim_{n\to \infty} f^{\Omega}_{n}(M;\alpha,\beta)
=\sup_{m\geq M+2} f(M,m;\alpha,\beta)\quad \text{ for } \P-a.e.\,\Omega.
\ee
\ep

\bp{pr:varar}
For all $M\in \N$,
\be{varar}
\sup_{m\geq M+2} f(M,m;\alpha,\beta)=\sup_{\rho\in \cR_{p,M}}\,
\sup_{(u_\Theta)_{\Theta\in \overline\cV_M}\, 
\in\,\cB_{\,\overline\cV_M}}   V(\rho,u),
\ee
\ep
where, in the righthand side of \eqref{varar}, we recognize the variational formula of 
Theorem \ref{varformula} and with $\cB_{\overline\cV_M}$ defined in \eqref{set2}.

\medskip\noindent 
{\bf Proof of Theorem~\ref{varformula} subject to Propositions \ref{pr:formimp}, 
\ref{pr:formimpp} and \ref{pr:varar}.}
The proof of Theorem~\ref{varformula} will be complete once we show that for all 
$(M,m)\in \EIGH$
\be{rstt}
\lim_{n\to \infty} |f_{n}^{\omega,\Omega}(M,m;\alpha,\beta)
-f_{n}^{\Omega}(M,m;\alpha,\beta)|=0
\quad \text{ for } \P-a.e.\,\,(\omega,\Omega).
\ee 
To that aim, we note that for all $n\in \N$ the $\Omega$-dependence of $f_{n}^{\omega,\Omega}
(M,m;\alpha,\beta)$ is restricted to $\big\{\Omega_x\colon\,x\in G_n\big\}$ with $G_n
=\{0,\dots,\frac{n}{L_n}\}\times\{-\frac{n}{L_n},\dots,\frac{n}{L_n}\}$. Thus, for $n\in\N$ 
and $\gep>0$ we set 
\be{defac}
A_{\gep,n}=\{| f_{n}^{\omega,\Omega}(M;\alpha,\beta)- f_{n}^{\Omega}(M;\alpha,\beta)|>\gep)\},
\ee
and by independence of $\omega$ and $\Omega$ we can write
\begin{align}
\label{concv}
{\textstyle \nonumber\P_{\omega,\Omega}(A_{\gep,n})}
&{\textstyle =\sum_{\Upsilon\in \{A,B\}^{G_n}} 
\P_{\omega,\Omega}(A_{\gep,n}\cap \{\Omega_{G_n}=\Upsilon\})}\\
&{\textstyle =\sum_{\Upsilon\in \{A,B\}^{G_n}} 
\P_{\omega}(| f_{n}^{\omega,\Upsilon}(M;\alpha,\beta)
- f_{n}^{\Upsilon}(M;\alpha,\beta)|>\gep)\ \P_{\Omega}(\{\Omega_{G_n}=\Upsilon\})}.
\end{align}
At this stage, for each $n\in \N$ we can apply the concentration inequality $\eqref{concmesut}$ 
in Appendix \ref{Ann2} with $\Gamma=\cW_{n,M}^{\,m}$, $l=n$, $\eta=\gep n$,
\be{addd25}
\xi_i = -\alpha\, 1\{\omega_i=A\} +\beta\, 
1\{\omega_i=B\}, \qquad i\in\N,
\ee 
and with $T(x,y)$ indicating in which block step $(x,y)$ lies in. Therefore, there exist 
$C_1, C_2>0$ such that for all $n\in \N$ and all $\Upsilon\in \{A,B\}^{G_n}$ we have
\be{resconc}
\P_{\omega}(| f_{n}^{\omega,\Upsilon}(M;\alpha,\beta)
- f_{n}^{\Upsilon}(M;\alpha,\beta)|>\gep) \leq C_1 e^{-C_2 \gep^2 n},
\ee
which, together with \eqref{concv} yields $\P_{\omega,\Omega}(A_{\gep,n})\leq 
C_1 e^{-C_2 \gep^2 n}$ for all $n\in \N$. By using the Borel-Cantelli Lemma, we 
obtain \eqref{rstt}.
\hspace*{\fill}$\square$

%%%%%%%%%%%%%%%%%%%%%%%%%%%%%%%

\subsection{Proof of Proposition \ref{pr:formimp}}
\label{prooffor} 
Pick $(M,m)\in \EIGH$ and $(\alpha,\beta)\in \CONE$. In Steps 1--2 in 
Sections~\ref{s22}--\ref{s23} we introduce an intermediate free energy 
$f_{3,n}^{\Omega}(M,m;\alpha,\beta)$ and show that
\begin{align}
\label{limsup1}
\lim_{n\to \infty} |f_{1,n}^{\Omega}(M,m;\alpha,\beta)
- f_{3,n}^{\Omega}(M,m;\alpha,\beta)|=0 \qquad \forall\,
\Omega\in \{A,B\}^{N_0\times \Z}.
\end{align}
Next, in Steps 3--4 in Sections~\ref{s24}--\ref{s24alt} we show that 
\be{limsup3}
\limsup_{n\to \infty} f_{3,n}^{\Omega}(M,m;\alpha,\beta)
= f(M,m;\alpha,\beta)
\qquad \text{for}\,\,\P-a.e.\,\, \Omega,
\ee
while in Step 5 in Section~\ref{s23alt} we prove that 
\be{convco1}
\liminf_{n\to \infty} f_{3,n}^{\Omega}(M,m;\alpha,\beta)
= \limsup_{n\to \infty} f_{3,n}^{\Omega}(M,m;\alpha,\beta) 
\qquad \text{for}\,\,\P-a.e.\,\,\Omega.
\ee
Combing (\ref{limsup1}--\ref{convco1}) we get
\be{fin}
\liminf_{n\to \infty} f_{1,n}^{\Omega}(M,m;\alpha,\beta)
= \limsup_{n\to \infty} f_{1,n}^{\Omega}(M,m;\alpha,\beta)=f(M,m;\alpha,\beta)
\qquad \text{for}\,\,\P-a.e.\,\, \Omega,
\ee
which completes the proof of Proposition~\ref{pr:formimp}.

In the proof we need the following order relation. 

\begin{definition}
For $g,\widetilde{g}\colon\,\N^3\times \CONE\mapsto\R$, write $g\prec \widetilde{g}$ 
if for all $(M,m)\in \EIGH$, $(\alpha,\beta)\in \CONE$ and $\gep>0$ there exists an 
$n_{\gep}\in \N$ such that
\be{add13}
g(n,M,m; \alpha,\beta) \leq \widetilde{g}(n,M,m; \alpha,\beta)
+\gep \qquad \forall\,n\geq n_{\gep}.
\ee
\end{definition}

\noindent
The proof of \eqref{limsup1} will be complete once we show that $f_1^\Omega \prec
f_3^\Omega$ and $f_3^\Omega \prec f_1^\Omega$ for all $\Omega\in \{A,B\}^{\N_0
\times \Z}$. We will focus on $f_1^\Omega\prec f_3^\Omega$, since the proof of 
the latter can be easily adapted to obtain $f_3^\Omega \prec f_1^\Omega$. To prove
$f_1^\Omega \prec f_3^\Omega$ we introduce another intermediate free energy 
$f_2^\Omega$, and we show that $f_1^\Omega\prec f_2^\Omega$ and $f_2^\Omega \prec 
f_3^\Omega$. 

For $L\in \N$, let
\begin{equation}
\label{defDD2}
\cD_L^M=\left\{\Xi
=(\Delta\Pi,b_0,b_1)\in \{-M,\dots,M\}\times 
\{\tfrac1L,\tfrac2L,\dots,1\}^2 \right\}.	
\end{equation}
For $L,N\in\N$, let 
\begin{align}
\label{defD2}
\widetilde{\cD}_{L,N}^M 
&= \Big\{\Theta_{\text{traj}}=(\Xi_i)_{i\in \{0,\dots, N-1\}}
\in (\cD_L^M)^N\colon\, b_{0,0}=\tfrac1L,\,
b_{0,i}=b_{1,i-1}\,\,\forall\,1\leq i\leq N-1\Big\},
\end{align}
and with each $\Theta_{\text{traj}}\in\widetilde{\cD}_{L,N}^M$ associate the sequence 
$(\Pi_i)_{i=0}^N$ defined by $\Pi_0=0$ and $\Pi_i=\sum_{j=0}^{i-1} \Delta\Pi_j$ 
for $1\leq i\leq N$. Next, for $\Omega\in \{A,B\}^{\N_0\times \Z}$ and $\Theta_{\text{traj}}
\in \widetilde{\cD}_{L,N}^M$, set
\be{defXalt}
\cX_{\Theta_{\text{traj}},\Omega}^{M,m}
=\big\{x\in \{1,2\}^{\{0,\dots,N-1\}}\colon 
(\Omega(i,\Pi_i+\cdot),\Xi_i,x_i)\in  \cV_{M}^{m} \,\,\forall\,0\leq i\leq N-1\big\},
\ee
and, for $x\in \cX_{\Theta_{\text{traj}},\Omega}^{M,m}$, set
\be{tthe}
\Theta_i=(\Omega(i,\Pi_i+\cdot),\Xi_i,x_i) \quad \text{for}\quad i\in \{0,\dots,N-1\}
\ee 
and 
\begin{equation}
\begin{aligned}
\label{defU2}
\cU_{\,\Theta_{\text{traj}},x,n}^{\,M,m,L} &= \Big\{u=(u_i)_{ i\in \{0,\dots, N-1\}}
\in [1,m]^{N}\colon u_i\in  t_{\Theta_i}+\tfrac{2\N}{L}
\ \,\,\forall\,0\leq i\leq N-1,\,\sum_{i=0}^{N-1} u_i=\tfrac{n}{L}\Big\}.
\end{aligned}
\end{equation}
Note that $\cU_{\,\Theta_{\text{traj}},x,n}^{\,M,m,L}$  is empty when $N\notin
\big[\frac{n}{m L},\frac{n}{L}\big]$.

For $\pi\in \cW_{n,M}^{\,m}$, we let $N_\pi$ be the number of columns crossed by $\pi$ 
after $n$ steps. We denote by $(u_0(\pi),\dots,u_{N_\pi-1}(\pi))$ the time spent 
by $\pi$ in each column divided by $L_n$, and we set $\widetilde{u}_0(\pi)=0$ 
and $\widetilde{u}_{j}(\pi)=\sum_{k=0}^{j-1} u_k(\pi)$ for $1\leq j\leq N_\pi$. With 
these notations, the partition function in \eqref{intermp} can be rewritten as
\begin{equation}
\label{A11}
Z_{1,n,L_n}^{\,\omega,\Omega}(M,m)=\sum_{N=n/m L_n}^{n/L_n}\,
\,\sum_{\Theta_{\text{traj}}\in \widetilde{\cD}_{L_n,N}^M }\,
\sum_{x\in \cX_{\Theta_{\text{traj}},\Omega}^{M,m} }
\,\sum_{u\in \,\cU_{\,\Theta_{\text{traj}},x,n}^{\,M,m,L_n}}   
\,A_1,
\end{equation}
with (recall \eqref{partfunc2})
\begin{align}
A_1=\prod_{i=0}^{N-1}\, Z^{\theta^{\widetilde{u}_{i} L_n}(\omega)}_{L_{n}}
(\Omega(i,\Pi_{i}+\cdot),\Xi_i,x_i, u_i).
\end{align}

%%%%%%%%%%%%%%%%%%%%%%%%%%%%%%%%%%%

\subsubsection{Step 1}
\label{s22}

In this step we average over the disorder $\omega$ in each column. To that end, we set
\be{interm3}
f_{2,n}^\Omega(M,m;\alpha,\beta) = \tfrac{1}{n} \log Z_{2,n,L_n}^{\Omega}(M,m)
\ee
with
\be{inteerm3}
Z_{2,n,L_n}^{\Omega}(M,m)=\sum_{N=n/m L_n}^{n/L_n} \ 
\,\sum_{\Theta_{\text{traj}}\in \widetilde{\cD}_{L_n,N}^M }\,
\sum_{x\in \cX_{\Theta_{\text{traj}},\Omega}^{M,m}}
\,\sum_{u\in \,\cU_{\,\Theta_{\text{traj}},x,n}^{\,M,m,L_n}}  
\,A_2,
\ee
where
\begin{align}
\label{defa24}
A_2 &=\prod_{i=0}^{N-1} e^{\E_\omega\big[\log  
Z^{\theta^{\widetilde{u}_{i}}(\omega)}_{L_{n}}
(\Omega(i,\Pi_{i}+\cdot),\Xi_i, x_i, u_i)\big]}
= \prod_{i=0}^{N-1}\, 
e^{u_i L_n \psi_{L_n}(\Omega(i,\Pi_{i}+\cdot),\Xi_i,x_i, u_i)}.
\end{align}
Note that the $\omega$-dependence has been removed from $Z^\Omega_{2,n,L_n}(M,m)$.

To prove that $f_1^\Omega\prec f_2^\Omega$, we need to show that for all $\gep>0$ 
there exists an $n_{\gep}\in \N$ such that, for $n\geq n_{\gep}$ and all $\Omega$,
\be{Eform}
\E_\omega\big[\log Z_{1,n,L_n}^{\,\omega,\Omega}(M,m)\big]
\leq \log Z_{2,n,L_n}^{\Omega}(M,m)+\gep n.
\ee
To this end, we rewrite $Z_{1,n,L_n}^{\,\omega,\Omega}(M,m)$ as
\be{rewr}
Z_{1,n,L_n}^{\,\omega,\Omega}(M,m)=\sum_{N=n/m L_n}^{n/L_n}
\sum_{\Theta_{\text{traj}}\in \widetilde{\cD}_{L_n,N}^M } 
\,\sum_{x\in \cX_{\Theta_{\text{traj}},\Omega}^{M,m} }
\,\sum_{u\in \,\cU_{\Theta_{\text{traj}},x,n}^{\,M,m,L_n}}
A_2\,\frac{A_1}{A_2},
\ee
where we note that
\be{A23}
\frac{A_1}{A_2} 
=\prod_{i=0}^{N-1}\,
e^{u_i L_n \big[\psi^{\theta^{\widetilde{u}_{i}L_n}(\omega)}_{L_n}
(\Omega(i,\Pi_{i}+\cdot),\Xi_i,x_i,u_i)
-\psi_{L_n}(\Omega(i,\Pi_{i}+\cdot),\Xi_i,x_i,u_i)\big]}.
\ee
In order to average over $\omega$, we apply a concentration of measure inequality. 
Set 
\be{add23}
\mathcal{K}_n = \bigcup_{N=n/m L_n}^{n/L_n}
\bigcup_{\Theta_{\text{traj}}\in \widetilde{\cD}_{L_n,N}^M} 
\,\bigcup_{x\in \cX_{\Theta_{\text{traj}},\Omega}^{M,m} }
\,\bigcup_{u\in \,\cU_{\Theta_{\text{traj}},x,n}^{\,M,m,L_n}}
\Big\{ |\log A_1 -\log A_2| \geq \gep n\Big\},
\ee
and note that $\omega\in \mathcal{K}_n^c$ implies that $Z_{1,n,L_n}^{\,\omega,\Omega}(M,m)
\leq e^{\gep n} Z_{2,n,L_n}^{\Omega}(M,m)$. Consequently, we can write
\begin{align}
\label{borne}
\nonumber 
\E_\omega\big[\log Z_{1,n,L_n}^{\,\omega,\Omega}(M,m)\big]
&= \E_\omega\big[\log Z_{1,n,L_n}^{\,\omega,\Omega}(M,m)
\,1_{\{\mathcal{K}_n\}} \big]+\E_\omega\big[\log Z_{1,n,L_n}^{\,\omega,\Omega}(M,m)
\,1_{\{\mathcal{K}_n^c\}}\big]\\
&\leq \E_\omega\big[\log Z_{1,n,L_n}^{\,\omega,\Omega}(M,m)
\,1_{\{\mathcal{K}_n\}} \big]+ \log Z_{2,n,L_n}^{\,\Omega}(M,m)\,
+ \gep n.
\end{align}
We can now use the uniform bound in \eqref{boundel} to control the first term in 
the right-hand side of \eqref{borne}, to obtain  
\begin{align}
\E_\omega\big[\log Z_{1,n,L_n}^{\,\omega,\Omega}(M,m)\big]
\leq &  \log Z_{2,n,L_n}^{\,\Omega}(M,m) +\gep n
+ C_{\text{uf}}(\alpha)\,n\,\P_\omega(\mathcal{K}_n).
\end{align}
Therefore the proof of this step will be complete once we show that 
$\P_\omega(\mathcal{K}_n)$ vanishes as $n\to\infty$.

\begin{lemma}
There exist $C_1,C_2>0$ such that, for all $\gep>0$, $n\in\N$, $N\in\big\{\tfrac{n}{m L_n},
\dots,\frac{n}{L_n}\big\}$, $\Omega\in \{A,B\}^{\N_0\times\Z}$, $\Theta_{\mathrm{traj}}
\in \widetilde{\cD}_{L_n,N}^M$, $x\in \cX_{\Theta_{\mathrm{traj}},\Omega}^{M,m}$ and $u\in \,\cU_{\Theta_{\mathrm{traj}},x,n}^{\,M,m,L_n}$,
\begin{equation}
\P_{\omega}(|\log  A_1-\log A_2|
\geq \gep n)\leq C_1 e^{-C_2\gep^2 n}.
\end{equation} 
\end{lemma}

\begin{proof}
Pick $\Theta_{\text{traj}}\in \widetilde{\cD}_{L_n,N}^M$, $x\in \cX_{\Theta_{\text{traj},
\Omega}}^{M,m}$ and $u\in\,\cU_{\Theta_{\text{traj}},x,n}^{\,M,m,L_n}$, and consider the 
subset $\Gamma$ of $\cW_{n,M}^{\,m}$ consisting of those paths of length $n$ that first 
cross the $(\Omega(0,\cdot),\Xi_{0},x_0)$ column such that $\pi_0=(0,1)$ and 
$\pi_{\widetilde{u}_{1}L_n}=(1,\Pi_{1}+b_{1,0}) L_n$, then cross the $(\Omega(1,\cdot),
\Xi_{1},x_1)$ column such that $\pi_{\widetilde{u}_{1}L_n+1}=(1+1/L_n,\Pi_{1}+b_{1,0})
L_n$ and $\pi_{\widetilde{u}_{2}L_n}=(2,\Pi_{2}+b_{1,1}) L_n$, and so on. We can apply 
the concentration of measure inequality stated in \eqref{concmesut} to the set 
$\Gamma$ defined above, with $l=n$, $\eta=\gep n$,
\be{add25}
\xi_i = -\alpha\, 1\{\omega_i=A\} +\beta\, 
1\{\omega_i=B\}, \qquad i\in\N,
\ee  
and with $T(x,y)$ indicating in which block step $(x,y)$ lies in. After noting that 
$\E_\omega(\log A_1)= \log A_2$, we  obtain that there exist $C_{1},C_{2}>0$ such 
that, for all $n\in\N$, $N\in\big\{\tfrac{n}{m L_n},\dots,\frac{n}{L_n}\big\}$, 
$\Omega\in \{A,B\}^{\N_0\times\Z}$, $\Theta_{\mathrm{traj}}\in \widetilde{\cD}_{L_n,N}^M$, 
$x\in \cX_{\Theta_{\text{traj},\Omega}}^{M,m}$ and $u\in \,\cU_{\Theta_{\mathrm{traj}},
x,n}^{\,M,m,L_n}$,
\be{rec2}
\P\big(|\log  A_1 -\log A_2| \geq \gep\,n \big) \leq C_{1}\,e^{-C_{2}\,\gep^3\, n}.
\ee
\end{proof}

It now suffices to remark that 
\be{setgr}
\big|\{(N,\Theta_{\text{traj}},x,u)\colon\,N\in \{\tfrac{n}{m L_n},
\dots,\tfrac{n}{L_n}\}, \Theta_{\text{traj}}\in
\widetilde{\cD}_{L_n,N}^M,\, x\in \cX_{\Theta_{\text{traj},\Omega}}^{M,m}, 
u\in \,\cU_{\Theta_{\mathrm{traj}},x,n}^{\,M,m,L_n}\}\big|
\ee 
grows subexponentially in $n$ to obtain that $f_1^\Omega\prec f_2^\Omega$ for all $\Omega$.

%%%%%%%%%%%%%%%%%%%

\subsubsection{Step 2}
\label{s23}

In this step we replace the finite-size free energy $\psi_{L_n}$ by its limit $\psi$. 
To do so we introduce a third intermediate free energy, 
\be{interm1}
f_{3,n}^\Omega(M,m;\alpha,\beta) =
\E\big[\tfrac{1}{n} \log Z_{3,n,L_n}^{\Omega}(M,m)\big],
\ee
where
\be{inteerm1}
Z_{3,n,L_n}^{\Omega}(M,m)=\sum_{N=n/m L_n}^{n/L_n}\  
\sum_{\Theta_{\text{traj}}\in \widetilde{\cD}_{L_n,N}^M } \,
\sum_{x\in \cX_{\Theta_{\text{traj}},\Omega}^{M,m} }
\,\sum_{u\in \,\cU_{\Theta_{\text{traj}},x,n}^{\,M,m,L_n}}A_3
\ee
with
\begin{align}
\label{defa23}
A_{3} &=\prod_{i=0}^{N-1}\, 
e^{u_i L_n \psi(\Omega(i,\Pi_{i}+\cdot),\Xi_i,x_i,u_i)}.
\end{align}
For all $\Omega$,
\be{a3}
\frac{A_2}{A_3} = \prod_{i=0}^{N-1}\,
e^{u_i L_n \big[\psi_{L_n}(\Omega(i,\Pi_{i}+\cdot),\Xi_i,x_i,u_i)
-\psi(\Omega(i,\Pi_{i}+\cdot),\Xi_i,x_i,u_i)\big]},
\ee 
and, for all $i\in \{0,\dots,N-1\}$, we have 
$(\Omega(i,\Pi_{i}+\cdot),\Xi_i,x_i,u_i)\in \cV^{*,\,m}_M$, so that 
Proposition~\ref{convfree1} can be applied.

%%%%%%%%%%%%%%%%%%%%%%%%%%%%

\subsubsection{Step 3}
\label{s24}

In this step we want the variational formula \eqref{modvaralt} to appear. Recall \eqref{defrho} 
and define, for $n\in \N$, $(M,m)\in \EIGH$, $N\in \{\frac{n}{m L_n},\dots,
\frac{n}{L_n}\}$, $\Theta_{\text{traj}}\in \widetilde{\cD}_{L_n,N}^M$ and $x\in
\cX_{\Theta_{\text{traj}},\Omega}^{M,m}$,
\be{defthalt}
\Theta_j=(\Omega(j,\Pi_{j}+\cdot),\Xi_j,x_j),
\qquad j = 0,\dots,N-1,
\ee
and
\begin{equation}
\label{defrho2}
\rho^\Omega_{\Theta_{\text{traj}},x}\big(\Theta,\Theta^{'}\big)
=\frac{1}{N} \sum_{j=1}^{N} 
1_{\big\{(\Theta_{j-1},\Theta_{j})=(\Theta,\Theta^{'})\big\}},
\end{equation}
and, for $u\in \,\cU_{\Theta_{\text{traj}},x,n}^{\,M,m,L_n}$,
\be{defHam}
H^\Omega(\Theta_{\text{traj}},x,u)=\sum_{j=0}^{N-1} u_j\,\psi(\Theta_j, u_{j}). 
\ee
In terms of these quantities we can rewrite $Z_{3,n,L_n}^\Omega(M,m)$ in \eqref{inteerm1} 
as
\begin{align}
\label{variaf2}
Z_{3,n,L_n}^{\Omega}(M,m)
&=\sum_{N=n/m L_n}^{n/L_n} \sum_{\Theta_{\text{traj}}\in \widetilde{\cD}_{L_n,N}^M } 
 \,\sum_{x\in \cX_{\Theta_{\text{traj}},\Omega}^{M,m} }
\,\sum_{u\in \,\cU_{\Theta_{\text{traj}},x,n}^{\,M,m,L_n}}
e^{L_n\,  H^\Omega(\Theta_{\text{traj}},x,u) }.
\end{align}
For $n\in \N$, denote by
\be{}
N_n^{\Omega}, \qquad \Theta_{\text{traj},n}^{\Omega}
\in \widetilde{\cD}_{L_n,N^{\Omega}_n}^M, 
\qquad x_n^\Omega\in \cX_{\Theta_{\text{traj},n}^\Omega,\Omega}^{M,m},
\qquad u^\Omega_n \in\cU_{\Theta_{\text{traj},n}^\Omega,x_n^\Omega,n}^{\,M,m,L_n},
\ee 
the indices in the summation set of \eqref{variaf2} that maximize $H^\Omega(\Theta_{
\text{traj}},x,u)$. For ease of notation we put  
\be{redef}
\Theta_{\text{traj},n}^{\Omega}
=(\Xi_{j}^n)_{j=0}^{N_n^{\Omega}-1},\quad x^\Omega_n
=(x_j^n)_{j=0}^{N_n^{\Omega}-1}, \quad u^\Omega_n
=(u_j^n)_{j=0}^{N_n^{\Omega}-1},
\ee
and
\be{carsum}
c_n =\big|\{(N,\Theta_{\text{traj}},x,u)\colon\, 
\tfrac{n}{m L_n} \leq N\leq \tfrac{n}{L_n},\, \Theta_{\text{traj}}\in
\widetilde{\cD}^M_{L_n,N},\, x\in \cX_{\Theta_{\text{traj}},\Omega}^{M,m},\,
u\in \,\cU_{\Theta_{\text{traj}},x,n}^{\,M,m,L_n}\}\big|.
\ee
Then we can estimate
\be{impo1alt}
\frac1n \log Z_{3,n,L_n}^{\Omega}(M,m)\leq \frac1n \log c_n
+ \tfrac{L_n}{n}\sum_{j=0}^{N_n^{\Omega}-1} u_j^n\, 
\psi(\Theta_j^n, u_{j}^n). 
\ee

We next note that $u\mapsto u \psi(\Theta,u)$ is concave for all $\Theta\in 
\overline{\cV}_M$ (see Lemma~\ref{concav}). Hence, after setting 
\begin{align}
v_\Theta^n=\sum_{j=0}^{N_n^{\Omega}-1} 1_{\{\Theta_j^n=\Theta\}}\, 
u_j^n, \qquad  d_\Theta^n=\sum_{j=0}^{N_n^{\Omega}-1} 
1_{\{\Theta_j^n=\Theta\}}, \qquad \Theta\in \overline\cV_M^{\,m},
\end{align}
we can estimate 
\be{utilconcav}
\sum_{j=0}^{N_n^{\Omega}-1} 1_{\{\Theta_j^n=\Theta\}}\, 
u_{j}^n\, \psi(\Theta_j^n, u_{j}^n)
\leq v_\Theta^n\,\psi\big(\Theta, \tfrac{ v_\Theta^n}{d_\Theta^n}\big) 
\quad \text{for} \quad \Theta\in \overline\cV_M^{\,m}\colon d_\Theta^n\geq 1.
\ee 
Next, we recall \eqref{defrho2} and we set $\rho_{n}=\rho_{\Theta_{\text{traj},
n}^{\Omega},x_n^\Omega}^\Omega$, so that $\rho_{n,1}(\Theta)=d_\Theta^n/N_n^\Omega$ 
for all $\Theta\in \overline{\cV}_M^{\,m}$. Since $\{\Theta\in \overline{\cV}_M^{\,m}\colon\,
d_\Theta^n\geq 1\}$ is a finite subset of $\overline{\cV}_M^{\,m}$, we can easily extend 
$\Theta\mapsto v_\Theta^n/d_\Theta^n$ from $\{\Theta\in \overline{\cV}_M\colon\,
d_\Theta^n\geq 1\}$ to $\overline{\cV}_M^{\,m}$ as a continuous function. Moreover, 
$\sum_{j=0}^{N_n^\Omega-1} u_j^n=n/L_n$ implies that $N_n^\Omega \int_{\overline\cV_M^{\,m}}
v_\Theta^n/d_\Theta^n\,\rho_{n,1} (d\Theta)=n/L_n,$ which, together with
\eqref{impo1alt} and \eqref{utilconcav} gives 
\begin{align}\label{immp}
\tfrac1n \log Z_{3,n,L_n}^{\Omega}(M,m)
&\leq \sup_{u \in \cB_{\,\overline \cV_M^{\,m}}}\frac{\int_{\overline \cV_M^{\,m}} 
u_{\Theta}\, \psi(\Theta,u_\Theta)\,
\rho_{n}(d\Theta)}{\int_{\overline\cV_M^{\,m}} u_{\Theta}\,
\rho_{n}(d\Theta)}+ o(1), \qquad n\to\infty,
\end{align}
where we use that $\lim_{n\to\infty} \frac1n \log c_n=0$. In what follows, we abbreviate 
the first term in the right-hand side of the last display by $l_n$. We want to show 
that $\limsup_{n\to \infty} \tfrac1n \log Z_{3,n,L_n}^{\Omega}(M,m)$ $\leq f(M,m;\alpha,\beta)$. 
To that end, we assume that $\tfrac1n \log Z_{3,n,L_n}^{\Omega}(M,m)$ converges to some
$t\in \R$ and we prove that $t\leq f(M,m;\alpha,\beta)$. Since $(l_n)_{n\in\N}$ is bounded 
and $\overline \cV_M^{\,m}$ is compact, it follows from the definition of $l_n$ that along an appropriate subsequence both $l_n \to l_\infty \geq t$ and $\rho_n\to\rho_\infty\in
\cR_{p,M}^{\,m}$ as $n\to\infty$.  Hence, the proof will be complete once we show that 
\be{eqfin}
l_\infty\leq \sup_{u\in \cB_{\,\overline \cV_M^{\,m}}} V(\rho_{\infty},u),
\ee
because the right-hand side in \eqref{eqfin} is bounded from above by $f(M,m;\alpha,\beta)$. 

Recall \eqref{deft} and, for $\Theta\in \overline\cV_M^{\,m}$ and $y\in \R$, define
\be{defdef}
u_\Theta^{M,m}(y)= \left\{
\begin{array}{ll}
\vspace{.1cm}
t_\Theta
& \mbox{if } \partial^+_u (u\,\psi(\Theta,u))(t_\Theta) \leq y, \\
\vspace{.1cm}
m   
& \mbox{if } \partial^-_u (u\,\psi(\Theta,u))(m) \geq y,\\
z
&\mbox{otherwise, with } z \mbox{ such that } \partial^-_u (u\,\psi(\Theta,u))(z) 
\geq y \geq \partial^+_u (u\,\psi(\Theta,u))(z), 
\end{array}
\right.
\ee
where $z$ is unique by strict concavity of $u\to u\psi(\Theta,u)$ (see Lemma~\ref{B.2}). 

\begin{lemma}
\label{conti}
(i) For all $y\in \R$ and $(M,m)\in \EIGH$, $\Theta \mapsto u_\Theta^{M,m}(y)$ is 
continuous on $(\overline \cV_M^{\,m},d_M)$, where $d_M$ is defined in \eqref{dist} in 
Appendix~{\rm \ref{B}}.\\
(ii) For all $(M,m)\in\EIGH$ and $\Theta\in\overline{\cV}_M^{\,m}$, $y \mapsto
u_\Theta^{M,m}(y)$ is continuous on $\R$.
\end{lemma}

\begin{proof}
The proof uses the strict concavity of $u\to u\psi(\Theta,u)$ (see Lemma~\ref{B.2}). 

\medskip\noindent
(i) The proof is by contradiction. Pick $y\in \R$, and pick a sequence $(\Theta_n)_{n
\in \N}$ in $\overline\cV_M^{\,m}$ such that $\lim_{n\to\infty} \Theta_n=\Theta_\infty
\in \overline\cV_M^{\,m}$. Suppose that $u_{\Theta_n}^{M,m}(y)$ does not tend to 
$u_{\Theta_\infty}^{M,m}(y)$ as $n\to \infty$. Then, by choosing an appropriate 
subsequence, we may assume that $\lim_{n\to \infty} u_{\Theta_n}^{M,m}(y) = u_1 
\in [t_{\Theta_\infty},m]$ with $u_1<u_{\Theta_\infty}^{M,m}(y)$. The case 
$u_1>u_{\Theta_\infty}^{M,m}(y)$ can be handled similarly.

Pick $u_2\in (u_1,u_{\Theta_\infty}^{M,m}(y))$. For $n$ large enough, we have 
$u_{\Theta_n}^{M,m}(y)< u_2<u_{\Theta_\infty}^{M,m}(y)$. By the definition of 
$u_{\Theta_n}^{M,m}(y)$ in \eqref{defdef} and the strict concavity of $u\mapsto
u \psi(\Theta_n,u)$ we have, for $n$ large enough,
\be{inegstrict}
\partial^+_u (u\,\psi(\Theta_n,u))(u_{\Theta_n}^{M,m}(y))
> \frac{u_{\Theta_\infty}^{M,m}(y)\psi(\Theta_n,u_{\Theta_\infty}^{M,m}(y))
-u_2\psi(\Theta_n,u_2)}{u_{\Theta_\infty}^{M,m}(y)-u_2}.
\ee
Let $n\to \infty$ in \eqref{inegstrict} and use the strict concavity once again, 
to get
\begin{align}
\label{inegstrict1}
\liminf_{n\to \infty} \partial^+_u (u\,\psi(\Theta_n,u))
(u_{\Theta_n}^{M,m}(y))&>\partial^-_u (u\,\psi(\Theta_\infty,u))
(u_{\Theta_\infty}^{M,m}(y)).
\end{align}
If $u_{\Theta_\infty}^{M,m}(y)\in (t_{\Theta_\infty},m]$, then \eqref{defdef} implies 
that the right-hand side of \eqref{inegstrict1} is not smaller than $y$. Hence 
\eqref{inegstrict1} yields that $\partial^+_u (u\,\psi(\Theta_n,u))(u_{\Theta_n}^{M,m}(y))
>y$ for $n$ large enough, which implies that $u_{\Theta_n}^{M,m}(y)=m$ by \eqref{defdef}. 
However, the latter inequality contradicts the fact that $u_{\Theta_n}^{M,m}(y)<u_2
<u_{\Theta_\infty}^{M,m}(y)$ for $n$ large enough. If $u_{\Theta_\infty}^{M,m}(y)
= t_{\Theta_\infty}$, then we note that $\lim_{n\to \infty} t_{\Theta_n}
=t_{\Theta_\infty}$, which again contradicts that $t_{\Theta_n}\leq u_{\Theta_n}^{M,m}(y) 
< u_2<u_{\Theta_\infty}^{M,m}(y)$ for $n$ large enough.

\medskip\noindent
(ii) The proof is again by contradiction. Pick $\Theta \in \overline{\cV}_M^{\,m}$, and pick
an infinite sequence $(y_n)_{n\in \N}$ such that $\lim_{n\to\infty} y_n=y_\infty\in \R$ 
and such that $u_{\Theta}^{M,m}(y_n)$ does not converge to $u_{\Theta}^{M,m}(y_\infty)$. Then, 
by choosing an appropriate subsequence, we may assume that there exists a $u_1 < 
u_{\Theta}^{M,m}(y_\infty)$ such that $\lim_{n\to\infty}u_{\Theta}^{M,m}(y_n)=u_1$. The case
$u_1>u_{\Theta}^{M,m}(y_\infty)$ can be treated similarly. 

Pick $u_2,u_3\in (u_1, u_{\Theta}^{M,m}(y_\infty))$ such that $u_2<u_3$. Then, for $n$ large 
enough, we have 
\be{relat}
t_\Theta\leq u_{\Theta}^{M,m}(y_n)<u_2<u_3<u_{\Theta}^{M,m}(y_\infty)\leq m.
\ee
Combining  \eqref{defdef} and \eqref{relat} with the strict concavity of $u\mapsto 
u\psi(\Theta,u)$ we get, for $n$ large enough, 
\be{relat1}
y_n>\partial^+_u (u\,\psi(\Theta,u))(u_2) >\partial^-_u (u\,\psi(\Theta,u))(u_3)
> y_\infty,
\ee
which contradicts $\lim_{n\to \infty} y_n=y_\infty$.
\end{proof}

We resume the line of proof. Recall that $\rho_{n,1}$, $n\in \N$, charges finitely 
many $\Theta\in\overline{\cV}_M^{\,m}$. Therefore the continuity and the strict concavity 
of $u\mapsto u\psi(\Theta,u)$ on $[t_\Theta,m]$ for all $\Theta\in \overline{\cV}_M^{\,m}$ 
(see Lemma \ref{concav}) imply that the supremum in \eqref{immp} is attained at some
$u_{n}^{M,m}\in \cB_{\,\overline{\cV}_M^{\,m}}$ that satisfies $u_{n}^{M,m}(\Theta) 
= u_\Theta^{M,m}(l_n)$ for $\Theta\in \overline{\cV}_M^{\,m}$. Set $u^{M,m}_{\infty}
(\Theta) = u^{M,m}_{\Theta}(l_\infty)$ for $\Theta\in \overline{\cV}_M^{\,m}$ and note 
that $(l_n)_{n\in \N}$ may be assumed to be monotone, say, non-decreasing. Then the 
concavity of $u\mapsto u\psi(\Theta,u)$ for $\Theta\in \overline{\cV}_M^{\,m}$ implies 
that $(u_n^{M,m})_{n\in\N}$ is a non-increasing sequence of functions on $\overline
\cV_M^{\,m}$. Moreover, $\overline{\cV}_M^{\,m}$ is a compact set and, by Lemma
\ref{conti}(ii), $\lim_{n\to \infty} u_n^{M,m}(\Theta)=u_\infty^{M,m}(\Theta)$ 
for $\Theta\in\overline{\cV}_M^{\,m}$. Therefore Dini's theorem implies that 
$\lim_{n\to\infty} u_n^{M,m}=u_\infty^{M,m}$ uniformly on $\overline{\cV}_M^{\,m}$. 
We estimate
\begin{align}
\label{inegd}
\nonumber&\left|l_n-\int_{\overline\cV_M^{\,m}}  
u^{M,m}_{\infty}(\Theta)\,\psi(\Theta,u^{M,m}_{\infty}(\Theta))
\rho_{\infty}(d\Theta)\right|\\
 &\qquad\leq \int_{\overline\cV_M^{\,m}} \Big|u^{M,m}_{n}(\Theta)\,
\psi(\Theta,u^{M,m}_{n}(\Theta))-u^{M,m}_{\infty}(\Theta)\,
\psi(\Theta,u^{M,m}_{\infty}(\Theta))\Big|\,\rho_{n}(d\Theta)\\
\nonumber&\qquad +\Big|\int_{\overline\cV_M^{\,m}} u^{M,m}_{\infty}(\Theta)\, 
\psi(\Theta,u^{M,m}_{\infty}(\Theta))\,\rho_{n}(d\Theta)
-\int_{\overline\cV_M^{m}} u^{M,m}_{\infty}(\Theta)\,
\psi(\Theta,u^{M,m}_{\infty}(\Theta))\, \rho_{\infty}(d\Theta)\Big|.
\end{align}
The second term in the right-hand side of \eqref{inegd} tends to zero as $n\to\infty$ 
because, by Lemma \ref{conti}(i), $\Theta \mapsto u_{\infty}^{M,m}(\Theta)$ is continuous 
on $\overline{\cV}_M^{\,m}$ and because $\rho_n$ converges in law to $\rho_\infty$ as 
$n\to\infty$. The first term in the right-hand side of \eqref{inegd} tends to zero as 
well, because $(\Theta,u)\mapsto u\psi(\Theta,u)$ is uniformly continuous on 
$\overline{\cV}^{\,*,m}_M$ (see Lemma~\ref{concavt}) and because we have proved above 
that $u^{M,m}_{n}$ converges to $u^{M,m}_{\infty}$ uniformly on $\overline\cV_M^{\,m}$. 
This proves \eqref{eqfin}, and so Step 3 is complete.

%%%%%%%%%%%%%%%%%%%%%%%%%%%%

\subsubsection{Step 4}
\label{s24alt}

In this step  we prove that 
\be{limsuplar}
\limsup_{n\to \infty} f_{3,n}^{\Omega}(M,m;\alpha,\beta)\geq f(M,m;\alpha,\beta)\, 
\text{ for } \P-a.e.\ \Omega.
\ee
Note that the proof will be complete once we show that 
\be{pr}
\limsup_{n\to \infty} f_{3,n}^\Omega(M,m,
\alpha,\beta)\geq V(\rho,u) \ \text{for} \ \rho\in \cR_{p,M}^m,  
u\in \cB_{\,\overline \cV_M^{\,m}}. 
\ee
Pick $\Omega\in \{A,B\}^{\N_0\times \Z}$, $\rho\in \cR^{\Omega,m}_{p,M}$ and $u\in
\cB_{\,\overline \cV_M^{\,m}}$. By the definition of $\cR_{p,M}^{\Omega,m}$, there exists a 
strictly increasing subsequence $(n_k)_{k\in \N}\in \N^\N$ such that, for all 
$k\in \N$, there exists an 
\be{Nkexist}
N_k\in\left\{\frac{n_k}{m L_{n_k}},\dots,
\frac{n_k}{L_{n_k}}\right\},
\ee
a $\Theta_{\text{traj}}^k\in \widetilde{\cD}_{L_{n_k},N_k}^M$ and a $x^k\in
\cX_{\Theta_{\text{traj}}^k,\Omega}^{M,m}$ such that $\rho_k=^\mathrm{def}
\rho_{\Theta_{\text{traj}}^k,x^k}^\Omega$ (see \eqref{defrho2}) converges in 
law to $\rho$ as $k\to\infty$. Recall \eqref{defD2}, and note that 
\begin{align}
\Xi_j^k
=\big(\Delta\Pi^k_j,\,b^k_j,\,b_{j+1}^k\big),
\quad \text{$j=0,\dots,N_k-1$},
\end{align}
with $\Delta\Pi^k_j\in \{-M,\dots,M\}$ and $b_j^k\in (0,1]
\cap\frac{\N}{L_{n_k}}$ for $j=0,\dots,N_k$. For ease of notation we define 
\be{notat}
\Theta_{j}^k
=\big(\Omega(j,\Pi^k_j+\cdot),\Xi_j^k,x_j^k\big)\quad \text{with} \quad 
\Pi_j^k=\sum_{i=0}^{j-1} \Delta\Pi_i^k, \qquad  j=0,\dots,N_k-1,
\ee
and
\be{defv}
v_k=N_k \int_{\Theta\in \cV_M^{\,m}} u_\Theta\,\rho_{k,1}(d\Theta)
= \sum_{j=0}^{N_k-1} u_{\Theta_j^k},
\ee
where we recall that $u=(u_\Theta)_{\Theta\in \overline\cV_M^{\,m}}$ was fixed at the 
beginning of the section.

Next, we recall that $\lim_{n\to\infty} n/L_n=\infty$ and that $L_n$ is non-decreasing. 
Together with the fact that $\lim_{n\to\infty} L_n/n=0$, this implies that $L_n$ is 
constant on intervals. On those intervals, $n/L_n$ takes constant increments. The latter 
implies that there exists an $\widetilde{n}_k\in \N$ satisfying
\be{deftil}
0\leq v_k-\tfrac{\widetilde{n}_k}{L_{\widetilde{n}_k}} 
\leq \tfrac{1}{L_{\widetilde{n}_k}}
\quad \text {and therefore}\quad 
0\leq  v_k L_{\widetilde{n}_k}-\widetilde{n}_k \leq 1.
\ee
Next, for $j=0,\dots,N_k-1$ we pick $\overline{b_j^k}\in (0,1] \cap 
\frac{N}{L_{\widetilde{n}_k}}$ such that $|\overline{b_j^k}-b_j^k|\leq 
\tfrac{1}{L_{\widetilde{n}_k}}$, define
\be{defbar}
\overline{\Xi_j^k} = \big(\Delta \Pi^k_j,
\overline{b_j^k},\overline{b_{j+1}^k}\,\big), 
\quad \overline{\Theta_j^k}
= \big(\Omega(j,\Pi^k_j+\cdot),\overline{\Xi_j^k},x_j^k\,\big),
\ee
and pick
\be{defs}
s_j^k\in t_{\,\overline{\Theta_j^k}}+\frac{2\N}{L_{\widetilde{n}_k}}\quad
\text{such that} \quad |s_j^k-u_{\Theta_j^k}|\leq 2/L_{\widetilde{n}_k}.
\ee
We use \eqref{defv} to write
\be{estis}
L_{\widetilde{n}_k} \sum_{j=0}^{N_k-1} s^k_j=L_{\widetilde{n}_k} 
\bigg(v_k +\sum_{j=0}^{N_k-1} (s_j^k-u_{\Theta_j^k})\bigg)
= L_{\widetilde{n}_k} (I+II).
\ee
Next, we note that \eqref{deftil} and \eqref{defs} imply that $|L_{\widetilde{n}_k} 
I-\widetilde{n}_k|\leq 1$ and $|L_{\widetilde{n}_k} II|\leq 2 N_k$. The latter 
in turn implies that, by adding or subtracting at most 3 steps per colum, the 
quantities $s_j^k$ for $j=0,\dots,N_k-1$ can be chosen in such a way that 
$\sum_{j=0}^{N_k-1} s_j^k=\widetilde{n}_k/L_{\widetilde{n}_k}$.

Next, set 
\be{}
\overline{\Theta_{\text{traj}}^k}=(\overline{\Xi_j^k})_{j=0}^{N_k-1}
\in \widetilde{\cD}_{L_{\widetilde{n}_k},N_k }^M,
\quad s^k=(s^k_j)_{j=0}^{N_k-1}
\in \cU_{\overline{\Theta_{\text{traj}}^k},\,
x^k,\widetilde{n}_k}^{M,m,L_{\widetilde{n}_k}},
\ee
and recall \eqref{inteerm1} to get $f_3^\Omega(\widetilde{n}_k,M)\geq R_k$ with
\be{rk}
R_k=\frac{L_{\widetilde{n}_k} \,
H^\Omega\big(\,\overline{\Theta_{\text{traj}}^k},x^k,s^k\,\big)}
{\widetilde{n}_k}= \frac{\sum_{j=0}^{N_k-1}\, s_j^k\,
\psi\Big(\overline{\Theta_j^k},\,s_j^k\Big)}{\sum_{j=0}^{N_k-1} s_j^k}
=\frac{R_{\text{nu}}^k}{R_{\text{de}}^k}.
\ee
Further set
\be{defR} 
R^{'}_k=\frac{R^{' k}_{\text{nu}}}{R^{' k}_{\text{de}}}
=\frac{\int_{\overline\cV_M^{\,m}} u_\Theta\, \psi(\Theta,u_\Theta) 
\rho_{k}(d\Theta)}{\int_{\overline\cV_M^{\,m}} u_\Theta \,\rho_{k}(d\Theta)},
\ee
and note that $\lim_{k\to\infty} R^{'}_k=V(\rho,u)$, since $\lim_{k\to\infty}
\rho_k=\rho$ by assumption and $\Theta\mapsto u_{\Theta}$ is continuous on 
$\cV_M^{\,m}$. We note that $R^{'}_k$ can be rewritten in the form 
\be{defRalt} 
R^{'}_k=\frac{R^{' k}_{\text{nu}}}{R^{' k}_{\text{de}}}
= \frac{\sum_{j=0}^{N_k-1}\, u_{\Theta_j^k}\,
\psi\big(\Theta_j^k,\,u_{\Theta_j^k}\big)}
{\sum_{j=0}^{N_k-1}\, u_{\Theta_j^k}}.
\ee
Now recall that $\lim_{k\to \infty} n_k=\infty$. Since $N_k\geq n_k/M L_{n_k}$, it 
follows that $\lim_{k\to \infty} N_k=\infty$ as well. Moreover, $N_k\leq\widetilde{n}_k
/L_{\widetilde{n}_k}$ with $\lim_{k\to\infty} \widetilde{n}_k=\infty$. Therefore 
(\ref{defv}--\ref{deftil}) allow us to conclude that $R_{\text{de}}^k
=\widetilde{n}_k/L_{\widetilde{n}_k}=R^{'k}_{\text{de}}[1+o(1)]$.

Next, note that $\cH_M$ is compact, and that $(\Theta,u)\mapsto u \psi(\Theta,u)$ is 
continuous on $\cH_M$ and therefore is uniformly continuous. Consequently, for 
all $\gep>0$ there exists an $\eta>0$ such that, for all $(\Theta,u),(\Theta^{'},u^{'})
\in \cH_M$ satisfying $|\Theta-\Theta^{'}|\leq \eta$ and $|u-u^{'}|\leq \eta$,
\be{unifcont}
|u\psi(\Theta,u)-u^{'}\psi(\Theta^{'},u^{'})|\leq \gep.
\ee
We recall \eqref{defbar}, which implies that $d_M(\overline{\Theta_j^k},\Theta_j)
\leq 2/L_{\widetilde{n}_k}$ for all $j\in \{0,\dots,N_k-1\}$, we choose $k$ large enough to 
ensure that $2/L_{\widetilde{n}_k}\leq \eta$, and we use \eqref{unifcont}, to obtain
\be{equa} 
R^{k}_{\text{nu}}=\sum_{j=0}^{N_k-1}\, 
s_j^k\, \psi\Big(\overline{\Theta_j^k},\,s_j^k\Big)
= \sum_{j=0}^{N_k-1}\, u_{\Theta_j^k}\, 
\psi\big(\Theta_j^k,\,u_{\Theta_j^k}\big)+T=R^{' k}_{\text{nu}}+T,
\ee
with $|T| \leq \gep N_k$. Since $\lim_{k\to\infty} R^{'}_k=V(\rho,u)$ and 
$\sum_{j=0}^{N_k-1} u_{\Theta_j^k}=v_k\geq \widetilde{n}_k/L_{\widetilde{n}_k}$ 
(see \eqref{deftil}), if $ V(\rho,u)\neq 0$, then $\big|R^{' k}_{\text{nu}}\big|
\geq \mathrm{Cst.}\,\,\widetilde{n}_k/L_{\widetilde{n}_k}$, whereas $|T|\leq \gep N_k
\leq  \gep \widetilde{n}_k/L_{\widetilde{n}_k} $ for $k$ large enough. Hence 
$T=o(R^{' k}_{\text{nu}})$ and 
\be{fina}
\frac{R^{k}_{\text{nu}}}{R^{k}_{\text{de}}}
= \frac{R^{' k}_{\text{nu}}\, [1+o(1)]}{R^{'k}_{\text{de}}\,
[1+o(1)]}\to V(\rho,u), \qquad k\to \infty. 
\ee
Finally, if $ V(\rho,u)= 0$, then $R^{' k}_{\text{nu}}=o(R^{'k}_{\text{de}})$ and 
$T=o(R^{'k}_{\text{de}})$, so that $R_k$ tends to $0$. This completes the proof
of Step 4.

%%%%%%%%%%%%%%%%%%%%%%%%%%%%%%%

\subsubsection{Step 5}
\label{s23alt}

In this step we prove \eqref{convco1}, suppressing the $(\alpha,\beta)$-dependence 
from the notation. For $\Omega\in \{A,B\}^{\N_0\times \Z^2}$, $n\in \N$, $N\in
\{n/m L_n,\dots,n/L_n\}$ and $r\in \{-N M,\dots,N M\}$, we recall \eqref{defD2} and 
define
\be{defDD2alt}
\widetilde{\cD}_{L,N}^{M,m,r}= \Big\{\Theta_{\text{traj}}
\in \widetilde{\cD}_{L,N}^{M,m}\colon \Pi_N=r\Big\},
\ee
where we recall that $\Pi_N=\sum_{j=0}^{N-1} \Delta\Pi_j$. We set 
\begin{equation}
\label{intermalt}
f_{3,n}^\Omega(M,m,N,r) = \tfrac{1}{n} \log Z_{3,n,L_n}^{\Omega}(N,M,m,r)
\ee
with
\be{iiiter}
Z_{3,n,L_n}^{\Omega}(N,M,m,r)
= \sum_{\Theta_{\text{traj}}\in \widetilde{\cD}_{L_n,N}^{M,m,r} }
\,\sum_{x\in \cX_{\Theta_{\text{traj}},\Omega}^{M,m} }
\,\sum_{u\in \,\cU_{\Theta_{\text{traj}},n}^{\,M,m,L_n}} A_3,
\end{equation}
where $A_3$ is defined in \eqref{defa23}. We further set $f_3(\cdot)=\E_\Omega
\big(f_3^\Omega(\cdot)\big)$.

%%%%%%%%%%%%%%%%%%%%%%%%%%%%%%%%%%%%%%

\subsubsection{Concentration of measure} 

In the first part of this step we prove that  for all $(M,m,\alpha,\beta)\in
\EIGH\times \CONE$ there exist $c_1,c_2>0$ (depending on $(M,m,\alpha,\beta)$ only) 
such that, for all  $n\in \N$,  $N\in \{n/(m L_n),\dots n/L_n\}$ and $r\in
\{-N M,\dots,N M\}$, 
\begin{align}
\label{conco}
&\P_{\Omega}\big(\big|f_{3,n}^{\Omega}(M,m)-f_{3,n}(M,m)\big|>\gep\big)
\leq c_1\  e^{-\frac{c_2 \gep^2 n}{L_n}}, \qquad \\
\nonumber 
&\P_{\Omega}\big(\big|f_{3,n}^{\Omega}(M,m,N,r)-f_{3,n}(M,m,N,r)\big|>\gep\big)
\leq c_1\  e^{-\frac{c_2 \gep^2  n}{L_n}}.
\end{align} 
We only give the proof of the first inequality. The second inequality is proved 
in a similar manner. The proof uses Theorem~\ref{theoco}. Before we start we note 
that, for all $n\in \N$, $(M,m)\in \EIGH$ and $\Omega\in \{A,B\}^{\N_0\times \Z}$, 
$f_{3,n}^\Omega(M,m)$ only depends on 
\be{defC}
\cC_{0,L_n}^\Omega,\dots,\cC_{n/L_n,L_n}^\Omega \quad \text{with}\quad 
\cC_{j,L_n}^\Omega=(\Omega(j,i))_{i=-n/L_n}^{n/L_n}.
\ee
We apply Theorem~\ref{theoco} with $\cS=\{0,\dots,n/L_n\}$, with  $X_i=\{A,B\}^{\{
-\frac{n}{L_n},\dots,\frac{n}{L_n}\}}$ and with $\mu_i$ the uniform measure on 
$X_i$ for all $i\in \cS$. Note that $|f_{3,n}^{\Omega_1}(M,m)-f_{3,n}^{\Omega_2}(M,m)|
\leq 2 C_{\text{uf}}(\alpha) m \tfrac{L_n}{n}$ for all $i\in \cS$ and all 
$\Omega_1,\Omega_2$ satisfying $\cC_{j,n}^{\Omega_1}=\cC_{j,n}^{\Omega_2}$ 
for all $j\neq i$. After we set $c=2 C_{\text{uf}}(\alpha) m$ we can apply 
Theorem~\ref{theoco} with $D=c^2 L_n/n$ to get \eqref{conco}. 

Next, we note that the first inequality in \eqref{conco}, the Borel-Cantelli lemma 
and the fact that $\lim_{n\to \infty} n/L_n \log n = \infty$ imply that, for all 
$(M,m)\in \EIGH$,
\be{limialt}
\lim_{n\to \infty} \Big[f_{3,n}^\Omega(M,m)- f_{3,n}(M,m)\Big]=0 
\quad \text{for } \P-a.e.~\Omega.
\ee 
Therefore \eqref{convco1} will be proved once we show that 
\be{fin1}
\liminf_{n\to \infty} f_{3,n}(M,m)
= \limsup_{n\to \infty} f_{3,n}(M,m). 
\ee
To that end, we first prove that, for all $n\in \N$ and all $(M,m)\in \EIGH$, there exist 
an $N_n\in \{n/m L_n,\dots,n/L_n\}$ and an $r_n\in \{-M N_n,\dots,M N_n\}$ such that 
\be{fin2}
\lim_{n\to \infty} \Big[f_{3,n}(M,m)-f_{3,n}(M,m,N_n,r_n)\Big] = 0.
\ee
The proof of \eqref{fin2} is done as follows. Pick $\gep>0$, and for $\Omega\in
\{A,B\}^{\N_0\times \Z}$, $n\in \N$ and $(M,m)\in \EIGH$, denote by $N_n^\Omega$ 
and $r_n^{\Omega}$ the maximizers of $f^\Omega_{3,n}(M,m,N,r)$. Then 
\be{rrde}
f^\Omega_{3,n} \big(M,m,N_n^{\Omega},r_n^{\Omega}\big)\leq f^\Omega_{3,n}(M,m)
\leq  \tfrac1n \log (\tfrac{n^2}{L_n^2})
+ f^\Omega_{3,n}\big(M,m,N_n^{\Omega},r_n^{\Omega}\big),
\ee
so that, for $n$ large enough and every $\Omega$, 
\be{recty}
0\leq f^\Omega_{3,n}(M,m)-f^\Omega_{3,n}\big(M,m,N_n^{\Omega},r_n^{\Omega}\big)\leq \gep.
\ee
For $n\in \N$, $N\in \{n/m L_n,\dots, n/L_n\}$ and $r\in \{-N M,\dots, N M\},$ we set
\be{defAnNr}
A_{n,N,r}=\{\Omega\colon (N_n^\Omega,r_n^\Omega)=(N,r)\}.
\ee
Next, denote by $N_n,r_n$ the maximizers of $\P(A_{n,N,r})$. Note that \eqref{fin2} 
will be proved once we show that, for all $\gep>0$, $|f_{3,n}(M,m)-f_{3,n}(M,m,N_n,r_n)|\leq
\gep$ for $n$ large enough. Further note that $\P(A_{n,N_n,r_n})\geq L_n^2/n^2$ for 
all $n\in \N$. For every $\Omega$ we can therefore estimate
\be{iimp1}
|f_{3,n}(M,m)-f_{3,n}(M,m,N_n,r_n)|\leq I+II+III
\ee 
with
\begin{align}
\label{iimp}
&I=|f_{3,n}(M,m)-f_{3,n}^\Omega(M,m)|,\\
\nonumber & II=|f_{3,n}^\Omega(M,m)-f_{3,n}^\Omega(M,m,N_n,r_n)|, \\
\nonumber &III=|f_{3,n}^\Omega(M,m,N_n,r_n)-f_{3,n}(M,m,N_n,r_n)|.
\end{align}
Hence, the proof of \eqref{fin2} will be complete once we show that, for $n$ large 
enough, there exists an $\Omega_{\gep,n}$ for which $I,II$ and $III$ in \eqref{iimp} 
are bounded from above by $\gep/3$.

To that end, note that, because of \eqref{conco}, the probabilities $\P(\{I> \gep/3\})$ 
and $\P(\{III> \gep/3\})$ are bounded from above by $c_1 e^{-c_2 \gep^2 n/9 L_n}$, 
while 
\be{boundA}
\P(\{II>\gep\})\leq \P(A_{n,N_n,r_n}^c)\leq 1-(L_n^2/n^2), \qquad n\in \N.
\ee
Since $\lim_{n\to \infty} n/L_n \log n=\infty$, we have $\P(\{I,II,III \leq \gep/3\})>0$ 
for $n$ large enough. Consequently, the set $\{I,II,III\leq \gep/3\}$ is non-empty 
and \eqref{fin2} is proven.

%%%%%%%%%%%%%%%%%%%%%%%%%%%%%%%%%%%%%%%%%%%%%%%%%%%

\subsubsection{Convergence}

It remains to prove \eqref{fin1}. Assume that there exist two strictly increasing subsequences 
$(n_k)_{k\in\N}$ and $(t_k)_{k\in\N}$ and two limits $l_2>l_1$ such that $\lim_{k\to\infty} 
f_{3,n_k}(M,m)=l_2$ and $\lim_{k\to\infty}$ $f_{3,t_k}(M,m) =l_1$. By using \eqref{fin2}, we 
have that for every $k\in\N$ there exist $N_k\in\{n_k/m L_{n_k},$ $\dots,n_k/L_{n_k}\}$ and 
$r_k\in \{-M N_k,\dots,M N_k\}$ such that $\lim_{k\to\infty}$ $f_{3,n_k}(M,m,N_k,r_k)
=l_2$. Denote by 
\be{tdt}
(\Theta_{\text{traj,max}}^{k,\Omega},
x_{\text{max}}^{k,\Omega},u_{\text{max}}^{k,\Omega})\in
\widetilde{\cD}_{L_{n_k},N_k}^{M,r_k}\times 
\cX_{\Theta_{\text{traj,max}}^{k,\Omega},\Omega}^{M,m}
\times  \cU_{\Theta_{\text{traj,max}}^{k,\Omega},
x_{\text{max}}^{k,\Omega},n_k}^{\,M,m,L_n}
\ee
the maximizer of $H^\Omega(\Theta_{\text{traj}},x,u)$. We recall that $\Theta_{\text{traj}},
x$ and $u$ take their values in sets that grow subexponentially fast in $n_k$, and therefore
\be{equan}
\lim_{k\to \infty} \tfrac{L_{n_k}}{n_k}\,
\E_\Omega\big[H^\Omega(\Theta_{\text{traj,max}}^{k,\Omega},x_{\text{max}}^{k,\Omega},
u_{\text{max}}^{k,\Omega})\big]=l_2.
\ee
Since $l_2>l_1$, we can use \eqref{equan} and the fact that $\lim_{k\to \infty} n_k/L_{n_k}
=\infty$ to obtain, for $k$ large enough, 
\be{equan1}
\E_\Omega\big[H^\Omega(\Theta_{\text{traj,max}}^{k,\Omega},
x_{\text{max}}^{k,\Omega},u_{\text{max}}^{k,\Omega})\big]
+(\beta-\alpha)\geq \tfrac{n_k}{L_{n_k}} \big(l_1+\tfrac{l_2-l_1}{2}\big).
\ee
(The term $\beta-\alpha$ in the left-hand side of \eqref{equan1} is introduced for 
later convenience only.) Next, pick $k_0\in \N$ satisfying \eqref{equan1}, whose value 
will be specified later. Similarly to what we did in \eqref{defs} and \eqref{estis}, 
for $\Omega\in \{A,B\}^{\N_0\times \Z}$ and $k\in \N$ we associate with 
\be{asso}
\Theta_{\text{traj,max}}^{k_0,\Omega}
=\big(\Delta\Pi_{\,j}^{k_0,\Omega}, b^{k_0,\Omega}_{\,0,j}, 
b_{\,1,j}^{k_0,\Omega}\big)_{j=0}^{N_{k_0}-1}\in 
\widetilde{\cD}_{L_{n_{k_0}},N_{k_0}}^{M,r_{k_0}}
\ee
and 
\be{assocc}
x_{\text{max}}^{k_0,\Omega}=\big(x_{\,j}^{k_0,\Omega}\big)_{j=0}^{N_{k_0}-1}
\in \cX_{\Theta_{\text{traj,max}}^{k_0,\Omega},\Omega}^{M,m}
\ee
and 
\be{assoc}
u_{\text{max}}^{k_0,\Omega}=\big(u_{\,j}^{k_0,\Omega}\big)_{j=0}^{N_{k_0}-1}
\in \cU_{\Theta_{\text{traj,max}}^{k_0,\Omega},
x_{\text{max}}^{k_0,\Omega},n_{k_0}}^{M,m,L_{n_{k_0}}}
\ee
the quantities
\be{defTh}
\overline\Theta^{\,k,\Omega}_{\text{traj}}
= \big(\Delta\Pi_{\,j}^{k_0,\Omega}, \overline b_{\,0,j}^{\,k,\,\Omega}, 
\overline b_{\,1,j}^{\,k,\,\Omega}\big)_{j=0}^{N_{k_0}-1}
\in \widetilde{\cD}_{L_{t_{k}},N_{k_0}}^{M,r_{k_0}}
\ee  
and 
\be{assoc1}
\overline u^{\,k,\Omega}=\big(\overline u_{\,j}^{\,k,\Omega}\big)_{j=0}^{N_{k_0}-1}
\in \cU_{\overline\Theta_{\text{traj}}^{\,k,\Omega},
x_{\text{max}}^{k_0,\Omega},*}^{M,m,L_{t_{k}}}
\ee
(where $*$ will be specified later), so that 
\be{supe}
\big|\overline b_{\,0,j}^{\,k,\,\Omega}-b_{\,0,j}^{k_0,\Omega}
\big|\leq \tfrac{1}{L_{t_k}},\ \ 
\big|\overline b_{\,1,j}^{\,k,\,\Omega}-b_{\,1,j}^{k_0,\Omega}\big|
\leq \tfrac{1}{L_{t_k}},
\ \ \ \ \big|\overline u_{\,j}^{\,k,\,\Omega}-u_{\,j}^{k_0,\Omega}\big|
\leq \tfrac{2}{L_{t_k}}, 
\ \ \ j=0,\dots,N_{k_0}-1.
\ee
Next, put $\overline s_{k}^{\,\Omega}=L_{t_k} \sum_{j=0}^{N_{k_0}-1} 
\overline u_j^{\,k,\,\Omega}$, which we substitute for $*$ above. The uniform 
continuity in Lemma~\ref{concavt} allows us to claim that, for $k$ large enough 
and for all $\Omega$,
\be{trt}
\Big|\overline u_j^{\,k,\,\Omega} \ \psi\Big(\overline\Theta_j^{\,k,\,\Omega},
\overline u_j^{\,k,\,\Omega}\Big)
-u_{\,j}^{k_0,\Omega}\  \psi\Big(\Theta_{\,j}^{k_0,\Omega},
u_{\,j}^{k_0,\Omega}\Big)\Big|\leq \tfrac{l_2-l_1}{4},
\ee
where we recall that, as in \eqref{notat}, for all $j=0,\dots,N_{k_0}-1$,  
\begin{align}
\label{restr}
\overline\Theta_j^{\,k,\,\Omega}
&=\Big(\Omega\big(j,\Pi^{k_0,\Omega}_{\,j}+\cdot\big),\,\Delta\Pi_{\,j}^{k_0,\Omega}, 
\,\overline b_{\,0,j}^{\,k,\,\Omega},\, 
\overline b_{\,1,j}^{\,k,\,\Omega},\,x_j^{k_0,\Omega}\Big),\\
\nonumber \Theta_{\,j}^{k_0,\Omega}
&=\Big(\Omega\big(j,\Pi^{k_0,\Omega}_{\,j}+\cdot\big),\,\Delta\Pi_j^{k_0,\Omega}, 
\,b_{\,0,j}^{k_0,\Omega},\, b_{\,1,j}^{k_0,\Omega},\,x_j^{k_0,\Omega}\Big).
\end{align}
Recall \eqref{defHam}. An immediate consequence of \eqref{trt} is that 
\be{reft}
\big|H^\Omega(\overline\Theta_{\text{traj}}^{\,k,\Omega},
x_{\text{max}}^{k_0,\Omega},\overline u^{\,k,\Omega})
-H^\Omega(\Theta_{\text{traj,max}}^{k_0,\Omega},
x_{\text{max}}^{k_0,\Omega},u_{\text{max}}^{k_0,\Omega})\big|
\leq N_{k_0} \tfrac{l_2-l_1}{4}.
\ee
Hence we can use \eqref{equan1}, \eqref{reft} and the fact that $N_{k_0}\leq 
n_{k_0}/L_{n_{k_0}}$, to conclude that, for $k$ large enough,
\be{equan2}
\E_\Omega\big[H^\Omega(\overline\Theta_{\text{traj}}^{\,k,\Omega},
x_{\text{max}}^{k_0,\Omega},\overline u^{\,k,\Omega})\big]+(\beta-\alpha)
\geq \tfrac{n_{k_0}}{L_{n_{k_0}}}\big(l_1+\tfrac{l_2-l_1}{4}\big).
\ee

At this stage we add a column at the end of the group of $N_{k_0}$ columns in such 
a way that the conditions $\widehat b^{k,\Omega}_{1, N_{k_0}-1}=\widehat 
b^{k,\Omega}_{0, N_{k_0}}$ and $\widehat b^{k,\Omega}_{1, N_{k_0}}=1/L_{t_k}$ 
are satisfied. We put 
\be{defsk}
\widehat{\Xi}_{N_{k_0}}^{k,\,\Omega}
=\big(\Delta \Pi_{N_{{k_0}}}^{k_0,\Omega},\widehat b_{0,N_{k_0}}^{k,\,\Omega},
\widehat b_{1,N_{k_0}}^{k,\,\Omega}\big)=\big(0,\widehat b_{1,N_{k_0}-1}^{k,\,\Omega},
\tfrac{1}{L_{t_k}}\big),
\ee
and we let $\widehat{\Theta}_{\text{traj}}^{k,\,\Omega}\in \widetilde{\cD}_{L_{t_k},
\,N_{k_0}+1}^{M,\,r_{k_0}}$ be the concatenation of $\overline \Theta_{\text{traj}}^{k,
\Omega}$ (see \eqref{defTh}) and $\widehat{\Xi}_{N_{k_0}}^{k,\Omega}$. We let 
$\widehat{x}^{k_0,\Omega}\in\cX_{\widehat\Theta_{\text{traj}}^{k,\Omega},\Omega}^{M,m}$ 
be the concatenation of $x_{\text{max}}^{k_0,\Omega}$ and $0$. We further let 
\be{defshat}
\widehat{s}_k^{\,\Omega}=\overline s_k^{\,\Omega}+\Big[1+b_{1,N_{k_0}-1}^{k,\Omega}
-\tfrac{1}{L_{t_k}}\Big] L_{t_k},
\ee 
and we let $\widehat{u}^{k,\Omega}\in\cU_{\,\widehat{\Theta}_{\text{traj}}^{k,\
\Omega},\, \widehat x^{k_0,\Omega},\,\widehat{s}_k^{\,\Omega}}^{M,m,\,L_{t_k}}$ be 
the concatenation of $\overline u^{\,k,\Omega}$ (see \eqref{assoc1}) and
\be{defunk0}
\widehat u_{N_{k_0}}^{k,\Omega}=1+(b_{1,N_{k_0}-1}^{k,\Omega}-\tfrac{1}{L_{t_k}}).
\ee
Next, we note that the right-most inequality in \eqref{supe}, together with the fact 
that
\be{}
\sum_{j=0}^{N_{k_{0}}-1} u_{\,j}^{k_0,\Omega} = n_{k_0}/L_{n_{k_0}},
\ee 
allow us to asset that $|\overline s_k^{\,\Omega}-L_{t_k} n_{k_0}/L_{n_{k_0}}|
\leq 2 N_{k_0}$. Therefore the definition of $\widehat{s}_k^{\,\Omega}$ in 
\eqref{defshat} implies that  
\be{impo}
\widehat{s}_k^{\,\Omega}=L_{t_k} \frac{n_{k_0}}{L_{n_{k_0}}}
+\widehat{m}_k^\Omega \quad \text{with} \quad |\widehat{m}_k^{\Omega}|
\leq 2 N_{k_0}+2 L_{t_k}.
\ee
Moreover,
\be{httraj}
H^\Omega\big(\widehat{\Theta}_{\text{traj}}^{k,\Omega},
\widehat x^{k_0,\Omega},\widehat{u}^{k,\Omega}\big)
\geq H^\Omega \big(\overline\Theta_{\text{traj}}^{\,k,\Omega},
x_{\text{max}}^{k_0,\Omega}, \overline u^{\,k,\Omega}\big)
+(\beta-\alpha),
\ee
because $\widehat u_{N_{k_0}}^{k,\Omega}\leq 2$ by definition (see \eqref{defunk0}) 
and the free energies per columns are all bounded from below by $(\beta-\alpha)/2$. 
Hence, \eqref{equan2} and \eqref{httraj} give that for all $\Omega$ there exist a
\be{redefTh}
\widehat{\Theta}^{k,\Omega}_{\text{traj}}\in
\widetilde{\cD}_{L_{t_{k}},\,N_{k_0}+1}^{\,M,\,r_{k_0}}
\colon\,b_{1,N_{k_0}}=\tfrac{1}{L_{t_k}},
\ee
an $\widehat{x}^{k_0,\Omega}\in\cX_{\widehat\Theta_{\text{traj}}^{k,\Omega},\Omega}^{M,m}$
and a $\widehat{u}^{k,\Omega}\in\cU_{\,\widehat{\Theta}_{\text{traj}}^{k,\,\Omega},\,
\widehat{x}^{k_0,\Omega},\,\widehat{s}_k^{\,\Omega}}^{M,m,\,L_{t_k}}$ such that, 
for $k$ large enough,
\be{rth}
\E_{\,\Omega}\big[H(\widehat{\Theta}_{\text{traj}}^{k,\Omega},
\widehat x^{k_0,\Omega},
\widehat{u}^{k,\Omega})\big]\geq \tfrac{n_{k_0}}{L_{n_{k_0}}}
\big(l_1+\tfrac{l_2-l_1}{4}).
\ee

Next, we subdivide the disorder $\Omega$ into groups of $N_{k_0}+1$ consecutive columns 
that are successively translated by $r_{k_0}$ in the vertical direction, i.e.,
$\Omega=(\Omega_1,\Omega_2,\dots)$ with (recall \eqref{add1})
\be{defO}
\Omega_j=\big(\Omega(i,\,(j-1)\, r_{k_0}+\cdot)\big)_{i=(j-1)
(N_{k_0}+1)}^{\,j (N_{k_0}+1)-1},
\ee
and we let $q_k^\Omega$  be the unique integer satisfying 
\be{defqo}
\widehat{s}_k^{\, \Omega_1}+\widehat{s}_k^{\, \Omega_2}+ \dots +
\widehat{s}_k^{\,\Omega_{q_k}}\leq t_k
<\widehat{s}_k^{\,\Omega_1}+\dots+\widehat{s}_k^{\,\Omega_{q_k+1}},
\ee
where we suppress the $\Omega$-dependence of $q_k$. We recall that
\begin{equation}
\label{interm1alt}
f_{3,t_k}^\Omega(M,m) =
\E\Bigg[\frac{1}{t_k} \log \sum_{N=t_k/m L_{t_k}}^{t_k/L_{t_k}}\  
\sum_{\Theta_{\text{traj}}\in \widetilde{\cD}_{L_{t_k},N}^{M}}\
\,\sum_{x\in \cX_{\Theta_{\text{traj}},\Omega}^{M,m} } 
\,\sum_{u\in\,\cU_{\,\Theta_{\text{traj}},\,x,\,t_k}^{M,m,\,L_{t_k}}} e^{L_{t_k}\,
H^\Omega(\Theta_{\text{traj}},x,u)}\Bigg],
\end{equation}
set $\widetilde{t}_k^{\,\Omega}=\widehat{s}_k^{\,\Omega_1}+\widehat{s}_k^{\,\Omega_2}
+\dots+\widehat{s}_k^{\,\Omega_{q_k}}$, and concatenate 
\be{conca}
\widehat{\Theta}^{k,\Omega}_{\text{traj,tot}}
=\Big(\widehat{\Theta}^{k,\Omega_1}_{\text{traj}},
\widehat{\Theta}^{k,\Omega_2}_{\text{traj}},\dots,
\widehat{\Theta}^{k,\Omega_{q_k}}_{\text{traj}}\Big) 
\in \widetilde{\cD}_{L_{t_k},\,q_k (N_{k_0}+1),\,}^{M,}
\ee
and
\be{concacat}
\widehat{x}^{k,\Omega}_{\text{tot}}=\big(\widehat x^{k_0,\Omega_1},
\widehat x^{k_0,\Omega_2},\dots,\widehat x^{k_0,\Omega_{q_k}}\big) 
\in \cX_{\widehat{\Theta}^{k,\Omega}_{\text{traj,tot}}
\Omega}^{M,m}.
\ee
and 
\be{concaca}
\widehat{u}^{k,\Omega}_{\text{tot}}=\big(\widehat{u}^{k,\Omega_1},
\widehat{u}^{k,\Omega_2},\dots,\widehat{u}^{k,\Omega_{q_k}}\big) 
\in \cU_{\widehat{\Theta}^{k,\Omega}_{\text{traj,tot}},
\widehat{x}^{k,\Omega}_{\text{tot}},
\widetilde{t}_k^{\,\Omega}}^{M,m,L_{t_k}}.
\ee
It still remains to complete $\widehat{\Theta}^{k,\Omega}_{\text{traj,tot}}$, 
$\widehat{x}^{k,\Omega}_{\text{tot}}$ and $\widehat{u}^{k,\Omega}_{\text{tot}}$ 
such that the latter becomes an element of $\cU_{\widehat{\Theta}^{k,
\Omega}_{\text{traj,tot}},\widehat{x}^{k,\Omega}_{\text{tot}},t_k}^{M,m,L_{t_k}}$. 
To that end, we recall \eqref{defqo}, which gives $t_k-\widetilde{t}_k^{\,\Omega}
\leq \widehat{s}_k^{\Omega_{q_k+1}}$. Then, using \eqref{impo}, we have that 
there exists a $c>0$ such that 
\be{bor}
t_k-\widetilde{t}_k^{\,\Omega}\leq c L_{t_k} \tfrac{n_{k_0}}{L_{n_{k_0}}}.
\ee
Therefore we can complete $\widehat{\Theta}^{k,\Omega}_{\text{traj,tot}}$, 
$\widehat{x}^{k,\Omega}_{\text{tot}}$ and $\widehat{u}^{k,\Omega}_{\text{tot}}$ 
with
\be{}
\Theta_{\text{rest}}\in \cD_{L_{t_k},\,g_k^\Omega}^M,
\qquad  x_{\text{rest}}\in\cX_{\Theta_{\text{rest}},\Omega}^{M,m},
\qquad u_{\text{rest}}\in\cU_{\Theta_{\text{rest}}, x_{\text{rest}}, t_k
-\widetilde{t}_k^{\,\Omega}}^{M,m,L_{t_k}},
\ee 
such that, by \eqref{bor}, the number of columns $g_{k}^\Omega$ involved in 
$\Theta_{\text{rest}}$ satisfies $g_k^\Omega\leq c n_{k_0}/L_{n_{k_0}}$. 
Henceforth $\widehat{\Theta}^{k,\Omega}_{\text{traj,tot}}$, $\widehat{x}^{k,\Omega}_{
\text{tot}}$ and $\widehat{u}^{k,\Omega}_{\text{tot}}$ stand for the quantities defined 
in \eqref{conca} and \eqref{concaca}, and concatenated with $\Theta_{\text{rest}}, 
x_{\text{rest}}$ and $u_{\text{rest}}$ so that they become elements of 
\be{}
\cD_{L_{t_k},\,q_k (N_{k_0}+1)+g_k^\Omega}^M,
\qquad \cX_{\widehat{\Theta}^{k,\Omega}_{\text{traj,tot}},\Omega}^{M,m},
\qquad\cU_{\widehat{\Theta}^{k,\Omega}_{\text{traj,tot}},
\widehat{x}^{k,\Omega}_{\text{tot}},
t_k}^{M,m,L_{t_k}},
\ee 
respectively. By restricting the summation in \eqref{interm1} to $\widehat{\Theta}^{k,
\Omega}_{\text{traj,tot}}$, $\widehat{x}^{k,\Omega}_{\text{tot}}$ and 
$\widehat{u}^{k,\Omega}_{\text{tot}}$, we get 
\be{f3s}
f_{3,t_k}(M,m)\geq \frac{L_{t_k}}{t_k} \E_{\Omega}\bigg[\sum_{j=1}^{q_k} H^{\Omega_j}
(\widehat{\Theta}^{k,\Omega_j}_{\text{traj}},
\widehat x^{k_0,\Omega_j},\widehat{u}^{k,\Omega_j})
+ H(\Theta_{\text{rest}}, x_{\text{rest}}, u_{\text{rest}})\bigg],
\ee
where the term $H(\Theta_{\text{rest}},x_{\text{rest}}, u_{\text{rest}})$ is 
negligible because, by \eqref{bor},  $(t_k-\widetilde{t}_k^{\,\Omega})/t_k$ 
vanishes as $k\to \infty$,  while all free energies per column are bounded from 
below by $(\beta-\alpha)/2$. Pick $\gep>0$ and recall \eqref{impo}. Choose 
$k_0$ such that $2 L_{n_{k_0}}/n_{k_0} \leq \gep/2$ and note that, for $k$ 
large enough, 
\be{impo1}
\widehat{s}_k^{\,\Omega}\in \Big[L_{t_k} 
\tfrac{n_{k_0}}{L_{n_{k_0}}}(1-\gep),L_{t_k} 
\tfrac{n_{k_0}}{L_{n_{k_0}}}(1+\gep)\Big].
\ee
By \eqref{defqo}, we therefore have 
\be{impo2}
q_k\in \Big[\tfrac{t_k L_{n_{k_0}}}{L_{t_k} n_{k_0}} \tfrac{1}{1+\gep}, 
\tfrac{t_k L_{n_{k_0}}}{L_{t_k} n_{k_0}} \tfrac{1}{1-\gep}\Big]=[a,b].
\ee
Recalling \eqref{f3s}, we obtain 
\be{f3s1}
f_{3,t_k}(M,m)\geq \frac{L_{t_k}}{t_k} \E_{\Omega}\bigg[\sum_{j=1}^{a} 
H^{\Omega_j}(\widehat{\Theta}^{k,\Omega_j}_{\text{traj}}, \widehat x^{k_0,\Omega_j},
\widehat{u}^{k,\Omega_j})-\sum_{j=a}^{b} \Big| H^{\Omega_j}
(\widehat{\Theta}^{k,\Omega_j}_{\text{traj}},
\widehat x^{k_0,\Omega_j},\widehat{u}^{k,\Omega_j})\Big|\bigg],
\ee
and, consequently,
\be{f3s2}
f_{3,t_k}(M,m)\geq \tfrac{L_{n_{k_0}}}{n_{k_0} (1+\gep)}\,
\E_{\,\Omega}\Big[H^{\Omega}(\widehat{\Theta}^{k,\Omega}_{\text{traj}},
\widehat x^{k_0,\Omega},\widehat{u}^{k,\Omega})\Big]-\frac{L_{t_k}}{t_k} 
(b-a) (N_{k_0}+1) m \tfrac{\beta-\alpha}{2}, 
\ee
and, by \eqref{rth},
\be{f3s2alt}
f_{3,t_k}(M,m)\geq \tfrac{l_1+\tfrac{l_2-l_1}{4}}{1+\gep} 
-(\tfrac{1}{1-\gep}-\tfrac{1}{1+\gep}) (b-a) m \tfrac{\beta-\alpha}{2}.
\ee
After taking $\gep$ small enough, we may conclude that $\liminf_{k\to\infty}
f_{3,t_k}(M,m) >l_1$, which completes the proof.

%%%%%%%%%%%%%%%

\subsection{Proof of Proposition \ref{pr:formimpp}}
\label{sMinf}

Pick $(M,m)\in \EIGH$ and note that, for every $n\in \N$, the set $\cW_{n,M}^{\,m}$ 
is contained in $\cW_{n,M}$. Thus, by using Proposition \ref{pr:formimp} we obtain 
\begin{align}
\label{limi}
\nonumber \liminf_{n\to \infty} f^{\Omega}_{1,n}(M;\alpha,\beta)
&\geq \sup_{m\geq M+2} \liminf_{n\to \infty} f^{\Omega}_{1,n}(M,m;\alpha,\beta)\\
&= \sup_{m\geq M+2} f(M,m;\alpha,\beta)\quad \text{ for } \P-a.e.\,\Omega.
\end{align}
Therefore, the proof of Proposition \ref{pr:formimpp} will be complete once we show that
\begin{align}
\label{lims}
\limsup_{n\to \infty} f^{\Omega}_{1,n}(M;\alpha,\beta)
\leq \sup_{m\geq M+2} \limsup_{n\to \infty} 
f^{\Omega}_{1,n}(M,m;\alpha,\beta)\quad \text{ for } \P-a.e.\,\Omega.
\end{align}
We will not prove \eqref{lims} in full detail, but only give the main steps in the proof. 
The proof consists in showing that, for $m$ large enough, the pieces of the trajectory 
in a column that exeed $m L_n$ steps do not contribute substantially to the free energy.

Recall (\ref{defDD2}--\ref{A11}) and use \eqref{A11} with $m=\infty$, i.e.,
\begin{align}
\label{variaf22}
Z_{n,L_n}^{\omega,\Omega}(M)
&=\sum_{N=1}^{n/L_n} \sum_{\Theta_{\text{traj}}\in \widetilde{\cD}_{L_n,N}^{M} } 
 \,\sum_{x\in \cX_{\Theta_{\text{traj}},\Omega}^{M,\infty} }
\,\sum_{u\in \,\cU_{\Theta_{\text{traj}},x,n}^{\,M,\infty,L_n}}  A_1.
\end{align} 
With each $(N,\Theta_{\text{traj}},x,u)$ in \eqref{variaf22}, we associate the trajectories
obtained by concatenating $N$ shorter trajectories $(\pi_i)_{i\in \{0,\dots,N-1\}}$ chosen 
in $(\cW_{\Theta_i,u_i,L_n})_{i\in \{0,\dots,N-1\}}$, respectively. Thus, the quantity $A_1$ 
in \eqref{variaf22} corresponds to the restriction of the partition function to the 
trajectories associated with $(N,\Theta_{\text{traj}},x,u)$. In order to discriminate between 
the columns in which more than $m L_n$ steps are taken and those in which less are taken, 
we rewrite $A_1$ as $A_2 \widetilde{A}_2$ with
\begin{align}
 A_2&=\prod_{i\in V_{u,m} }\, 
Z_{L_n}^{\omega_{I_i}}(\Theta_i,u_i),
\qquad 
\widetilde{A}_2=\prod_{i\in \widetilde V_{u,m} }\, 
Z_{L_n}^{\omega_{I_i}}(\Theta_i,u_i),
\end{align}
with $\widetilde{u}_{i}=\sum_{k=0}^{i-1} u_k$, $\Theta_i=(\Omega(i,\Pi_{i}+\cdot),\Xi_i,x_i)$ 
and $I_i=\{\widetilde{u}_{i} L_n,\dots,\widetilde{u}_{i+1} L_n-1\}$ for $i\in \{0,\dots,N-1\}$, 
with $\omega_{I}=(\omega_i)_{i\in I}$ for $I\subset \N$, where $\{0,\dots, N-1\}$ is partitioned 
into 
\be{defI}
\widetilde V_{u,m}\cup V_{u,m}\quad \text{with}\quad 
\widetilde V_{u,m}=\{i\in \{0,\dots,N-1\}\colon\; u_i>m\}.
\ee
For all $(N,\Theta_{\text{traj}},x,u)$, we rewrite $\widetilde V_{u,m}$ in the form of 
an increasing sequence $\{i_1,\dots,i_{\widetilde k}\}$ and we drop the $(u,m)$-dependence 
of $\widetilde{k}$ for simplicity. We also set $\widetilde{u}=u_{i_1}+\dots+u_{i_{
\widetilde{k}}}$, which is the total number of steps taken by a trajectory associated 
with $(N,\Theta_{\text{traj}},x,u)$ in those columns where more than $m L_n$ steps are 
taken. Finally, for $s\in \{1,\dots,\widetilde{k}\}$ we partition $I_{i_s}$ into
\begin{align}
\label{labint}
J_{i_s}\cup \widetilde{J}_{i_s} \quad \text{with} 
\quad J_{i_s}
&=\{\widetilde{u}_{i_s} L_n,\dots,(\widetilde{u}_{i_s}+M+2) L_n\}, \\
\quad \widetilde{J}_{i_s}
&=\{(\widetilde{u}_{i_s}+M+2) L_n+1,\dots,\widetilde{u}_{i_{s}+1} L_n-1\},
\end{align}
and we partition $\{1,\dots,n\}$ into 
\begin{align}
\label{twoeq}
& J\cup \widetilde{J} \quad \text{with} \quad 
\widetilde{J}=\cup_{s=1}^{\widetilde{k}} \widetilde{J}_{i_s}, 
\qquad J=\{1,\dots,n\}\setminus \widetilde{J},
\end{align}
so that $\widetilde{J}$ contains the label of the steps constituting the pieces of 
trajectory exeeding $(M+2)L_n$ steps in those columns where more than $m L_n$ steps 
are taken.  

%%%%%%%%%%%%%%%%%%%%%%%%%%%%%%%%%
 
\subsubsection{Step 1}
\label{step1}  
In this step we replace the pieces of trajectories in the columns indexed in 
$\widetilde V_{u,m}$  by shorter trajectories of length $(M+2) L_n$. To that aim, 
for every $(N,\Theta_{\text{traj}},x,u)$ we set 
\begin{align}
\widehat{A}_2=\prod_{i\in \widetilde V_{u,m} }\, 
Z_{L_n}^{\,\omega_{J_i}}(\Theta^{'}_i,M+2)
\end{align}
with $\Theta^{'}_i=(\Omega(i,\Pi_{i}+\cdot),\Xi_i,1)$. We will show that for all 
$\gep>0$ and for  $m$ large enough, the event 
\be{event}
B_n=\{\omega\colon\, \widetilde{A}_2\leq 
\widehat{A}_2\, e^{3\gep n} \ \text{for all}\ (N,\Theta_{\text{traj}},x,u)\}
\ee
satisfies $\P_{\omega}(B_n)\to 1$ as $n\to \infty$.

Pick, for each $s\in \{1,\dots,\widetilde{k}\}$, a trajectory $\pi_s$ in the set 
$\cW_{\Theta_{i_s},u_{i_s},L_n}$. By concatenating them we obtain a trajectory in 
$\cW_{\widetilde u L_n}$ satisfying $\pi_{\widetilde u L_n,1}=\widetilde{k}L_n$. 
Thus, the total entropy carried by those pieces of trajectories crossing the columns 
indexed in $\{i_1,\dots,i_ {\widetilde{k}}\}$ is bounded above by
\be{boundentro}
{\textstyle \prod_{s=1}^{\widetilde{k}}\, 
|\cW_{\Theta_{i_s},u_{i_s},L_n}|
\leq \big|\{\pi\in \cW_{\widetilde{u}L_n}\colon\,
\pi_{\widetilde{u}L_n,1}=\widetilde{k}L_n\}\big|.}
\ee
Since $\widetilde{u}/\widetilde{k}\geq m$, we can use Lemma \ref{convularge} in 
Appendix \ref{Path entropies} to assert that, for $m$ large enough, the right-hand side 
of \eqref{boundentro} is bounded above by $e^{\gep n}$.

Moreover, we note that an $\widetilde{u} L_n$-step trajectory satisfying $\pi_{\widetilde{u}
L_n,1}=\widetilde{k}L_n$ makes at most $\widetilde{k} L_n+ \widetilde{u}$ excursions in 
the $B$ solvent because such an excursion requires at least one horizontal step or at 
least $L_n$ vertical steps. Therefore, by using the inequalities $\widetilde{k} L_n
\leq n/m$ and $\widetilde{u}\leq n/L_n$ we obtain that, for $n$ large enough, the sum 
of the Hamiltonians associated with $(\pi_1,\dots,\pi_{\widetilde{k}})$ is bounded from 
above, uniformly in $(N,\Theta_{\text{traj}},x,u)$ and $(\pi_1,\dots,\pi_{\widetilde{k}})$, 
by
\be{upb}
{\textstyle \sum_{s=1}^{\widetilde{k}} 
H_{u_{i_s} L_n,L_n}^{\omega_{I_{i_s}},\Omega(i_s,\Pi_{i_s}+\cdot)}
(\pi_s)\leq \max\{\sum_{i\in I} \xi_i\colon\, I\in \cup_{r=1}^{2n/m}\cE_{n,r}\}},
\ee
with $\cE_{n,r}$ defined in \eqref{add26} in Appendix \ref{Computation} and $\xi_i=\beta 
1_{\{\omega_i=A\}}-\alpha 1_{\{\omega_i=B\}}$ for $i\in \N$. At this stage we use the 
definition in \eqref{add28} and note that, for all $\omega \in \cQ_{n,m}^{\gep/\beta,
(\alpha-\beta)/2+\gep}$, the right-hand side in \eqref{upb} is smaller than $\gep n$.
Consequently, for $m$ and $n$ large enough we have that, for all $\omega \in 
\cQ_{n,m}^{\gep/\beta,(\alpha-\beta)/2+\gep}$, 
\be{indif}
\widetilde{A}_2\leq e^{2\gep n}\quad \text{for all}\quad (N,\Theta_{\text{traj}},x,u).
\ee

Recalling \eqref{boundel} and noting that $\widetilde{k} L_n\leq n/m$, we can write
\be{boundbe}
\widehat{A}_2\geq e^{-\widetilde{k} (M+2) L_n C_{\text{uf}}(\alpha)}
\geq e^{-n \tfrac{M+2}{m}  C_{\text{uf}}(\alpha)},
\ee
and therefore, for $m$ large enough, for all $n$ and all $(N,\Theta_{\text{traj}},x,u)$ 
we have $ \widehat{A}_2\geq e^{-\gep n}$.

Finally, use \eqref{indif} and \eqref{boundbe} to conclude that, for $m$ and $n$ large 
enough, $\cQ_{n,m}^{\gep/\beta,(\alpha-\beta)/2+\gep}$ is a subset of $B_n$. Thus, Lemma
\ref{lele} ensures that, for $m$ large enough, $\lim_{n\to\infty} P_\omega(B_n)=1$.

%%%%%%%%%%%%%%%%%%%%%%%%%%%%%%%%%%%%%%%%%%%%%%%%%%%%%%%
\subsubsection{Step 2}
\label{step2}

Let $(\widetilde{w}_i)_{i\in\N}$ be an i.i.d.\ sequence of Bernouilli trials, independent 
of $\omega,\Omega$. For $(N,\Theta_{\text{traj}},x,u)$ we set $\widehat{u}=\widetilde{u}
-\widetilde{k} (M+2)$. In Step 1 we have removed $\widehat{u}L_n$ steps from the trajectories
associated with $(N,\Theta_{\text{traj}},x,u)$ so that they have become trajectories 
associated with $(N,\Theta_{\text{traj}},x^{'},u)$. In this step, we will concatenate 
the trajectories associated with $(N,\Theta_{\text{traj}},x^{'},u)$ with an 
$\widehat{u} L_n$-step trajectory to recover a trajectory that belongs to 
$\cW_{n,M}^{\,m}$.
 
For $\Omega\in \{A,B\}^{\N_0\times \Z}$, $t,N\in \N$ and $k\in \Z$, let
\be{mesemp}
P_{A}^\Omega(N,k)(t)=\frac{1}{t} \sum_{j=0}^{t-1} 1_{\{\Omega(N+j,k)=A\}}
\ee
be the proportion of $A$-blocks on the $k^{\text{th}}$ line and between the $N^{\text{th}}$ 
and the $(N+t-1)^{\text{th}}$ column of $\Omega$. Pick $\eta>0$ and $j\in \N$, and set 
\be{defsn}
S_{\eta,j}=\bigcup_{N=0}^{j}\bigcup_{k=-m_1 N}^{m_1 N} 
\bigcup_{t\geq \eta j}\Big\{ P_{A}^\Omega(N,k)(t)\leq \frac{p}{2}\Big\}.
\ee
By a straightforward application of Cramer's Theorem for i.i.d.\ random variables, we have
that $\sum_{j\in\N} P_\Omega(S_{\eta,j})<\infty$. Therefore, using the Borel-Cantelli Lemma, 
it follows that for $\P_\Omega$-a.e.\ $\Omega$, there exists a $j_\eta(\Omega)\in \N$ such 
that $\Omega\notin S_{\eta,j}$ as soon as $j\geq j_\eta(\Omega)$. In what follows, we consider 
$\eta=\gep/\alpha m$ and we take $n$ large enough so that  $n/L_n\geq j_{\gep/\alpha m}
(\Omega)$, and therefore $\Omega\notin S_{\frac{n}{L_n},\frac{\gep}{\alpha m}}$. 

Pick $(N,\Theta,x,u)$ and consider one trajectory $\widehat \pi$, of length $\widehat{u} L_n$, 
starting from $(N,\Pi_N+b_N)L_n$, staying in the coarsed-grained line at height $\Pi_N$, 
crossing the $B$-blocks in a straight line and the $A$-blocks in $m L_n$ steps. The number 
of columns crossed by $\widehat{\pi}$ is denoted by $\widehat{N}$ and satisfies $\widehat{N}
\geq \widehat{u}/m$. If $\widehat{u} L_n \leq \gep n/\alpha$, then the Hamiltonian associated 
with $\widehat{\pi}$ is clearly larger than $-\gep n$. If $\widehat{u} L_n \geq \gep n/\alpha$ 
in turn, then
\be{hampi}
H_{\widehat{u} L_n,L_n}^{\,\widetilde{w},
\Omega(N+\cdot,\Pi_N)}(\widehat{\pi})
\geq  -\alpha L_n  \widehat{N} \big[1-P_{A}^\Omega(N,\Pi_N)(\widehat{N})\big].
\ee
Since $N\leq n/L_n$, $|\Pi_N|\leq m_1 N$ and $\widehat{N}\geq \gep n/(\alpha m L_n)$, we can 
use the fact that $\Omega\notin S_{\frac{n}{L_n},\frac{\gep}{\alpha m}}$ to obtain 
\be{boundP}
P_{A}^\Omega(N,\Pi_N)(\widehat{N})\geq \frac{p}{2}.
\ee
At this point it remains to bound $\widehat{N}$ from above, which is done by noting that
\be{boundN}
\widehat{N} \big[m P_{A}^\Omega(N,\Pi_N)(\widehat{N})
+1-P_{A}^\Omega(N,\Pi_N)(\widehat{N})\big]
=\widehat{u}\leq \tfrac{n}{L_n}.
\ee
Hence, using \eqref{boundP} and \eqref{boundN}, we obtain $\widehat{N}\leq 2 n/p m L_n$
and therefore the right-hand side of \eqref{hampi} is bounded from below by $-\alpha (2-p)
n/p m$, which for $m$ large enough is larger than $-\gep n$.

Thus, for $n$ and $m$ large enough and for all $(N,\Theta,x,u)$, we have a trajectory 
$\widehat{\pi}$ at which the Hamiltonian is bounded from below by $-\gep n$ that can 
be concatenated with all trajectories associated with $(N,\Theta,x{'},u)$ to obtain a 
trajectory in $\cW_{n,M}^{\,m}$. Consequently, recalling \eqref{labint}, for $n$ and 
$m$ large enough we have 
\be{boundA2}
A_2\widehat{A_2}\leq e^{\gep n} 
Z_{\,n,L_n}^{(\omega_{J}\,,\,\widetilde \omega), \Omega} (M,m)
\qquad \forall\, (N,\Theta,x,u).
\ee

%%%%%%%%%%%%%%%%%%%%%%%%%%%%%%%%%%%%%%%%%%%%%%
\subsubsection{ Step 3}
\label{step3}

In this step, we average over the microscopic disorders $\omega,\widetilde{\omega}$. 
Use \eqref{boundA2} to note that, for $n$ and $m$ large enough and all $\omega\in B_n$, 
we have
\begin{align}
\label{variaf222}
Z_{n,L_n}^{\omega,\Omega}(M)
&\leq e^{4\gep n} \sum_{N=1}^{n/L_n} 
\sum_{\Theta_{\text{traj}}\in \widetilde{\cD}_{L_n,N}^{M} } 
 \,\sum_{x\in \cX_{\Theta_{\text{traj}},\Omega}^{M,\infty} }
\,\sum_{u\in \,\cU_{\Theta_{\text{traj}},x,n}^{\,M,\infty,L_n}} 
Z_{\,n,L_n}^{(\omega_{J}\,,\,\widetilde{\omega}), \Omega} (M,m).
\end{align}  
We use \eqref{concmesut} to claim that there exists $C_1,C_2>0$ so that for all $n\in \N$, all 
$m\in \N$ and all $J$, 
\be{inegco}
\P_{\omega,\widetilde{\omega}}\Big(\Big|\tfrac1n
\log Z_{n,L_n}^{(\omega_{J}\,,\,\widetilde{\omega}),\Omega}(M,m)
-f_{1,n}^{\Omega}(M,m)\Big|\geq \gep \Big)\leq C_1 e^{-C_2 \gep^2 n}.
\ee 
We set also
\be{inegco2}
D_{n}=\bigcap_{(N,\Theta_{\text{traj}},x,u)} 
\Big\{\Big|\tfrac1n\log Z_{n,L_n}^{(\omega_{J}\,,
\,\widetilde{\omega}),\Omega}(M,m)-f_{1,n}^{\Omega}(M,m)\Big|\leq \gep\Big\},
\ee
recall the definition of $c_n$ in \eqref{carsum} (used with $(M,\infty)$), and use 
\eqref{inegco} and the fact that $c_n$ grows subexponentially, to obtain $\lim_{n\to\infty} 
\P_{\omega,\widetilde{\omega}}(D_n^c)= 0$. For all $(\omega,\widetilde{\omega})$ satisfying 
$\omega\in B_n$ and $(\omega,\widetilde{\omega})\in D_n$, we can rewrite \eqref{variaf222} as 
\begin{align}
\label{variafin}
Z_{n,L_n}^{\omega,\Omega}(M)
&\leq c_n\ e^{n  f_{1,n}^{\Omega}(M,m) +5\gep n}.
\end{align}  
As a consequence, recalling \eqref{boundel}, for $m$ large enough we have 
\be{endres}
f^{\Omega}_{n}(M;\alpha,\beta)\leq \P(B_n^c\cup D_n^c) 
\, C_{\text{uf}}(\alpha)+ \frac{\log c_n}{n}+\frac1n
\E\Big(1_{\{B_n\cup D_n\}}  \big(n f_{1,n}^{\Omega}(M,m) +5\gep n\big)\Big).
\ee
Since $\P(B_n^c\cup D_n^c)$ and $(\log c_n)/n$ vanish when $n\to \infty$, it suffices to 
apply Proposition \ref{pr:formimp} and to let $\gep\to 0$ to obtain \eqref{lims}. This 
completes the proof of Proposition \ref{pr:formimpp}.

%%%%%%%%%%%%%%%

\subsection{Proof of Proposition \ref{pr:varar}}
\label{svarar}

Note that, for all $m\geq M+2$, we have $\cR_{p,M}^{m}\subset  \cR_{p,M}$. Moreover, any 
$(u_\Theta)_{\Theta\in \overline\cV_M^{\,m}}\in \cB_{\overline \cV_M^{\,m}}$  can be 
extended to $\overline\cV_M$ so that it belongs to $\cB_{\overline \cV_M}$. Thus, 
\be{}
\sup_{m\geq M+2} f(M,m;\alpha,\beta)\leq \sup_{\rho\in \cR_{p,M}} 
\sup_{(u)\in \cB_{\overline\cV_M} }  V(\rho,u).
\ee
As a consequence, it suffices to show that for all $\rho\in \cR_{p,M}$ and $(u_\Theta)_{
\Theta\in \overline\cV_M}\in \cB_{\overline \cV_M}$,
\be{finalst}
V(\rho,u)\leq\sup_{m\geq M+2} \sup_{\rho\in \cR_{p,M}^{m}} 
\sup_{(u)\in \cB_{\overline\cV_M^{\,m}} } V(\rho,u).
\ee
If $\int_{\overline \cV_M} u_\Theta\,\rho(d\Theta)=\infty$, then \eqref{finalst} is 
trivially satisfied since $V(\rho,u)=-\infty$. Thus, we can assume that $\rho(\overline
\cV_M\setminus D_M)=1$, where $D_M=\{\Theta\in\overline \cV_M\colon\,\chi_\Theta
\in\{A^{\Z},B^{\Z}\}, x_\Theta=2\}$. Since $\int_{\overline \cV_M} u_\Theta\,\rho(d\Theta)
<\infty$ and since (recall \eqref{boundel}) $\psi(\Theta,u)$ is uniformly bounded by 
$C_{\text{uf}}(\alpha)$ on $(\Theta,u)\in \overline \cV_M^{\,*}$, we have by dominated 
convergence that for all $\gep>0$ there exists an $m_0\geq M+2$ such that, for all 
$m\geq m_0$, 
\be{mamj}
V(\rho,u)\leq \frac{\int_{\overline\cV_M^{\,m}} 
u_\Theta \psi(\Theta,u_\Theta) \rho(d\Theta)}{\int_{\overline\cV_M^{\,m}} 
u_\Theta  \rho(d\Theta)}+\tfrac{\gep}{2}.
\ee
Since $\rho(\overline \cV_M\setminus D_M)=1$ and since $\cup_{m\geq M+2} \overline 
\cV_M^{\,m}= \overline \cV_M\setminus D_M$, we have $\lim_{m\to\infty}\rho(\overline
\cV_M^{\,m})=1$. Moreover, for all $m\geq m_0$ there exists a $\widehat{\rho}_{m}\in
\cR_{p,M}^{m}$ such that $\widehat\rho_{m}=\rho_{m}+\overline{\rho}_{m}$, with $\rho_{m}$ 
the restriction of $\rho$ to $\overline \cV_M^{\,m}$ and $\overline \rho_m$ charging only 
those $\Theta$ satisfying $x_\Theta=1$. Since all $\Theta \in \overline\cV_M$ with 
$x_\Theta=1$ also belong to $\overline \cV_M^{\,M+2}$, we can state that $\overline
\rho_m$ only charges $\overline \cV_M^{\,M+2}$ . Therefore
\be{malmj}
V(\widehat\rho_m,u) 
=\frac{\int_{\overline\cV_M^{\,m}} u_\Theta \psi(\Theta,u_\Theta) \rho(d\Theta)
+\int_{\overline \cV_M^{\,M+2}} u_\Theta \psi(\Theta,u_\Theta) 
\overline \rho_m(d\Theta)}{\int_{\overline\cV_M^{\,m}} 
u_\Theta  \rho(d\Theta)+\int_{\overline \cV_M^{\,M+2}} u_\Theta \overline\rho_m(d\Theta)}.
\ee
Since $\Theta\mapsto u_\Theta$ is continuous on $\overline\cV_M$, there exists an $R>0$ 
such that $u_\Theta\leq R$ for all $\Theta\in \overline\cV_M^{M+2}$. Therefore we can 
use \eqref{mamj} and \eqref{malmj} to obtain, for $m\geq m_0$,
\be{malmj2}
V(\widehat\rho_m,u)\geq (V(\rho,u)-\tfrac{\gep}{2}) 
\frac{\int_{\overline\cV_M^{\,m}} u_\Theta \rho(d\Theta)}{\int_{\overline\cV_M^{\,m}} 
u_\Theta  \rho(d\Theta)+\int_{\overline  \cV_M^{\,M+2}} 
u_\Theta \overline\rho_m(d\Theta)}
- R\, C_{\text{uf}}(\alpha)\, (1-\rho(\overline\cV_M^{\,m})).
\ee
The fact that $\overline \rho_m(\cV_M^{M+2})=\rho(\overline\cV_M\setminus\overline\cV_M^{\,m})$ 
for all $m\geq m_0$ impliess that $\lim_{m\to\infty}\overline \rho_m(\cV_M^{M+2})=0$.
Consequently, the right-hand side in \eqref{malmj2} tends to $V(\rho,u)-\gep/2$ as $m\to\infty$. 
Thus, there exists a $m_1\geq m_0$  such that $V(\widehat{\rho}_{m_1},u)\geq V(\rho,u)-\gep$. 
Finally, we note that there exists a $m_2\geq m_1+1$ such that $u_\Theta\leq m_2$ for all 
$\Theta\in \overline\cV_M^{\,m_1}$, which allows us to extend $(u_\Theta)_{\Theta\in \overline
\cV_M^{\,m_1}}$ to $\overline \cV_M^{\,m_2}$ such that $(u_\Theta)_{\Theta\in \overline 
\cV_M^{\,m_2}}\in \cB_{\overline \cV_M^{\,m_2}}$. It suffices to note that $\widehat{\rho}_{m_1}
\in \cR_{p,M}^{m_1}\subset  \cR_{p,M}^{m_2}$ to conclude that 
\be{finaalt}
V(\rho,u)\leq f(M,m_2;\,\alpha,\beta)+\gep.
\ee

%%%%%%%%%%%%%% APPENDICES %%%%%%%%%%%%%%%%%%%%%%%%%%%%%%%%%%%%%%%%%%%

\begin{appendix}

\section{Properties of path entropies}
\label{Path entropies}

In Appendix~\ref{A.1} we state a basic lemma (Lemma~\ref{conunif1}) about uniform 
convergence of path entropies in a single column. This lemma is proved with the
help of three additional lemmas (Lemmas~\ref{convularge}--\ref{convularge3}),
which are proved in Appendix~\ref{A.2}. The latter ends with an elementary lemma 
(Lemma~\ref{l:lemconv2}) that allows us to extend path entropies from rational to 
irrational parameter values. In Appendix~\ref{A.3}, we extend Lemma~\ref{conunif1} 
to entropies associated with sets of paths fullfilling certain restrictions on 
their vertical displacement. 

%%%%%%%%%%%%%%%%%%%%%

\subsection{Basic lemma}
\label{A.1}

We recall the definition of $\widetilde{\kappa}_L$, $L\in \N$, in \eqref{ttrajblock} 
and $\widetilde{\kappa}$ in \eqref{conventr}.

\begin{lemma}
\label{conunif1}
For every $\gep>0$ there exists an $L_\gep\in \N$ 
such that 
\begin{align}
\label{unif}
|\tilde{\kappa}_L(u,l)-\tilde{\kappa}(u,l)|
&\leq \gep \text{ for } L\geq L_\gep \text{ and } (u,l)\in \cH_L.
\end{align}
\end{lemma}

\begin{proof} 
With the help of Lemma~\ref{convularge} below we get rid of those $(u,l) \in 
\cH\cap\mathbb{Q}^2$ with $u$ large, i.e., we prove that $\lim_{u\to\infty}
\kappa_L(u,l)=0$ uniformly in $L\in \N$ and $(u,l)\in \cH_L$. Lemma~\ref{convularge2} 
in turn deals with the moderate values of $u$, i.e., $u$ bounded away from infinity and 
$1+|l|$. Finally, with~Lemma \ref{convularge3} we take into account the small values 
of $u$, i.e., $u$ close to $1+|l|$. To ease the notation we set, for $\eta\geq 0$ and 
$M>1$, 
\be{defHM}
\cH_{L,\eta,M}=\{(u,l)\in \cH_L\colon  1+|l|+\eta\leq u\leq M\},
\qquad 
\cH_{\eta,M}=\{(u,l)\in \cH\colon  1+|l|+\eta \leq u\leq M\}.
\ee 

\bl{convularge}
For every $\gep>0$ there exists an $M_\gep>1$ such that 
\be{convularge1}
\tfrac{1}{uL}\log \big|\{\pi\in \cW_{uL}\colon \pi_{uL,1}=L\}\big|
\leq \gep \qquad \forall\,L\in \N,\,u\in 1+\tfrac{\N}{L}\colon u\geq M_\gep.
\ee
\el

\bl{convularge2}
For every $\gep>0$, $\eta>0$ and $M>1$ there exists an $L_{\gep,\eta,M}\in \N$ 
such that
\begin{align}
\label{unif1}
|\tilde{\kappa}_L(u,l)-\tilde{\kappa}(u,l)|
&\leq \gep \qquad \forall\,L\geq L_{\gep,\eta,M},\,(u,l)\in \cH_{L,\eta,M}.
\end{align}
\el

\bl{convularge3} 
For every $\gep>0$ there exist $\eta_\gep \in (0,\tfrac12)$ and $L_\gep\in \N$ 
such that
\be{convunf4alt}
|\tilde{\kappa}_L(u,l)-\tilde{\kappa}_{L}(u+\eta,l)|
\leq \gep \qquad \forall\,L\geq L_\gep,\,(u,l)\in \cH_{L},\, 
\eta\in(0,\eta_\gep)\cap \tfrac{2\N}{L}.
\ee
\el

\noindent
Note that, after letting $L\to\infty$ in Lemma \ref{convularge3}, we get 
\be{complco}
|\tilde{\kappa}(u,l)-\tilde{\kappa}(u+\eta,l)|
\leq \gep \qquad \forall\,(u,l)\in \cH\cap\mathbb{Q}^2,\, 
\eta\in(0,\eta_\gep)\cap \mathbb{Q}.
\ee

Pick $\gep>0$ and $\eta_\gep\in (0,\tfrac12)$ as in Lemma~\ref{convularge3}. Note 
that Lemmas~\ref{convularge}--\ref{convularge2} yield that, for $L$ large enough, 
\eqref{unif} holds on $\{(u,l)\in \cH_L\colon u\geq 1+|l|+\frac{\eta_\gep}{2}\}$. 
Next, pick $L\in \N$, $(u,l)\in\cH_{L}\colon  u\leq 1+|l|+\frac{\eta_\gep}{2}$ 
and $\eta_L\in (\frac{\eta_\gep}{2},\eta_\gep)\cap\frac{2\N}{L}$, and write
\be{bounfpr}
|\tilde{\kappa}_L(u,l)-\tilde{\kappa}(u,l)|\leq A+B+C,
\ee
where
\be{abc}
A=|\tilde{\kappa}_L(u,l)-\tilde{\kappa}_L(u+\eta_L,l)|, 
\quad 
B=|\tilde{\kappa}_L(u+\eta_L,l)-\tilde{\kappa}(u+\eta_L,l)|,
\quad 
C=|\tilde{\kappa}(u+\eta_L,l)-\tilde{\kappa}(u,l)|.
\ee
By \eqref{complco}, it follows that $C\leq \gep$. As mentioned above, the fact 
that $(u+\eta_L,l)\in \cH_L$ and $u+\eta_L\geq |l|+\frac{\eta_\gep}{2}$ implies 
that, for $L$ large enough, $B\leq \gep$ uniformly in $(u,l)\in\cH_{L}\colon\,
u\leq 1+|l|+\frac{\eta_\gep}{2}$. Finally, from Lemma~\ref{convularge3} we obtain 
that $A\leq \gep$ for $L$ large enough, uniformly in $(u,l)\in\cH_{L}\colon\,
u\leq 1+|l|+\frac{\eta_\gep}{2}$. This completes the proof of Lemma~\ref{conunif1}.
\end{proof}

%%%%%%%%%%%%%%%%%%%%%%%%%%%%%%%%%%%%%%%%%%%%%%%%%%%%

\subsection{Proofs of Lemmas \ref{convularge}--\ref{convularge3}}
\label{A.2} 

%%%%%%%%%%%%%%%

\subsubsection{Proof of Lemma~\ref{convularge}} 
\label{A.2.1}

The proof relies on the following expression:
\be{compex}
v_{u,L} = \big|\{\pi\in \cW_{uL}\colon \pi_{uL,1}=L\}\big|
=\sum_{r=1}^{L+1} \binom{L+1}{r} \binom{(u-1) L}{r} 2^r,
\ee
where $r$ stands for the number of vertical stretches made by the trajectory (a 
vertical stretch being a maximal sequence of consecutive vertical steps). Stirling's 
formula allows us to assert that there exists a $g\colon\,[1,\infty)\to (0,\infty)$ 
satisfying $\lim_{u\to \infty}g(u)=0$ such that
\be{stir}
\binom{uL}{L}\leq e^{g(u) uL}, \qquad u\geq 1,\,L\in \N.
\ee
Equations (\ref{compex}--\ref{stir}) complete the proof. 

%%%%%%%%%%%%%%%%%%

\subsubsection{Proof of Lemma~\ref{convularge2}}
\label{A.2.2}

We first note that, since $u$ is bounded from above, it is equivalent to prove 
\eqref{unif1} with $\tilde{\kappa}_L$ and $\tilde{\kappa}$, or with $G_L$ and 
$G$ given by
\be{defG}
G(u,l) = u\tilde{\kappa}(u,l),\qquad
G_L(u,l) = u\tilde{\kappa}_L(u,l),\quad (u,l)\in \cH_L.
\ee
Via concatenation of trajectories, it is straightforward to prove that $G$ is 
$\mathbb{Q}$-concave on $\cH\cap \mathbb{Q}^2$, i.e., 
\be{Gconcav}
G(\lambda(u_1,l_1)+(1-\lambda)(u_2,l_2))
\geq \lambda G(u_1,l_1)+(1-\lambda) G(u_2,l_2),
\ \  \lambda\in \mathbb{Q}_{[0,1]},\,(u_1,l_1),(u_2,l_2)\in \cH\cap \mathbb{Q}^2.
\ee
Therefore $G$ is Lipschitz on every $K\cap\cH\cap \mathbb{Q}^2$ with $K \subset
\cH^0$ (the interior of $\cH$) compact. Thus, $G$ can be extended on $\cH^0$ to 
a function that is Lipschitz on every compact subset in $\cH^0$.

Pick $\eta>0$, $M>1$, $\gep>0$, and choose $L_\gep\in \N$ such that $1/L_\gep\leq 
\gep$. Since $\cH_{\eta,M}\subset\cH^0$ is compact, there exists a $c>0$ (depending 
on $\eta,M$) such that $G$ is $c$-Lipschitz on $\cH_{\eta,M}$. Moreover, any point 
in $\cH_{\eta,M}$ is at distance at most $\gep$ from the finite lattice $\cH_{L_\gep,
\eta,M}$. Lemma~\ref{lementr} therefore implies that there exists a $q_\gep\in \N$ 
satisfying
\be{convunf4}
|G_{qL_\gep}(u,l)-G(u,l)|
\leq \gep \qquad \forall\,(u,l)\in \cH_{L_\gep,\eta,M},\, q\geq q_\gep.
\ee
Let $L'=q_\gep L_\gep$, and pick $q\in \N$ to be specified later. Then, for $L\geq 
q L'$ and $(u,l)\in \cH_{L,\eta,M}$, there exists an $(u',l')\in\cH_{L_{\gep},\eta,M}$ 
such that $|(u,l)-(u',l')|_\infty \leq \gep$, $u>u'$, $|l|\geq |l'|$ and $u-u'\geq 
|l|-|l'|$. We recall \eqref{conventr} and write
\be{maj1}
0\leq G(u,l)-G_L(u,l)\leq A+B+C,
\ee
with
\begin{align}
A= |G(u,l)-G(u',l')|,\quad 
B= |G(u',l')-G_{L'}(u',l')|,\quad 
C= G_{L'}(u',l')-G_L(u,l).
\end{align}
Since $G$ is $c$-Lipschitz on $\cH_{\eta,M}$, and since $|(u,l)-(u',l')|_\infty\leq\gep$, 
we have $A\leq c\gep$. By \eqref{convunf4} we have that $B\leq \gep$. Therefore only $C$ 
remains to be considered. By Euclidean division, we get that $L=sL'+r$, where $s\geq q$ 
and $r\in \{0,\dots,L'-1\}$. Pick $\pi_1,\pi_2,\dots,\pi_s\in \cW_{L'}(u',|l'|)$, 
and concatenate them to obtain a trajectory in $\cW_{sL'}(u',|l'|)$. Moreover, 
note that 
\begin{align}
\label{impin}
uL-u'sL'&=(u-u')sL'+ur\\
\nonumber 
&\geq (|l|-|l'|)sL'+(1+|l|)r= (L-sL')+ (|l|L-s|l'|L'),
\end{align}
where we use that $L-sL'=r$, $u-u'\geq |l|-|l'|$ and $u\geq 1+|l|$. Thus, \eqref{impin} 
implies that any trajectory in $\cW_{L'}(u',|l'|)$ can be concatenated with 
an ($uL-u'sL'$)-step trajectory, starting at $(sL',s|l'|L')$ and ending at $(L,|l|L)$, 
to obtain a trajectory in  $\cW_{L}(u,|l|)$. Consequently,
\be{concat}
G_{L}(u,l)\geq \tfrac{s}{L}\log \kappa_{L'}(u',l')\geq \tfrac{s}{s+1} G_{L'}(u',l').
\ee
But $s\geq q$ and therefore $G_{L'}(u',l')-G_{L}(u,l)\leq \tfrac1q G_{L'}(u',l')\leq
\tfrac1q M \log 3$ (recall that $\log 3$ is an upper bound for all entropies per step). 
Thus, by taking $q$ large enough, we complete the proof.

%%%%%%%%%%%%%%%%%%%%%
\subsubsection{Proof of Lemma~\ref{convularge3}}
\label{A.2.3}

Pick $L\in \N$, $(u,l)\in \cH_L$, $\eta\in\frac{2\N}{L}$, and define the map $T\colon
\cW_{L}(u,l)\mapsto \cW_L(u+\eta,l)$ as follows. Pick $\pi\in  \cW_{L}(u,l)$, find 
its first vertical stretch, and extend this stretch by $\tfrac{\eta L}{2}$ steps. 
Then, find the first vertical stretch in the opposite direction of the stretch just
extended, and extend this stretch by $\tfrac{\eta L}{2}$ steps. The result of this 
map is $T(\pi)\in\cW_L(u+\eta,l)$, and it is easy to verify that $T$ is an injection, 
so that $|\cW_{L}(u,l)|\leq | \cW_{L}(u+\eta,l)|$. 

Next, define a map $\widetilde{T}\colon\,\cW_{L}(u+\eta,l)\mapsto \cW_L(u,l)$ as 
follows. Pick $\pi\in\cW_{L}(u+\eta,l)$ and remove its first $\tfrac{\eta L}{2}$ 
steps north and its first $\tfrac{\eta L}{2}$ steps south. The result is 
$\widetilde{T}(\pi)\in \cW_L(u,l)$, but $\widetilde{T}$ is not injective. However, 
we can easily prove that for every $\gep>0$ there exist $\eta_\gep>0$ and $L_\gep
\in \N$ such that, for all $\eta<\eta_\gep$ and all $L\geq l_\gep$, the number of 
trajectories in  $\cW_{L}(u+\eta,l)$ that are mapped by $\widetilde{T}$ to a 
particular trajectory in $\pi\in\cW_{L}(u,l)$ is bounded from above by $e^{\gep L}$, 
uniformly in $(u,l)\in\cH_L$ and $\pi\in\cW_{L},(u,l)$.

\medskip
This completes the proof of Lemmas~\ref{convularge}--\ref{convularge3}.

%%%%%%%%%%%%%%%%%%%%%%%
\subsubsection{Observation}
\label{A.2.4}

We close this appendix with the following observation. Recall Lemma~\ref{lementr}, 
where $(u,l)\mapsto \tilde{\kappa}(u,l)$ is defined on $\cH\cap\mathbb{Q}^2$. 

\bl{l:lemconv2}
(i) $(u,l)\mapsto u\tilde{\kappa}(u,l)$ extends to a continuous and strictly concave 
function on $\cH$.\\
(ii) $l\mapsto \tilde{\kappa}(u,l)$ is increasing on $[-u+1,0]$ and decreasing on 
$[0,u-1]$,\\  
(iii) $\lim_{u\to \infty} \tilde{\kappa}(u,0)=0$.\\
(iv) $u\mapsto u\tilde{\kappa}(u,l)$ is strictly increasing on $[1+|l|,\infty)$ and  $\lim_{u\to \infty} u \tilde{\kappa}(u,l)=\infty$. 
\el

\begin{proof}
(i) In the proof of Lemma~\ref{conunif1} we have shown that $\tilde{\kappa}$ can be 
extended to $\cH^0$ in such a way that $(u,l)\mapsto u\tilde{\kappa}(u,l)$ is 
continuous and concave on $\cH^0$. Lemma~\ref{convularge3} allows us to extend 
$\tilde{\kappa}$ to the boundary of $\cH$, in such a way that continuity and 
concavity of $(u,l)\mapsto u\tilde{\kappa}(u,l)$ hold on all of $\cH$. To obtain 
the strict concavity, we recall the formula in \eqref{kapexplform}, i.e., 
\be{kapexplform1}
u\tilde{\kappa}(u,l) = \left\{\begin{array}{ll}
u\kappa(u/|l|,1/|l|), &\l \neq 0,\\
u\hat{\kappa}(u), &l = 0,
\end{array}
\right.
\ee
where $(a,b)\mapsto a\kappa(a,b)$, $a\geq 1+b$, $b\geq 0$, and $\mu\mapsto \mu
\hat{\kappa}(\mu)$, $\mu \geq 1$, are given in \cite{dHW06}, Section 2.1, and 
are strictly concave. In the case $l\neq 0$, \eqref{kapexplform1} provides 
strict concavity of $(u,l)\mapsto u\tilde{\kappa}(u,l)$ on $\cH^+=\{(u,l)
\in \cH\colon l>0\}$ and on $\cH^-=\{(u,l)\in \cH\colon l<0\}$, while in the 
case $l=0$ it provides strict concavity on $\overline\cH=\{(u,0),u\geq 1\}$. 
We already know that $(u,l)\mapsto u\tilde{\kappa}(u,l)$ is concave on $\cH$, 
which, by the strict concavity on $\cH^+$, $\cH^-$ and $\overline\cH$, implies 
strict concavity of $(u,l)\mapsto u\tilde{\kappa}(u,l)$ on $\cH$. 

\medskip\noindent
(ii) This follows from concavity of $l\mapsto \tilde{\kappa}(u,l)$ and the
fact that $\tilde{\kappa}(u,l) = \tilde{\kappa}(u,-l)$.

\medskip\noindent
(iii) This is a direct consequence of Lemma~\ref{convularge}.

\medskip\noindent
(iv) By (i) we have that  $u\mapsto u\tilde{\kappa}(u,l)$ is strictly concave 
on $[1+|l|,\infty)$. Therefore, proving that $\lim_{u\to \infty} u \tilde{\kappa}
(u,l)=\infty$ is sufficient to obtain that  $u\mapsto u\tilde{\kappa}(u,l)$ is 
strictly increasing. It is proven in \cite{dHW06}, Lemma 2.1.2 (iii), that 
$\lim_{\mu\to \infty} u\hat{\kappa}(u)=\infty$, so that \eqref{kapexplform1} 
completes the proof for $l=0$. If $l\neq 0$, then we use \eqref{kapexplform1} 
again and the variational formula in the proof of \cite{dHW06}, Lemma 2.1.1, 
to check that $\lim_{a\to \infty} a \kappa(a,b)=\infty$ for all $b>0$.
\end{proof}

%%%%%%%%%%%%%%%%%%%%%%%%%%%%%%%%%%%%

\subsection{A generalization of Lemma \ref{conunif1}}
\label{A.3} 

In Section~\ref{proofofgene} we sometimes needed to deal with subsets of trajectories 
of the following form. Recall \eqref{add4}, pick $L\in \N$, $(u,l)\in \cH_L$ and 
$B_0, B_1\in \tfrac{Z}{L}$ such that
\be{condb}
B_1\,\geq 0\vee l\,\geq\, 0\wedge l\,\geq B_0 \quad \text{and}\quad B_1-B_0\geq 1.
\ee 
Denote by $\widetilde\cW_L(u,l,B_0,B_1)$ the subset of $\cW_L(u,l)$ containing those 
trajectories that are constrained to remain above $B_0 L$ and below $B_1L$ (see 
Fig.~\ref{figtratra}), i.e.,  
\begin{align}
\label{trajblockalt}
\widetilde\cW_L(u,l,B_0,B_1) &= \big\{\pi\in \cW_L(u,l) \colon\, 
B_0 L< \pi_{i,2}<B_1 L  \text{ for } i\in \{1,\dots,uL-1\}\big\},
\end{align}
and let 
\be{ttraj}
\widetilde\kappa_L(u,l,B_0,B_1) = \frac{1}{uL} \log |\widetilde\cW_L(u,l,B_0,B_1) |
\ee 
be the entropy per step carried by the trajectories in $\widetilde\cW_L(u,l,B_0,B_1)$. 
With Lemma~\ref{conunifalt1} below we prove that the effect on the entropy of the 
restriction induced by $B_0$ and $B_1$ in the set $\widetilde\cW_L(u,l)$ vanishes 
uniformly as $L\to \infty$.

%%%%%%%%%%%%%%%%%%%%%%%%%%%%%%%%%%%%%%%%%%%%%%%%%%%%%%%%%%%%%%
\begin{figure}[htbp]
\begin{center}
\includegraphics[width=.48\textwidth]{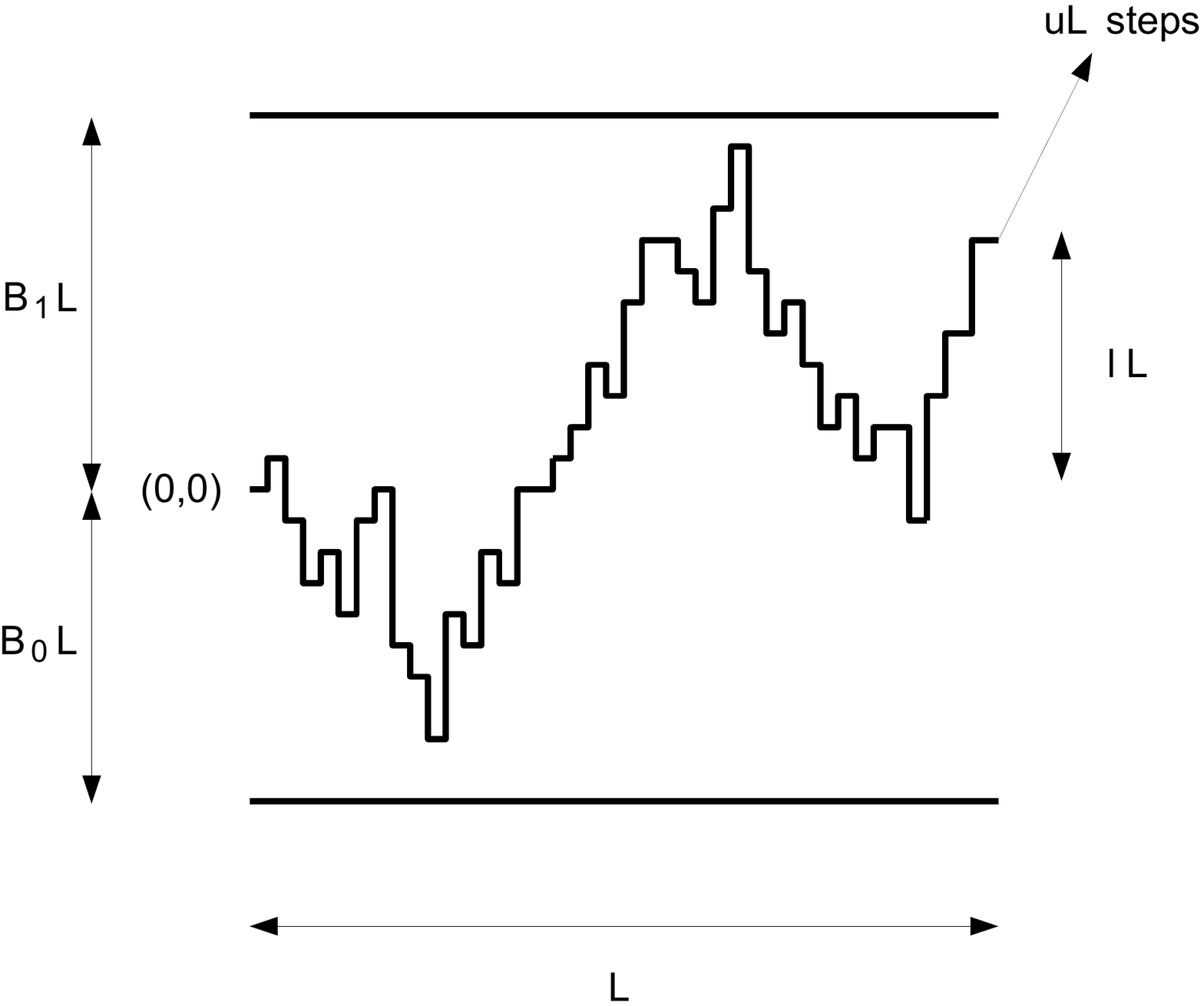}
\end{center}
\vspace{-4.3cm}
\caption{A trajectory in $\widetilde\cW_L(u,l,B_0,B_1)$.}
\label{figtratra}
\end{figure}
%%%%%%%%%%%%%%%%%%%%%%%%%%%%%%%%%%%%%%%%%%%%%%%%%%%%%%%%%%%%%%%

\begin{lemma}
\label{conunifalt1}
For every $\gep>0$ there exists an $L_\gep\in \N$ such that, for $L\geq L_\gep$, $(u,l)
\in \cH_L$ and $B_0,B_1\in \Z/L$ satisfying $B_1-B_0\geq 1$, $B_1\,\geq \max\{0,l\}$ 
and $B_0\leq \min\{0,l\}$,
\begin{align}
\label{unifalt}
& |\tilde{\kappa}_L(u,l,B_0,B_1)-\tilde{\kappa}_L(u,l)|
\leq \gep. 
\end{align}
\end{lemma}

\begin{proof}
The key fact is that $B_1-B_0 \geq 1$. The vertical restrictions $B_1\,\geq \max\{0,l\}$ 
and $B_0\leq \min\{0,l\}$ gives polynomial corrections in the computation of the 
entropy, but these corrections are harmless because $(B_1-B_0)L$ is large.  
\end{proof}

%%%%%%%%%%%%%%%%%%%%%%%%%%%%%%

\section{Properties of free energies}
\label{B}

%%%%%%%%%%%%
 
\subsection{Free energy along a single linear interface}
\label{B.1}

Also the free energy $\mu\mapsto\phi^\cI(\mu;\alpha,\beta)$ defined in 
Proposition~\ref{l:feinflim} can be extended from $\mathbb{Q}\cap [1,\infty)$ to
$[1,\infty)$, in such a way that $\mu\mapsto \mu\phi^\cI(\mu;\alpha,\beta)$ is concave and continous on $[1,\infty)$.  By concatenating trajectories, we can indeed check that 
$\mu\mapsto \mu\phi^\cI(\mu;\alpha,\beta)$ is concave on $\mathbb{Q}\cap [1,\infty)$. 
Therefore it is Lipschitz on every compact subset of $(1,\infty)$ and can be extended 
to a concave and continuous function on $(1,\infty)$. The continuity at $\mu=1$ comes 
from the fact that $\phi^\cI(1;\alpha,\beta)=0$ and $\lim_{\mu \downarrow 1} 
\phi^\cI(\mu)=0$, which is obtained by using Lemma~\ref{lele} below.

\bl{l:lemconv}
For all $(\alpha,\beta)\in\CONE$:\\
(i) $\mu \mapsto \mu \phi^\cI(\mu;\alpha,\beta)$ is strictly increasing on $[1,\infty)$ 
and $\lim_{\mu\to \infty} \mu\phi^\cI(\mu;\alpha,\beta)=\infty$.\\
(ii) $\lim_{\mu\to \infty}\phi^\cI(\mu;\alpha,\beta)=0$.
\el

\begin{proof}
(i) Clearly, $\phi^\cI(\mu;\alpha,\beta)\geq \widetilde{\kappa}(\mu,0)$ for $\mu\geq 1$.
Therefore Lemma \ref{l:lemconv2}(iv) implies that $\lim_{\mu\to \infty} \mu\phi^\cI
(\mu;\alpha,\beta)=\infty$. Thus, the concavity of $\mu\mapsto \mu \phi^\cI(\mu;\alpha,
\beta)$ is sufficient to obtain that it is strictly increasing on $[1,\infty)$.\\
(ii) See \cite{dHP07b}, Lemma 2.4.1(i).
\end{proof}

\noindent
Recall Assumption \ref{assu}, in which we assumed that $\mu \mapsto \mu \phi^\cI(\mu;
\alpha,\beta)$ is strictly concave on $[1,\infty)$. The next lemma states that the 
convergence of the average quenched free energy $\phi^\cI_L$ to $\phi^\cI$ as 
$L\to\infty$ is uniform on $\mathbb{Q} \cap [1,\infty)$.

\bl{l:feinflim1} 
For every $(\alpha,\beta)\in\CONE$ and $\gep>0$ there exists an $L_\gep\in \N$ 
such that 
\be{fesainfalt}
|\phi_L(\mu)-\phi(\mu)|\leq \gep \qquad \forall\,\mu\in 1+\tfrac{2\N}{L},\,
L\geq L_\gep.
\ee
\el

\begin{proof}
Similarly to what we did for Lemma~\ref{conunif1}, the proof can be done by 
treating separately the cases $\mu$ large, moderate and small. We leave the 
details to the reader.
\end{proof}

%%%%%%%%%%

\subsection{Free energy in a single column}
\label{B.2}

We can extend $(\Theta,u)\mapsto \psi(\Theta,u)$ from $\cV_M^{*}$ to $\overline\cV_M^{*}$ 
by using the variational formulas in \eqref{Bloc of type I} and \eqref{Bloc of type NI,B} 
and by recalling that $\widetilde{\kappa}$ and $\phi^{\cI}$ have been extended to $\cH$ 
and $[1,\infty)$ in Appendices \ref{A.2} and \ref{B.1}.

Pick $M\in \N$ and recall \eqref{set2}. Define a distance $d_M$ on $\overline\cV_M$ 
as follows. Pick $\Theta_1,\Theta_2 \in \overline \cV_M$, abbreviate 
\be{defth1}
\Theta_1=(\chi_1,\Delta\Pi_1, b_{0,1},b_{1,1},x_1),
\qquad \Theta_2=(\chi_2,\Delta\Pi_2, b_{0,2},b_{1,2},x_2),
\ee 
and define
\be{dist}
d_M(\Theta_1,\Theta_2)= \sum_{j\in \Z} \frac{1_{\{\chi_1(j)\neq \chi_2(j)\}}}{2^{|j|}} +|\Delta \Pi_1-\Delta \Pi_2|+|b_{0,1}-b_{0,2}|+|b_{1,1}-b_{1,2}| 
\ee
%\be{dist}
%d_M(\Theta_1,\Theta_2)= \left\{
%\begin{array}{ll}
%&\hspace{-.4cm} 1 \mbox{ if  } \ \Delta\Pi_1\neq \Delta \Pi_2\ \,\text{or} 
%\ \, x_1\neq x_2, \\
%&\hspace{-.4cm} 1 \mbox{ if  } \ \exists\,j\in \{-m+1,\dots,m-1\} 
%\text{ such that } \chi_1(j)\neq \chi_2(j),\\
%&\hspace{-.4cm} |b_{0,1}-b_{0,2}|+|b_{1,1}-b_{1,2}| 
% \mbox{ otherwise},
%\end{array}
%\right.
%\ee
so that $\widetilde{d}_M((\Theta_1,u_1),(\Theta_2,u_2))=\max \{|u_1-u_2|,d_M
(\Theta_1,\Theta_2)\}$ is a distance on $\overline{\cV}^{\,*,m}_M$ for which 
$\overline{\cV}^{\,*,m}_M$ is compact.

\bl{concavt}
For every $(M,m)\in \EIGH$ and $(\alpha,\beta)\in \CONE$,
\begin{align}
\label{fuun}
(u,\Theta) \mapsto u\,\psi(\Theta,u;\alpha,\beta)
\end{align}
is uniformly continuous on $\overline{\cV}^{\,*,m}_M$ endowed with $\widetilde{d}_M$. 
\el

\begin{proof}
Pick $(M,m)\in \EIGH$. By the compactness of $\overline{\cV}^{\,*,m}_M$, it suffices to show 
that $(u,\Theta) \mapsto u\,\psi(\Theta,u)$ is continuous on $\overline{\cV}^{\,*,m}_M$. 
Let $(\Theta_n,u_n)=(\chi_n,\Delta\Pi_n,b_{0,n},b_{1,n},u_n)$ be the general term of 
an infinite sequence that tends to $(\Theta,u)=(\chi,\Delta\Pi,b_{0},b_{1},u)$ 
in $(\overline{\cV}^{\,*,m}_M,\widetilde{d}_M)$. We want to show that $\lim_{n\to\infty}
u_n\psi(\Theta_n,u_n)=u \psi(\Theta,u)$. By the definition of $\widetilde{d}_M$, we 
have $\chi_n=\chi$ and $\Delta\Pi_n=\Delta\Pi$ for $n$ large enough. We assume that 
$\Theta\in \cV_{\AB}$, so that $\Theta_n\in \cV_{\AB}$ for $n$ large enough as well. 
The case $\Theta\in \cV_{\text{nint}}$ can be treated similarly.

Set
\be{defrm}
\cR_m=\{(a,h,l)\in [0,m]\times [0,1]\times \R\colon h+|l|\leq a\}
\ee
and note that $\cR_m$ is a compact set. Let $g\colon\,\cR_m\mapsto [0,\infty)$ be 
defined as $g(a,h,l)=a\,\widetilde{\kappa}(\tfrac{a}{h},\tfrac{l}{h})$ if $h>0$ 
and $g(a,h,l)=0$ if $h=0$. The continuity of $\widetilde{\kappa}$, stated in 
Lemma~\ref{l:lemconv2}(i), ensures that $g$ is continuous on $\{(a,h,l)\in
\cR_m\colon h>0\}$. The continuity at all $(a,0,l)\in \cR_m$ is obtained by 
recalling that $\lim_{u\to\infty}\tilde{\kappa}(u,l)=0$ uniformly in $l\in
[-u+1,u-1]$ (see Lemma~\ref{l:lemconv2}(ii-iii)) and that $\widetilde{\kappa}$ 
is bounded on $\cH$.

In the same spirit, we may set $\cR'_m=\{(u,h)\in [0,m]\times[0,1]\colon\,
h\leq u\}$ and define $g'\colon\,\cR'_m\mapsto [0,\infty)$ as $g'(u,h)
= u\,\phi^{\cI}(\tfrac{u}{h})$ for $h>0$ and $g'(u,h)=0$ for $h=0$. With the 
help of Lemma~\ref{l:lemconv} we obtain the continuity of $g'$ on $\cR'_m$ 
by mimicking the proof of the continuity of $g$ on $\cR_m$.

Note that the variational formula in \eqref{Bloc of type I} can be rewriten as 
\begin{align}
\label{Bloc of type IIa} 
u \,\psi(\Theta,u)
&=\sup_{(h),(a) \in \cL(l_A,\, l_B;\,u)} Q((h),(a),l_A,l_B),
\end{align}
with
\be{defQ}
Q((h),(a),l_A,l_B)=g(a_A,h_A,l_A)+g(a_B,h_B,l_B)+a_B \,
\tfrac{\beta-\alpha}{2}+g^{'}(a^\cI,h^\cI),
\ee
and with $l_A$ and $l_B$ defined in \eqref{bl}. Note that $\cL(l_A,\,l_B;\,u)$ is 
compact, and that $(h),(a)\mapsto Q((h),(a),l_A,l_B)$ is continuous on 
$\cL(l_A,\, l_B;\,u)$ because $g$ and $g'$ are continuous on $\cR_m$ and 
$\cR^{'}_m$, respectively. Hence, the supremum in \eqref{Bloc of type IIa} is 
attained.

Pick $\gep>0$, and note that $g$ and $g'$ are uniformly continuous on $\cR_m$ 
and $\cR'_m$, which are compact sets. Hence there exists an $\eta_\gep>0$ such 
that $|g(a,h,l)-g(a',h',l')|\leq \gep$ and $|g'(u,b)-g'(u',b')| \leq \gep$ when 
$(a,h,l),(a',h',l')\in  \cR_m$ and $(u,b),(u',b')\in \cR'_m$ are such that 
$|a-a'|,|h-h'|,|l-l'|,|u-u'|$ and $|b-b'|$ are bounded from above by $\eta_\gep$.

Since $\lim_{n\to \infty}(\Theta_n,u_n)=(\Theta,u)$ we also have that 
$\lim_{n\to\infty} b_{0,n}=b_0$, $\lim_{n\to \infty} b_{1,n}=b_1$ and 
$\lim_{n\to \infty} u_{n}=u$. Thus, $\lim_{n\to\infty} l_{A,n}= l_A$ and 
$\lim_{n\to \infty} l_{B,n}= l_B$, and therefore $|l_{A,n}- l_A|\leq \eta_\gep$, 
$|l_{B,n}- l_B|\leq \eta_\gep$ and $|u_n-u|\leq \eta_\gep$ for $n\geq n_\gep$ 
large enough.

For $n\in \N$, let $(h_n),(a_n)\in\cL(l_{A,n},\, l_{B,n};\,u_n)$ be a maximizer 
of \eqref{Bloc of type IIa} at $(\Theta_n,u_n)$, and note that, for $n\geq n_\gep$, 
we can choose $(\widetilde{h}_n),(\widetilde{a}_n)\in\cL(l_A,\,l_B;\,u)$ such 
that $|\widetilde{a}_{A,n}-a_{A,n}|$, $|\widetilde{a}_{B,n}-a_{B,n}|$, 
$|\widetilde{a}_{n}^\cI-a_{n}^\cI|$, $|\widetilde{h}_{A,n}-h_{A,n}|$, 
$|\widetilde{h}_{B,n}-h_{B,n}|$ and $|\widetilde{h}_{n}^\cI-h_{n}^\cI|$ are 
bounded above by $\eta_\gep$. Consequently,
\be{boundonpsi}
u_n \psi(\Theta_n,u_n)- u\psi(\Theta,u)
\leq Q((h_n),(a_n), l_{A,n},l_{B,n})-Q((\widetilde{h}_n),
(\widetilde{a}_n), l_{A},l_{B})\leq 3\gep.
\ee
We bound $u\psi(\Theta,u)-u_n \psi(\Theta_n,u_n)$ from above in a similar manner, 
and this suffices to obtain the claim.
\end{proof}

\bl{concav}
For every $\Theta\in \overline\cV_M$, the function $u \mapsto u\psi(\Theta,u)$ 
is continuous and strictly concave on $[t_\Theta,\infty)$.
\el

\begin{proof}
The continuity is a straightforward consequence of Lemma~\ref{concavt}: simply 
fix $\Theta$ and let $m\to\infty$. To prove the strict concavity, we note that 
the cases $\Theta\in \cV_{\AB}$ and $\Theta\in \cV_{\text{nint}}$ can be treated 
similarly. We will therefore focus on $\Theta\in \cV_{\AB}$.

For $l\in \R$, let
\be{defrmalt}
\cN_l=\{(a,h)\in [0,\infty) \times [0,1]\colon a\geq h+|l|\},
\quad \cN_l^+=\{(a,h)\in \cN_l\colon h>0\},
\ee
and let $g_l\colon\,\cN_l\mapsto [0,\infty)$ be defined as $g_l(a,h)=a\,
\widetilde{\kappa}(\tfrac{a}{h},\tfrac{l}{h})$ for $h>0$ and $g_l(a,h)=0$ for 
$h=0$. The strict concavity of $(u,l)\mapsto u\widetilde{\kappa}(u,l)$ on 
$\cH$, stated in Lemma \ref{l:lemconv2}(i), immediately yields that $g_l$ 
is strictly concave on $\cN_l^+$ and concave on $\cN_l$. Consequently, for 
all $(a_1,h_1)\in \cN_l^+$ and $(a_2,h_2)\cN_l\setminus \cN_l^+ $, $g_l$ is 
strictly concave on the segment $[(u_1,h_1),(u_2,h_2)]$.

Let $\widetilde\cN=\{(u,h)\in [0,\infty)\times [0,1]\colon\, h\leq u\}$ and 
define $\widetilde g\colon\,\widetilde \cN\mapsto [0,\infty)$ as 
$\widetilde g(u,h)= u\,\phi^{\cI}(\tfrac{u}{h})$ for $h>0$ and $\widetilde g(u,h)=0$ 
for $h=0$. The strict concavity of $u\mapsto u\phi^{\cI}(u)$ on $[1,\infty)$, 
stated in Assumption~\ref{assu}, immediately yields that $\widetilde g$ is 
strictly concave on $\widetilde\cN^+=\{(u,h)\in \widetilde \cN\colon h>0\}$ 
and concave on $\widetilde\cN$. Consequently, for all $(u_1,h_1)\in \widetilde
\cN^+$ and $(u_2,h_2)\in \widetilde\cN\setminus \widetilde{\cN}^+$, $\widetilde{g}$ 
is strictly concave on the segment $[(u_1,h_1),(u_2,h_2)]$.

Similarly to what we did in \eqref{Bloc of type IIa}, we can rewrite the variational 
formula in \eqref{Bloc of type I} as 
\begin{align}
\label{Bloc of type IIaa} 
u \,\psi(\Theta,u)
&=\sup_{(h),(a) \in \cL(l_A,\, l_B;\,u)} \widetilde{Q}((h),(a))
\end{align}
with
\be{defQtil}
\widetilde{Q}((h),(a))=g_{l_A}(a_A,h_A)+g_{l_B}(a_B,h_B)+a_B \,
\tfrac{\beta-\alpha}{2}+\widetilde{g}(u-a_A-a_B,1-h_A-h_B),
\ee
and the supremum in \eqref{Bloc of type IIaa} is attained. Next we show that if 
$(h),(a)\in \cL(l_A,\,l_B;\,u)$ realizes the maximum in \eqref{Bloc of type IIaa}, 
then $(h),(a)\notin \widetilde{\cL}(l_A,\,l_B;\,u)$ with
\be{restor}
\widetilde{\cL}(l_A,\, l_B;\,u) 
= \widetilde{\cL}_A(l_A,\, l_B;\,u)\cup\widetilde{\cL}_B(l_A,\, l_B;\,u)
\cup\widetilde{\cL}^{\, \cI}(l_A,\, l_B;\,u)
\ee
and
\begin{align}
\label{restore}
\nonumber \widetilde{\cL}_A(l_A,\, l_B;\,u)
&=\{(h),(a)\in \cL(l_A,\, l_B;\,u)\colon\,h_A=0 \ \ \text{and}\ \ a_A>l_A\},\\
\nonumber \widetilde{\cL}_B(l_A,\, l_B;\,u)
&=\{(h),(a)\in \cL(l_A,\, l_B;\,u)\colon\, h_B=0 \ \ \text{and}\ \ a_B>l_B\},\\
\widetilde{\cL}^{\,\cI}(l_A,\, l_B;\,u)
&=\{(h),(a)\in \cL(l_A,\, l_B;\,u)\colon\, h_I=0 \ \ \text{and}\ \ a_I>0\}.
\end{align}
Assume that $(h),(a)\in \widetilde{\cL}(l_A,\, l_B;\,u)$, and that $h_A>0$ or $h^\cI>0$. 
For instance, $(h),(a)\in \widetilde{\cL}^\cI(l_A,\, l_B;\,u)$ and $h_A>0$. Then, by 
Lemma~\ref{l:lemconv2}(iv), $\widetilde{Q}$ strictly increases when $a_A$ is replaced 
by $a_A+a^\cI$ and $a^\cI$ by $0$. This contradicts the fact that $(h),(a)$ is a maximizer. 
Next, if $(h),(a)\in \widetilde{\cL}(l_A,\,l_B;\,u)$ and $h_A=h^\cI=0$, then $h_B=1$, 
and the first case is $(h),(a)\in \widetilde{\cL}_A(l_A,\,l_B;\,u)$, while the second 
case is $(h),(a)\in \widetilde{\cL}^\cI(l_A,\, l_B;\,u)$. In the second case, as before, 
we replace $a_A$ by $a_A+a^\cI$ and $a^\cI$ by $0$, which does not change $\widetilde{Q}$ 
but yields that $a_A>l_A$ and therefore brings us back to the first case. In this first 
case, we are left with an expression of the form
\be{fca}
Q((h),(a))=g_{l_B}(a_B,1) + a_B\, \tfrac{\beta-\alpha}{2} 
\ee
with $h_A=h^\cI=0$ and $a_A>l_A$. Thus, if we can show that there exists an $x\in (0,1)$ 
such that 
\be{fca1}
g_{l_A}(a_A,x)+g_{l_B}(a_B,1-x)>g_{l_B}(a_B,1), 
\ee
then we can claim that $(h),(a)$ is not a maximizer of \eqref{Bloc of type IIaa} and 
the proof for $(h),(a)\notin \widetilde{\cL}(l_A,\,l_B;\,u)$ will be complete. 

To that end, we recall \eqref{kapexplform}, which allows us to rewrite the left-hand 
side in \eqref{fca1} as
\be{fca2}
g_{l_A}(a_A,x)+g_{l_B}(a_B,1-x)=a_A \,\kappa\big(\tfrac{a_A}{l_A},\tfrac{x}{l_A}\big)+
a_B \,\kappa\big(\tfrac{a_B}{l_B},\tfrac{1-x}{l_B}\big)+ a_B\, \tfrac{\beta-\alpha}{2}.
\ee
We recall \cite{dHW06}, Lemma 2.1.1, which claims that $\kappa$ is defined on 
$\DOM=\{(a,b)\colon a\geq 1+b, b\geq 0\}$, is analytic on the interior of $\DOM$  
and is continuous on $\DOM$. Moreover, in the proof of this lemma, an expression 
for $\partial_b\, \kappa(a,b)$ is provided, which is valid on the interior of $\DOM$. 
From this expression we can easily check that if $a>1$, then $\lim_{b\to 0} \partial_b\,
\kappa(a,b)=\infty$. Therefore, by the continuity of $\kappa$ on $(a_A/l_A,0)$ with 
$a_A/l_A>1$ we can assert that the derivative with respect to $x$ of the left-hand side 
in \eqref{fca2} at $x=0$ is infinite, and therefore there exists an $x>0$ such that 
\eqref{fca1} is satisfied.

Pick $u_1>u_2\geq t_\Theta$, and let $(h_1),(a_1) \in \cL(l_A,\, l_B;\,u_1)$ and 
$(h_2),(a_2) \in \cL(l_A,\, l_B;\,u_2)$ be maximizers of \eqref{Bloc of type IIaa} 
for $u_1$ and $u_2$, respectively. We can write
\begin{align}
\label{aa}
\nonumber (a_1),(h_1)
&=\big(a_{A,1},a_{B,1},a^{\cI}_1),(h_{A,1},h_{B,1},h^{\cI}_1\big),\\
(a_2),(h_2)
&=\big(a_{A,2},a_{B,2},a^{\cI}_2),(h_{A,2},h_{B,2},h^{\cI}_2\big).
\end{align}
Thus, $(\tfrac{a_1+a_2}{2}),(\tfrac{h_1+h_2}{2})\in \cL(l_A,\,l_B;\,\tfrac{u_1+u_2}{2})$ 
and, with the help of the concavity of $g_{l_A}, g_{l_B},\widetilde{g}$ proven above, 
we can write
\be{aaalt}
\tfrac{u_1+u_2}{2}\,\psi(\Theta,\tfrac{u_1+u_2}{2})
\geq \widetilde{Q}((\tfrac{a_1+a_2}{2}), (\tfrac{h_1+h_2}{2}))
\geq \tfrac12 \big(u_1\,\psi(\Theta,u_1)+u_2\,\psi(\Theta,u_2)\big).
\ee
We have proven above that $(a_1),(h_1) \notin \widetilde{\cL}(l_A,\, l_B;\,u_1)$ and 
$(a_2),(h_2) \notin \widetilde{\cL}(l_A,\, l_B;\,u_2)$. Thus, we can use \eqref{defQtil} 
and the strict concavity of $g_{l_A}, g_{l_B},\widetilde{g}$ on $\cN_{l_A}^+,\cN_{l_B}^+
\widetilde{\cN}^+$, to conclude that the right-most inequality in \eqref{aaalt} 
is an equality only if 
\be{}
\begin{aligned}
&(a_{A,1},h_{A,1}) = (a_{A,2},h_{A,2}), \quad (a_{B,1},h_{B,1}) = (a_{B,2},h_{B,2}),\\
&(u_1-a_{A,1}-a_{B,1},1-h_{A,1}-h_{B,1}) = (u_2-a_{A,2}-a_{B,2},1-h_{A,2}-h_{B,2}),
\end{aligned}
\ee 
which clearly is not possible because $u_1>u_2$.
\end{proof}

%%%%%%%%%%%%%%%%%%%%%%%%%%%

\section{Concentration of measure}
\label{Ann2}

Let $\cS$ be a finite set and let $(X_i,\cA_i,\mu_i)_{i\in \cS}$ be a family of 
probability spaces. Consider the product space $X=\prod_{i\in \cS} X_i$ endowed 
with the product $\sigma$-field $\cA=\otimes_{i\in \cS}\cA_i$ and with the product 
probability measure $\mu=\otimes_{i\in \cS} \mu_i$.  

\begin{theorem}
\label{theoco} {\rm (Talagrand~\cite{T96})}
Let $f\colon\,X\mapsto \R$ be integrable with respect to $(\cA,\mu)$ and, for 
$i\in \cS$, let $d_i>0$ be such that $|f(x)-f(y)|\leq d_i$ when $x,y\in X$ differ 
in the $i$-th coordinate only. Let $D=\sum_{i\in \cS} d_i^2$. Then, for all $\gep>0$,
\be{gteo}
\mu\left\{x\in X\colon \left|f(x)-\int fd\mu\right|>\gep\right\}
\leq 2 e^{-\frac{\gep^2}{2D}}.
\ee
\end{theorem}

The following corollary of Theorem~\ref{theoco} was used several times in the paper. 
Let $(\alpha,\beta)\in \CONE$ and let $(\xi_i)_{i\in\N}$ be an i.i.d.\ sequence 
of Bernouilli trials taking the values $-\alpha$ and $\beta$ with probability $\tfrac12$ 
each. Let $l\in \N$, $T\colon\,\,\{(x,y)\in\Z^2\times \Z^2\colon |x-y|=1\}\to\{0,1\}$ 
and $\Gamma\subset\cW_l$ (recall \eqref{defw}). Let $F_l\colon\,[-\alpha,\alpha]^l
\to \R$ be such that
\be{defF}
F_l(x_1,\dots,x_l) = \log \sum_{\pi\in \Gamma} e^{\sum_{i=1}^{l}
x_i\, T( (\pi_{i-1},\pi_i))}. 
\ee
For all $x,y\in [-\alpha,\alpha]^l$ that differ in one coordinate only we have 
$|F_l(x)-F_l(y)|\leq 2\alpha$. Therefore we can use Theorem~\ref{theoco} with 
$\cS=\{1,\dots,l\}$, $X_i=[-\alpha,\alpha]$ and $\mu_i=\tfrac{1}{2} (\delta_{-\alpha} 
+ \delta_{\beta})$ for all $i\in \cS$, and $D=4\alpha^2 l$, to obtain that there 
exist $C_1,C_2>0$ such that, for every $l\in\N$, $\Gamma\subset \cW_n$ and 
$T\colon\,\{(x,y)\in\Z^2\times \Z^2\colon |x-y|=1\}\to \{0,1\}$,
\begin{equation}
\label{concmesut}
\P\big(|F_l(\xi_1,\dots,\xi_m)-\E(F_l(\xi_1,\dots,\xi_m))|>\eta\big)
\leq C_1e^{-\tfrac{C_2\eta^2}{l}}.
\end{equation}  

%%%%%%%%%%%%%%%%%%%%%%%%%%%

\section{Large deviation estimate}
\label{Computation}

Let $(\xi_i)_{i\in\N}$ be an i.i.d.\ sequence of Bernouilli trials taking values 
$\beta$ and $-\alpha$ with probability $\frac12$ each. For $N\leq n\in \N$, denote 
by $\cE_{n,N}$ the set of all ordered sequences of $N$ disjoint and non-empty 
intervals included in $\{1,\dots,n\}$, i.e.,
\begin{align}
\label{add26}
\nonumber\cE_{n,N} &= \big\{(I_j)_{1\leq j\leq N}\subset\{1,\dots,n\}\colon 
I_j=\{\min I_j,\dots,\max I_j\} \,\,\forall\,1\leq j\leq N,\\
& \max I_j<\min I_{j+1} 
\,\,\forall\,1\leq j\leq N-1\ \text{and}\ I_j \neq 
\emptyset\,\,\forall\,1\leq j\leq N\big\}.
\end{align}
For $(I)\in \cE_{n,N}$, let $T(I)=\sum_{j=1}^N |I_j|$ be the cumulative length of 
the intervals making up $(I)$. Pick $\gamma>0$ and $M\in \N$, and denote by 
$\widehat{\cE}_{n,M}^{\,\gamma}$ the set of those $(I)$ in $\cup_{1\leq N\leq (n/M)}\,
\cE_{n,N}$ that have a cumulative length larger than $\gamma n$, i.e.,
\be{add27}
\widehat{\cE}_{n,M}^{\,\gamma} = \cup_{N=1}^{n/M}\big\{(I)\in \cE_{n,N}
\colon\,T(I)\geq \gamma n\big\}.
\ee
Next, for $\eta>0$ set 
\be{add28}
\cQ_{n,M}^{\gamma,\eta} = \bigcap_{\,(I)\in \widehat{\cE}_{n,M}^{\,\gamma}} 
\left\{\sum_{j=1}^N \sum_{i\in I_j} \xi_i
\leq (\tfrac{\beta-\alpha}{2}+\eta)\,T(I)\right\}.
\ee

\begin{lemma}
\label{lele}
For all $(\alpha,\beta)\in \CONE$, $\gamma>0$ and $\eta>0$ there exists an $\widehat{M}
\in \N$ such that, for all $M\geq \widehat{M}$,
\be{convsubset}
\lim_{n\to \infty} P((\cQ_{n,M}^{\gamma,\eta})^c)=0.
\ee
\end{lemma} 

\begin{proof}
An application of Cram\'er's theorem for i.i.d.\ random variables gives that
there exists a $c_\eta>0$ such that, for every $(I)\in \widehat{\cE}_{n,M}^{\,\gamma}$,
\be{add29}
\P_\xi\bigg(\sum_{j=1}^N \sum_{i\in I_j} 
\xi_i\geq (\tfrac{\beta-\alpha}{2}+\eta)\, T(I)\bigg)\leq\, 
e^{-c_\eta T(I)}\leq\, e^{-c_\eta\gamma n},
\ee
where we use that $T(I)\geq \gamma n$ for every $(I)\in \widehat{\cE}_{n,M}^{\,\gamma}$. 
Therefore
\begin{align}
\P_\xi((\cQ_{n,M}^{\gamma,\eta})^c)&\leq  |\widehat{\cE}_{n,M}^{\,\gamma}| 
e^{-c(\eta)\gamma n },
\end{align}
and it remains to bound  $|\widehat{\cE}_{n,M}^{\,\gamma}|$ as
\be{domiE}
\widehat{\cE}_{n,M}^{\,\gamma} = \sum_{N=1}^{n/M}\big|\big\{(I)\in \cE_{n,N}
\colon T(I)\geq \gamma n\big\}\big|
\leq \sum_{N=1}^{n/M} \binom{n}{2N}, 
\ee
where we use that choosing $(I)\in \cE_{n,N}$ amounts to choosing in $\{1,\dots,n\}$ 
the end points of the $N$ disjoint intervals. Thus, the right-hand side of \eqref{domiE} 
is at most $(n/M) \binom{n}{2n/M}$, which for $M$ large enough is 
$o( e^{c(\eta)\gamma n})$ as $n\to \infty$. 
\end{proof}

\end{appendix}

%%%%%%%%%%%%%%%%%%%%%%%%%%%%% REFERENCES %%%%%%%%%%%%%%%%%%%%%%%%%%%%%%%%%%%

\end{document}